\RequirePackage{amsmath}
\documentclass[a4paper, 10pt, reqno]{amsart}
\usepackage{amssymb, url, color, pb-diagram, graphicx, amscd, pb-diagram,mathrsfs}
\usepackage[colorlinks=true, bookmarks=true, pdfstartview=FitH, pagebackref=true]{hyperref}
\usepackage[nodayofweek]{datetime}
\usepackage{comment}
\usepackage{enumerate,stmaryrd} 
\usepackage{euscript,enumitem}
\usepackage{graphicx}
\usepackage{mathabx}
\usepackage{leftindex}
\usepackage[space, compress, sort]{cite}
\usepackage{empheq}
\usepackage{multirow}
\usepackage{microtype}

\numberwithin{equation}{section}
\theoremstyle{plain}

\newtheorem{thm}{Theorem}[section]
\newtheorem{theorem}[thm]{Theorem}

\newtheorem{lemma}[thm]{Lemma}

\newtheorem{proposition}[thm]{Proposition}

\theoremstyle{remark}

\newtheorem{remark}[thm]{Remark}
\newtheorem{remarks}[thm]{Remarks}

\newtheorem{claim}{Claim}
\newtheorem{claim2}{Claim}

\theoremstyle{definition}

\newtheorem{definition}[thm]{Definition}

\newcounter{mnotecount}[section]

\newcommand{\definedas}{\coloneqq}

\def\epsilon{{\varepsilon}}
\def\phi{{\varphi}}

\let\<\langle 
\let\>\rangle 

\newcommand{\hot}{\mathrm{h.o.t.}}

\DeclareMathOperator{\sh}{sh}
\DeclareMathOperator{\ch}{ch}

\DeclareMathOperator\supp{supp}
\DeclareMathOperator{\lot}{l.o.t.}

\newcommand{\ebar}{\overline{e}}

\newcommand{\bR}{\mathbb{R}}

\newcommand{\bS}{\mathbb{S}}

\newcommand{\bN}{\mathbb{N}}
\newcommand{\bH}{\mathbb{H}}

\newcommand{\bU}{\mathbb{U}}

\newcommand{\bL}{\mathbb{L}}

\newcommand{\cA}{\mathcal{A}}
\newcommand{\cB}{\mathcal{B}}

\newcommand{\cD}{\mathcal{D}}
\newcommand{\cE}{\mathcal{E}}

\newcommand{\cH}{\mathcal{H}}
\newcommand{\cI}{\mathcal{I}}

\newcommand{\cL}{\mathcal{L}}

\newcommand{\cN}{\mathcal{N}}

\newcommand{\cR}{\mathcal{R}}
\newcommand{\cQ}{\mathcal{Q}}
\newcommand{\cS}{\mathcal{S}}
\newcommand{\cT}{\mathcal{T}}
\newcommand{\cU}{\mathcal{U}}
\newcommand{\cV}{\mathcal{V}}

\newcommand{\cX}{\mathcal{X}}
\newcommand{\cY}{\mathcal{Y}}

\newcommand{\fX}{\mathfrak{X}}

\newcommand{\Bbar}{\overline{B}}

\newcommand{\Vbar}{\overline{V}}

\renewcommand{\hbar}{\overline{h}}

\newcommand{\Vtil}{\widetilde{V}}

\newcommand{\ptil}{\widetilde{p}}

\newcommand{\Gtil}{\widetilde{G}}

\newcommand{\vtil}{\widetilde{v}}
\newcommand{\wtil}{\widetilde{w}}
\newcommand{\Deltabar}{\overline{\Delta}}

\newcommand{\btil}{\widetilde{b}}

\newcommand{\chibar}{\overline{\chi}}

\DeclareMathOperator{\diag}{diag}

\DeclareMathOperator{\tr}{tr}

\DeclareMathOperator{\divg}{div}

\DeclareMathOperator{\vol}{Vol}

\DeclareMathOperator{\id}{Id}

\newcommand{\lie}{\mathcal{L}}
\newcommand{\rlie}{\mathring{\mathcal{L}}}

\newcommand{\pdiff} [2]{\frac{\partial #1}{\partial #2}}
\newcommand{\pdiffs}[2]{\frac{\partial^2 #1}{\partial {#2}^2}}
\newcommand{\pdiffm}[3]{\frac{\partial^2 #1}{\partial #2 \partial #3}}

\newcommand{\riem}{\mathcal{R}}

\newcommand{\riemudud}[4]{\riem^{#1 \phantom{#2} #3}_{\phantom{#1} #2 \phantom{#3} #4}}

\newcommand{\ric}{\mathrm{Ric}}

\newcommand{\ricud}[2]{\ric^{#1}_{\phantom{#1} #2}}

\newcommand{\scal}{\mathrm{Scal}}

\DeclareMathOperator{\hess}{Hess}

\newcommand{\hessbar}{\overline{\hess}}

\newif\ifshownotes
\shownotestrue 

\begin{document}

\title[Weak definition of the mass aspect]
{A definition of the mass aspect function for weakly regular asymptotically hyperbolic manifolds}

\begin{abstract}
In contrast to the well-known and unambiguous notion of ADM mass for asymptotically Euclidean manifolds, the notion of mass for asymptotically hyperbolic manifolds admits several interpretations. Historically, there are two approaches to defining the mass in the asymptotically hyperbolic setting: the mass aspect function of Wang defined on the conformal boundary at infinity, and the mass functional of Chru\'sciel and Herzlich which may be thought of as the closest asymptotically hyperbolic analogue of the ADM mass. In this paper we unify these two approaches by introducing an ADM-style definition of the mass aspect function that applies to a broad range of asymptotics and in very low regularity.  Additionally, we show that the mass aspect function can be computed using the Ricci tensor. Finally, we demonstrate that this function exhibits favorable covariance properties under changes of charts at infinity, which includes a proof of the asymptotic rigidity of hyperbolic space in the context of weakly regular metrics.
\end{abstract}

\author[R. Gicquaud]{Romain Gicquaud}
\address[R. Gicquaud]{
  Institut Denis Poisson\\
  Universit\'e de Tours\\
  Parc de Grandmont\\ 37200 Tours \\ France}
\email{romain.gicquaud@idpoisson.fr}

\author[A. Sakovich]{Anna Sakovich}
\address[A. Sakovich]{Department of Mathematics, Uppsala University, Sweden}
\email{anna.sakovich@math.uu.se}


%
%

\date{\today}
\maketitle
\tableofcontents

\section{Introduction}
The positive mass theorem is a longstanding question in both general relativity and Riemmanian geometry. The definition of the mass of an isolated system in general relativity dates back to the work of R. Arnowitt, S. Deser and C. W. Misner in 1959 \cite{ArnowittDeserMisner}. The interest of the mathematical community in this subject grew first of all because of the proof of the so-called positive mass conjecture in dimension 3 by Schoen and Yau \cite{SchoenYau79} as part of their seminal work on the influence of scalar curvature on minimal surfaces. A second major impetus came from the connection established by Schoen between the study of asymptotically Euclidean manifolds and the Yamabe problem \cite{SchoenYamabe}. The main theorem of \cite{SchoenYau79} states that, if $(M, g)$ is a complete Riemannian manifold of dimension $3$ with non-negative scalar curvature and whose geometry near infinity approaches that of Euclidean space in a certain sense, then a quantity  depending only on the asymptotic geometry of $(M, g)$, the ADM mass of $(M, g)$, is non-negative and vanishes if and only if $(M, g)$ is isometric to the Euclidean space. Shortly after, Schoen and Yau observed in  \cite{SchoenYau79b} that their proof can be extended to cover dimensions $3 \leq n< 8$, while Witten \cite{Witten81} and Bartnik \cite{BartnikMass} obtained a different proof for spin manifolds of arbitrary dimension. 

There is an important generalization of this result to a certain class of asymptotically hyperbolic manifolds $(M, g)$. In this context, the mass is a vector in Minkowski space defined, as before, solely in terms of the asymptotic geometry of the manifold. The positive mass theorem, proven under various assumptions in \cite{WangMass,ChruscielHerzlich,ChruscielDelay2019,HuangJangMartin} (see also \cite{AnderssonCaiGalloway, ChruscielGallowayNguyenPaetz}), states that, if $(M, g)$ has scalar curvature greater than or equal to $-n(n-1)$ (i.e. the scalar curvature of the hyperbolic space of the same dimension as $M$), then the mass vector is future pointing and time-like unless $(M, g)$ is isometric to the hyperbolic space.

We now describe the construction of the mass vector in greater detail, following the original definition of Wang \cite{WangMass} that involves the so called mass aspect function. Let $(\bH^n, b)$ denote the hyperbolic space of dimension $n \geq 3$:
\[
 \bH^n = B_1(0), \quad b = \rho^{-2} \delta,\quad  \rho = \frac{1-|x|^2}{2},
\]
where $B_1(0)$ denotes the open unit ball in $\bR^n$ and $\delta$ is the Euclidean metric. Assume that we are given a chart at infinity
\[
\Phi: M\setminus K \to \bH^n \setminus K'
\]
where $K$ (resp. $K'$) is a compact subset of $M$ (resp. $\bH^n$) and that the 2-tensor $e \definedas \Phi_* g - b$ extends smoothly to $\overline{B}_1(0) \setminus K'$ and satisfies
\begin{equation}\label{eqWangAsymptotics}
 e = \rho^{n-2} \ebar + \hot
\end{equation}
where $\ebar$ is a smooth 2-tensor on $\overline{B}_1(0)$, which in particular implies $|e|_b = O(\rho^n)$. Assume further that $e$ satisfies the transversality condition $\ebar_{ij}x^j \equiv 0$ (see \cite{CortierDahlGicquaud, WangMass} for the relevance of this condition). Then the restriction of $m \definedas \tr_\delta \ebar$ to $\bS_1(0)$\footnote{See our notation convention for the spheres on page \pageref{rkNotation}.} is called the \emph{mass aspect function} of $g$.

Let $x^1, \ldots, x^n$ denote the standard coordinate functions on $\bR^n$. The components of the mass vector $p$ are defined by
\[
 p^0 = \int_{\bS_1(0)} m(x) d\mu^\sigma(x),\quad  p^i = \int_{\bS_1(0)} x^i m(x) d\mu^\sigma(x).
\]
Assuming that $\scal^g \geq -n(n-1)$, the positive mass theorem states that
\[
 p^0 \geq \sqrt{\sum_{i=1}^n (p^i)^2},
\]
with equality if and only if $(M, g)$ is isometric to $(\bH^n, b)$.

An extensive discussion of the dependence of $p$ on the chart $\Phi$ can be found in \cite{CortierDahlGicquaud}, see also \cite{ChruscielNagy, ChruscielHerzlich, HerzlichMassFormulae} and \cite[Section 3]{ChruscielGallowayNguyenPaetz}.\\

Note that passing from $\ebar$ to $m$ and then to $p$ results in a drastic loss of information about the asymptotic geometry of $(M, g)$. However, the results in \cite{CortierDahlGicquaud} show that both $\ebar$ and $m$ enjoy covariance properties under a change of chart at infinity $\Phi$, indicating that $\ebar$ and $m$ probably have an intrinsic definition. At the same time, the above definition of mass aspect function requires very stringent asymptotic conditions to be satisfied, see \eqref{eqWangAsymptotics}, while the ADM-style definition of mass for asymptotically hyperbolic manifolds studied by Chrusciel and Herzlich in \cite{ChruscielHerzlich}, see also \cite{ChruscielNagy, HerzlichMassFormulae} merely assumes that $|e|_b+|De|_b = o(\rho^{n/2})$.

The aim of this paper is to unify the two approaches to defining the mass of asymptotically hyperbolic manifolds, namely the mass aspect function approach of Wang and the ADM-like definition of Chru\'sciel and Herzlich. More specifically, we obtain an ADM-style definition of the mass aspect function that applies to very general asymptotics by allowing for more general ``test'' functions in the ADM formalism of \cite{ChruscielHerzlich, ChruscielNagy, HerzlichMassFormulae}. We also extend the range of regularity of metrics allowed to define the mass to encompass metrics with local regularity $L^\infty \cap W^{1, 2}$ and with distributional scalar curvature as was done for asymptotically Euclidean manifolds in \cite{LeeLeFloch}. Further, we provide an approach where the integrand at infinity is related to the Ricci tensor of the metric (see e.g. \cite{HerzlichRicciMass}). This approach has the merit of being more geometric, but nevertheless requires greater regularity for the metric.\\

The outline of the paper is as follows. In Section \ref{secPrelim}, we recall basic facts about the hyperbolic space, function spaces on asymptotically hyperbolic manifolds and on the linearisation of the scalar curvature operator. Section \ref{secEstimates} is dedicated to Hessian estimates for eigenfunctions of the Laplacian. The main results of the paper are contained in Section \ref{secMassAspect} where we show how to define the mass aspect function in the ADM formalism (Section \ref{secADM}) and by using the Ricci tensor (Section \ref{secRicci}). Finally, Section \ref{secCoordinates} studies the behavior of the mass aspect function under  changes of a chart at infinity. This section essentially consists of two parts: in the first one (Section \ref{secRigidity}), we describe how two asymptotically hyperbolic charts at infinity are related to each other. Then, in the second one (Section \ref{secCovariance}), we derive the actual transformation law of the mass aspect function.

\section*{Acknowledgments} RG is partially supported by the grant ANR-23-CE40-0010-02 of the French National Research Agency (ANR): Einstein constraints: past, present, and future (EINSTEIN-PPF). AS was funded by the Swedish Research Council grants dnr. 2016-04511 and dnr. 2024-04845. We are grateful to Rafe Mazzeo for helpful discussions.

\section{Preliminaries}\label{secPrelim}

The main object of study in this paper is the class of asymptotically hyperbolic manifolds that we now introduce. Note that, while there is an essentially unique definition of an asymptotically Euclidean manifold, various nonequivalent definitions of asymptotically hyperbolic manifolds coexist in the literature depending on whether the conformal infinity is assumed to be the round sphere or more general conformal infinities are allowed. For an extensive discussion, see \cite{LeeFredholm}, \cite{GicquaudCompactification}, and \cite{CortierDahlGicquaud}. Here, following our earlier work, e.g. \cite{DahlGicquaudSakovichSmallMass}, we restrict our focus to a specific class of asymptotically hyperbolic manifolds that provides a natural framework for defining mass.

Throughout this paper, we will use different models for hyperbolic space $(\mathbb{H}^n, b)$, in particular the ball model as described in the introduction. Other models for the hyperbolic space will be introduced when needed. The Levi-Civita connection associated with the hyperbolic metric $b$ will be denoted by $D$, while the Levi-Civita connection for the metric $g$ under consideration will be denoted by $\nabla$. 

\begin{definition}\label{defAH}
Let $(M, g)$ be a complete Riemannian manifold. Then $(M, g)$ is said to
be \emph{asymptotically hyperbolic of order} $\tau > 0$ if there exist compact
subsets $K \subset M$ and $K' \subset \bH^n$, a diffeomorphism $\Phi: M\setminus K \to \bH^n \setminus K'$ and a constant $C > 1$ such that
$\frac{1}{C} b \leq \Phi_* g \leq C b$ and
\[
    e \definedas \Phi_* g - b
\]
satisfies 
\begin{equation}\label{eqEstimateE}
    \int_{\bH^n \setminus K'} \rho^{-2\tau} \left( |De|^2_b + |e|^2_b\right) d\mu^b < \infty.
\end{equation}
$\Phi$ is then called a \emph{chart at infinity} for $(M, g)$.
\end{definition}

Since we are only interested in the asymptotic regularity of the metric $g$, or more exactly of the difference $e$ between $g$ and $b$ as defined above, the precise shape of $K$ and $K'$ will be irrelevant. In particular, we will always assume that they have smooth boundaries to avoid the pitfalls associated with the boundary regularity of functions in Sobolev spaces.

\subsection{Function spaces}\label{secFunctSpaces}
Here we introduce the function spaces to be used in the sequel. Properties of these function spaces will be recalled when needed, the proofs can be found in \cite{LeeFredholm,GicquaudSakovich}, see also \cite{And93,AnderssonChrusciel,AllenIsenbergLeeStavrov,AllenLeeMaxwell}.

As we will work only with metrics that are close to the hyperbolic metric we only need to define the function spaces for the hyperbolic metric. General definitions are given in the aforementioned references. Let $E \to \bH^n$ be a geometric bundle (i.e. a vector bundle associated to the principal orthonormal frame bundle $SO(T\bH^n)$). We will work with two types of function spaces.

\begin{itemize}[leftmargin=*]
\item\textsc{Weighted Sobolev spaces:} Let $k \geq 0$ be an integer,
let $1 \leq p < \infty$ be a real number, and let $\delta \in \bR$.
The weighted Sobolev space $W^{k, p}_\delta(\bH^n, E)$ is the set of
sections $u$ of $E$ such that $u \in W^{k, p}_{loc}(\bH^n, E)$ and such
that the norm
\[
    \left\| u \right\|_{W^{k, p}_\delta(\bH^n, E)} = \sum_{i=0}^k \left( \int_{\bH^n} \left| \rho^{-\delta} D^{(i)} u \right|^p_g d\mu^b\right)^{\frac{1}{p}}
\]
is finite. For $k = 0$, we use the shorthand $L^p_\delta(\bH^n, E)$ to denote $W^{0, p}_\delta(\bH^n, E)$.

\item\textsc{Weighted H\"older spaces:} Let $k \geq 0$ be an
integer and $\alpha \in [0, 1]$, let $\delta \in \bR$
be given. The weighted H\"older space $C^{k, \alpha}_\delta(\bH^n, E)$
is the set of sections $u$ such that $u \in C^{k, \alpha}_{loc}(\bH^n, E)$
and such that the norm
\[
    \left\| u \right\|_{C^{k, \alpha}_\delta(\bH^n, E)} = \sup_{x \in \bH^n} \rho^{-\delta}(x) \left\|u\right\|_{C^{k, \alpha}(B_r(x))}
\]
is finite, where $B_r(x)$ denotes the ball of radius $r$ centered at $x$.
We also denote
\[
C^\infty_{\delta}(\bH^n, E) \definedas \bigcap_{k = 0}^\infty C^{k, \alpha}_\delta(\bH^n, \bR).
\]
\end{itemize}

\subsection{First order variation of the scalar curvature}
\label{secScalarCurvature}
As is well known, the scalar curvature operator has tight connection to the mass. In this paper, we will need its first order variation, but we will also need to keep track of the error terms, especially those containing second order derivatives as they will require some attention later.

We compute the first order Taylor expansion of $\scal^g$ at $b$ where, as this calculation may serve other purposes, we will not assume for the moment that $b$ is the hyperbolic metric. However, we will assume that there exists a constant $C > 1$ such that
\[
    \frac{1}{C} b \leq g \leq C b,
\]
where the inequalities are understood in the sense of quadratic forms.

In some given coordinate chart, we set
\[
    g_{ij} = b_{ij} + e_{ij}, \qquad g^{ij} = b^{ij} + f^{ij},
\]
where, as usual, $(g^{ij})_{i, j}$ (resp. $(b^{ij})_{i, j}$) denotes the matrix inverse of $(g_{ij})_{i, j}$ (resp. of $(b_{ij})_{i, j}$). We also make the assumption that, except for in the case of the metric $g$, all indices are raised and lowered using the metric $b$.

There is a nice relation between $e$ and $f$ that will be used in the
sequel. By definition, we have
\[
    \delta_i^j = g_{ik} g^{kj} = (b_{ik} + e_{ik})(b^{kj} + f^{kj}) = b_{ik} b^{kj} + b_{ik} f^{kj} + e_{ik} b^{kj} + e_{ik} f^{kj}.
\]
Since $b_{ik} b^{kj} = \delta_i^j$, we conclude that
\[
    0 = b_{ik} f^{kj} + e_{ik} b^{kj} + e_{ik} f^{kj}.
\]
Raising the $i$-index using the metric $b$, we get
\begin{equation}\label{eqEF}
    f^{ij} = - e^{ij} - e^i_{\phantom{i}k} f^{kj}.
\end{equation}
Note that if we adopt a matrix notation and write $I+F = (I+E)^{-1}$,
where $I$ is the identity matrix, this formula reads
$F = - E - EF$, which, by induction, yields the well-known Taylor
expansion for the matrix inverse
$\displaystyle (I + E)^{-1} = \sum_{i=0}^\infty (-1)^i E^i$.\\

We note that the inequalities $C^{-1} b \leq g \leq C b$ imply that the
eigenvalues $\lambda_i$ of the symmetric matrix $(g_{ij})_{i, j}$ in a
$b$-orthonormal basis all lie between $C^{-1}$ and $C$. The eigenvalues
of $e$ are then $\lambda_i - 1$ and those of $f$ are $\lambda_i^{-1} - 1$.
From $\lambda_i \geq C^{-1}$, we conclude that
\[
    \left|\lambda_i^{-1}-1\right| = \left|\frac{\lambda_i-1}{\lambda_i}\right| \leq C |\lambda_i-1|.
\]
As $|e|$ (resp. $|f|$) is the square root of the sum of the squares of
the eigenvalues of $e$ (resp. $f$), we obtain the following inequality:
\begin{equation}\label{eqEFnorm}
    |f| \leq C |e|.
\end{equation}
In particular, from \eqref{eqEF}, we have
\[
|f - (-e)| = |e \star f| \leq n C |e|^2,
\]
where $e \star f$ is the tensor $(e \star f)_i^{\phantom{i}j} =  e_{ik}f^{kj}$ and the coefficient $n$ in the inequality comes from the fact that $|\cdot|$ is the Frobenius norm on 2-tensors.

Let $D$ denote the Levi-Civita connection of $b$ and $\nabla$ that of $g$.
We introduce the ``Christoffel symbols'' $\Gamma$ as the difference between
both connections, namely, for any vector field $X$:
\[
    \nabla_i X^k - D_i X^k = \Gamma^k_{ij} X^j.
\]
These Christoffel symbols are given by the usual formula except that the partial derivatives are replaced by the covariant derivative with respect to $b$:
\begin{equation}\label{eqChristoffel}
\begin{aligned}
    \Gamma^k_{ij}
    & \definedas \frac{1}{2} g^{kl}\left(D_i g_{lj} + D_j g_{il} - D_l g_{ij}\right)\\
    & = \frac{1}{2} g^{kl}\left(D_i e_{lj} + D_j e_{il} - D_l e_{ij}\right).
\end{aligned}
\end{equation}

Note that, since $0 = D_l \delta_i^j = D_l (g_{ik} g^{kj})$, we have
\[
    0 = (D_l g_{ik}) g^{kj} + g_{ik} D_l g^{kj} = (D_l e_{ik}) g^{kj} + g_{ik} D_l f^{kj}.
\]
Raising the $i$-index with respect to the metric $g$, we get
\begin{equation}\label{eqEF2}
    D_l f^{kj} = - g^{kp} g^{jq} D_l e_{pq}.
\end{equation}

From the formula \eqref{eqChristoffel}, we easily deduce that the scalar curvature of $g$ can be written as follows:
\begin{equation}\label{eqScalar}
    \scal^g = g^{ij} \ric_{ij} + g^{jl} \left(D_i \Gamma^i_{jl} - D_l \Gamma^i_{ji} + \Gamma^i_{ip} \Gamma^p_{jl} - \Gamma^i_{lp} \Gamma^p_{ji}\right).
\end{equation}
In this formula we have made the following convention which will be used in the rest of the
paper:\\

\noindent\fbox{%
\parbox{\textwidth}{%
 All terms without further specification of the metric (e.g. $\divg$, $\Delta$, $\tr$, $\ric$...) are computed with respect to the metric $b$.
}%
}\\

Our goal now is to obtain the first order Taylor formula for the scalar curvature of $g$. For this, we  compute the Taylor expansion of each term in \eqref{eqScalar}. In what follows $\cQ(e, De)$ denotes
terms that are quadratic with respect to $e$ and $De$, that is such that we have
\[
 |\cQ(e, De)| \lesssim |e|^2 + |De|^2,
\]
where the notation $A \lesssim B$ means that there exists a constant $\Lambda > 0$
independent of $A$ and $B$ but depending on the value of $C$ such that
$A \leq \Lambda B$. We will also make use of \eqref{eqEF} and \eqref{eqEF2} without mentioning. We have:
\begin{align*}
g^{ij} \ric_{ij}
    &= (b^{ij} + f^{ij}) \ric_{ij}\\
    &= \scal + f^{ij} \ric_{ij}\\
    &= \scal - e^{ij} \ric_{ij} + \cQ(e, De),\\
g^{jl} \left(D_i \Gamma^i_{jl} - D_l \Gamma^i_{ji}\right)
    &= D_i \left(g^{jl} \Gamma^i_{jl}\right) - D_l \left(g^{jl} \Gamma^i_{ji}\right) - \Gamma^i_{jl} D_i g^{jl} + \Gamma^i_{ji} D_l g^{jl}\\
    &= D_i \left(g^{jl} \Gamma^i_{jl} - g^{ji} \Gamma^l_{jl}\right) + \cQ(e, De)\\
    &= \frac{1}{2} D_i \left[g^{jl} g^{ik} \left(2 D_j e_{kl} - D_k e_{jl}\right) - g^{ji} g^{lk} D_j e_{lk}\right] + \cQ(e, De),\\
g^{jl} \Gamma^i_{ip} \Gamma^p_{jl} - g^{jl} \Gamma^i_{lp} \Gamma^p_{ji}
    &= \cQ(e, De).
\end{align*}
Summing up, we have proven the following formula:
\begin{equation}\label{eqScalarVar}
 \scal^g
 = \scal - e^{ij} \ric_{ij} + D_i \left[g^{jl} g^{ik} \left(D_j e_{kl} - D_k e_{jl}\right)\right] + \cQ(e, De).
\end{equation}

In particular, specifying that $b$ is the hyperbolic metric, we
obtain the following proposition:

\begin{proposition}\label{propScalar1}
The first order variation of the scalar curvature operator near the
hyperbolic metric $b$ is given by
\begin{equation}\label{eqScalarVar1}
 \scal^g = -n(n-1) + (n-1) \tr(e) + D_i \left[g^{jl} g^{ik} \left(D_j e_{kl} - D_k e_{jl}\right)\right] + \cQ(e, De).
\end{equation}
\end{proposition}

This proposition follows immediately from \eqref{eqScalarVar}. It should
be noted that the third term in the right hand side of \eqref{eqScalarVar1} can be expanded as
\begin{equation}\label{eqScalarVar2}
 D_i \left[g^{jl} g^{ik} \left(D_j e_{kl} - D_k e_{jl}\right)\right] = \divg(\divg(e)) - \Delta \tr(e) + \hot
\end{equation}
so that we recover
the usual formula for the first order variation of the scalar curvature. Yet, the higher order terms that appear here contain second order
derivatives of $e$ that will require a special treatment in what follows
so, following \cite{ChruscielHerzlich}, we chose to keep the first order
variation in \eqref{eqScalarVar1} as such.

\subsection{The hyperboloid model} \label{secGeodesics}
In this section we collect some general facts about hyperbolic space that will be used in Section \ref{secRigidity} to prove the covariance of the mass functional. For these purposes it will be convenient to work with the hyperboloid model of $\mathbb{H}^n$. More specifically, let $\bR^{n,1}=(\mathbb{R}^{n+1},\eta)$ be Minowski space with its standard inner product $\eta=\diag\{-1,1,\ldots, 1\}$. We will denote its points by $X=(X^0,X^1,\ldots,X^n) $ or, as a shorthand, $(X^0, \vec{X})$ with $\vec{X}=(X^1, \ldots, X^n)$. Then $\mathbb{H}^n$ can be identified with the upper unit hyperboloid:
\begin{align*}
    \bH^n
    &=\{X: \eta(X,X)=-1,~ X^0 > 0\}\\
    &=\left\{(X^0,\vec{X}):X^0=\sqrt{1+|\vec{X}|^2}\right\}.
\end{align*}

The restriction of linear forms on $\bR^{n, 1}$ to $\bH^n$ plays an important role in what follows. We state here the main result:

\begin{proposition}\label{propLapse}
The set $\cN^b$ of functions $V$ that are solutions to the equation
\begin{equation}\label{eqLapse}
    \hess V = V b
\end{equation}
is $(n+1)$-dimensional. Its elements are the restrictions of linear forms on $\bR^{n, 1}$ to the hyperboloid. In particular, a basis for $\cN^b$ in the ball model of the hyperbolic space is given by
\begin{equation}\label{eqLapses}
    V^0 = \frac{1+|x|^2}{1-|x|^2} = \frac{1}{\rho}-1,\quad
    V^i = \frac{2 x^i}{1 - |x|^2} = \frac{x^i}{\rho}.
\end{equation}
\end{proposition}
The functions $V \in \cN^b$ will be called lapse functions in the sequel. This terminology comes from the fact that if $V \in \cN^b$ does not change sign, the manifold $\bR \times \bH^n$ endowed with the warped product metric $h = - V^2 dt^2 + b$ is the anti-de Sitter space (see e.g. \cite[Appendix A]{DahlGicquaudSakovichSmallMass}). As these functions will play an important role in the sequel, and for lack of reference, we provide a quick proof of this proposition.

\begin{proof}
Let $N = X^\mu \partial_\mu$ denote the normal to $\bH^n$ with respect to the Minkowski metric. $N$ is the vector field generating dilations in $\bR^{n, 1}$ so, by Euler's formula for homogeneous functions, we have, for any function $\theta$ on a cone of $\bR^{n, 1}$,
\[
d\theta(N) = \theta \Leftrightarrow \theta \text{ is homogeneous of degree } 1.
\]
In particular, if $V$ is a linear form on $\bR^{n, 1}$, we have $dV(N) = V$. 

The Hessian of a function $V$ on $\bH^n$ is defined as follows:
\[
\begin{aligned}
    \hess^b V(X, Y)
    &= X(Y V) - dV(D_X Y)\\
    &= \hess^\eta V(X, Y) + dV(\partial_X Y - D_X Y),
\end{aligned}
\]
where $X$ and $Y$ are vector fields defined in a neighborhood of $\bH^n$ and tangent to $\bH^n$, and $\partial_X Y$ is the usual derivative on $\bR^{n, 1}$. As the Levi-Civita connection $D_X Y$ is the projection of $\partial_X Y$ onto $T\bH^n$, we have
\[
    D_X Y = \partial_X Y + \eta(\partial_X Y, N) N = \partial_X Y - \eta(\partial_X N, Y) N,
\]
where we used the fact that $\eta(Y, N) = 0$ over $\bH^n$ to conclude that 
$\eta(\partial_X Y, N) = - \eta(Y, \partial_X N)$. All in all, we get
\begin{equation}\label{eqDiffHessian}
    \hess^b V(X, Y) = \hess^\eta V(X, Y) + \eta(X, Y) dV(N),
\end{equation}
where we used the definition of $N$ to conclude that $\partial_X N = X$.

Now if $V$ is a linear form on $\bR^{n, 1}$, it is homogeneous of degree $1$ and its second order derivative vanishes: $dV(N) = V$ and $\hess^\eta V = 0$. As a consequence, we have
\[
    \hess^b V(X, Y) = V \eta(X, Y) = V b(X, Y)
\]
since $X$ and $Y$ are tangent to $\bH^n$. 

Conversely, if $V$ satisfies $\hess^b V = V b$, we can extend it to a homogeneous function of degree $1$ defined in the inside of the future pointing light cone:
if $X \in \bR^{n, 1}$ is a vector in the future pointing light cone, then $X / \sqrt{-\eta(X, X)}$ is a point on the hyperboloid so we can set
\[
V(X) = \sqrt{-\eta(X, X)} V\left(\frac{X}{\sqrt{-\eta(X, X)}}\right).
\]
In this case, equation \eqref{eqLapse} coupled with \eqref{eqDiffHessian} shows that $\hess^\eta V(X, Y) = 0$ for vector fields tangent to $\bH^n$ and, by homogeneity, to all vector fields orthogonal to $N$. To end the proof of the proposition, we only need to check that, for any vector field $Z$, we have $\hess^\eta V (Z, N) = 0$. Indeed, since $dV(N)=V$ we have
\begin{align*}
\hess^\eta V (Z, N)
&= Z (dV(N)) - dV(\partial_Z N)\\
&= dV(Z) - dV(Z)\\
&= 0. 
\end{align*}
\end{proof}

Our next goal is to derive the equation for the geodesic joining two given points $X,Y\in \mathbb{H}^n$. 

\begin{proposition}\label{propExponentialXY-intrinsic}
Let $X,Y\in \bH^n$. Then we may write $Y = \exp_X \xi(X,Y)$ for $\xi(X,Y)\in T_X \mathbb{H}^n$ given by 
\begin{equation}\label{eqLogarithm}
\xi(X,Y)=\frac{\ln\left(-\eta (X,Y)+\sqrt{-1 + \eta(X,Y)^2}\right)}{\sqrt{-1 + \eta(X,Y)^2}} \left( \eta(X,Y)X +Y \right).
\end{equation}
\end{proposition}

\begin{proof} 
As is well known, the geodesic joining the points $X,Y\in \mathbb{H}^n$ arises as the intersection of the hyperboloid and the plane $\Pi(X,Y)$ through $X$, $Y$ and the origin. To parametrize the geodesic, we introduce a vector $\widetilde{Z} = Y+\eta(Y,X)X$. This vector satisfies $\eta(X, \widetilde{Z})=0$ and belongs to $\Pi(X,Y)$ which implies that it is a spacelike vector tangent to our geodesic at the point $X$. Having defined 
\[
Z=\frac{\widetilde{Z}}{\sqrt{\eta(\widetilde{Z}, \widetilde{Z})}}=\frac{\eta(X,Y) X + Y}{\sqrt{-1 + \eta(X,Y)^2}}
\]
we may now parametrize our geodesic as 
\[
\begin{split}
\gamma(t) &= \cosh (st) X +  \sinh (st) \, Z \\& =\left(\cosh (st) + \frac{\eta(X,Y)}{\sqrt{-1 + \eta(X,Y)^2}}\sinh (st) \right) X +\frac{1}{\sqrt{-1 + \eta(X,Y)^2}} \sinh (st)  Y,
\end{split}
\]
so that $\|\dot{\gamma}\|_\eta=s$. If we now set
\[
s = d_{\bH^n}(X, Y) = \cosh^{-1} (-\eta (X,Y))=\ln\left(-\eta (X,Y)+\sqrt{-1 + \eta(X,Y)^2}\right),
\]
we get $\gamma(0)=X$ and $\gamma(1)=Y$. Moreover, we have 
\[
\dot{\gamma}(0)= \ln\left(-\eta (X,Y)+\sqrt{-1 + \eta(X,Y)^2}\right) Z .
\]
Consequently, we obtain the formula:
\[
Y=\exp_X \left(\frac{\ln\left(-\eta (X,Y)+\sqrt{-1 + \eta(X,Y)^2}\right)}{\sqrt{-1 + \eta(X,Y)^2}} \left( \eta(X,Y)X +Y \right) \right).
\]
\end{proof}

For later applications, it will be convenient to rewrite \eqref{eqLogarithm} in terms of the difference $U \definedas Y - X$.
As both $X, Y \in \bH^n$, we have $\eta(X, X) = \eta(Y, Y) = -1$. Thus,
\[
-1 = \eta(Y, Y) = \eta(X, X) + 2 \eta(X, U) + \eta(U, U),
\]
showing that
\[
\eta(X, U) = -\frac{1}{2} \eta(U, U),
\]
and, hence,
\[
\eta(X, Y) = \eta(X, X) + \eta(X, U) = -1 - \frac{1}{2} \eta(U, U).    
\]
Setting $u \definedas \frac{1}{2}\eta(U, U)$, we have $\eta(X, Y) = -1 - u$ so Equation \eqref{eqLogarithm} becomes
\[
\xi(X, Y)
= f(u) \left(U - u X\right),
\]
where
\[
f(u) \definedas \frac{\ln\left(1 + u + \sqrt{2 u + u^2}\right)}{\sqrt{2u + u^2}}.
\]

\begin{remark}
The function $f(u)$ defined above is actually analytic despite its appearance because the exponential map is analytic. This can be seen directly by noting that
\[
1 = \left(1 + u + \sqrt{2u+u^2}\right)\left(1 + u - \sqrt{2u+u^2}\right),
\]
so
\[
f(u) = \frac{1}{2\sqrt{2u+u^2}}
\left(\ln\left(1 + u + \sqrt{2 u + u^2}\right) - \ln\left(1 + u - \sqrt{2 u + u^2}\right)\right).
\]
Setting $u = w^2$, we get a new function $h(w) = f(w^2)$ defined by
\[
h(w) = \frac{1}{2w\sqrt{2+w^2}}
\left(\ln\left(1 + w^2 + w \sqrt{2 + w^2}\right) - \ln\left(1 + w^2 - w \sqrt{2 + w^2}\right)\right).
\]
Changing $w$ to $-w$ in the expression for $h$ permutes the two log terms in the parentheses and changes the denominator to its opposite showing that $h(w)$ is even in $w$. As $h$ obviously admits a Taylor series expansion in $w$ near $0$, this expansion involves only even powers of $w$. This concludes the proof of the fact that $f$ is analytic near $0$. Further, for small $u$, we see that $f(u) = 1 + O(u)$.  
\end{remark}

Conversely, we will need the formula for the exponential map of $\bH^n$. This can be found  in \cite[Proposition A.5.1]{BenedettiPetronio}: For any $X \in \bH^n$ and $\xi \in T_X \bH^n$, we have
\begin{equation}\label{eqExp}
 \exp_X(\xi) = c(|\xi|^2) X + s(|\xi|^2) \xi,
\end{equation}
where $c$ and $s$ are defined by 
\[
 c(t) = \cosh(\sqrt{t}) = \sum_{k=0}^\infty \frac{t^k}{(2k)!}, \quad s(t) = \frac{1}{\sqrt{t}} \sinh(\sqrt{t}) = \sum_{k=0}^\infty \frac{t^k}{(2k+1)!}.
\]

Given any vector field $\xi$ over $\bH^n$, we define a map
$\Phi_\xi: \bH^n \to \bH^n$ by $\Phi_\xi(x) = \exp_x(\xi(x))$.
It follows from Equation \eqref{eqExp} that, for any vector field $\xi \in \Gamma(T\bH^n)$,
\[
 d\Phi_\xi(X) = c(|\xi|^2) X + 2 c'(|\xi|^2) \eta(\xi, d\xi(X)) x + s(|\xi|^2) d\xi(X) + 2 s'(|\xi|^2) \eta(\xi, d\xi(X)) \xi.
\]
A lengthy but straightforward calculation then yields 
\begin{equation}\label{eqPullback}
\begin{aligned}
 \btil(\xi)(X, Y)
 &\definedas \Phi_\xi^*b(X, Y)\\
 &= b(d\Phi_\xi(X), d\Phi_\xi(Y))\\
 &= c^2 b(X, Y) + s^2 \left(b(D_X\xi, D_Y\xi) - b(X, \xi) b(Y, \xi)\right)\\
 &\qquad + c s \left(b(X, D_Y\xi) + b(Y, D_X \xi)\right)\\
 &\qquad + \frac{1}{4|\xi|^2} (1-s^2) (\partial_X |\xi|^2)(\partial_Y |\xi|^2)\\
 &\qquad + \frac{1-cs}{2|\xi|^2}\left[b(X, \xi) \partial_Y |\xi|^2 + b(Y, \xi) \partial_X |\xi|^2\right],
\end{aligned}
\end{equation}
where we set $c = c(|\xi|^2)$ and $s = s(|\xi|^2)$. For the reader interested in the details of this calculation, we note that, since $\xi(x) \in T_x \bH^n$ for all $x \in \bH^n$, we have $\eta(\xi, x) \equiv 0$. Differentiating this relation with respect to $X$ we obtain 
\[
 0 = \partial_X \eta(\xi, x) = \eta(\partial_X \xi, x) + \eta(\xi, X). 
\]
Further, we recall that the Levi-Civita connection of the hyperbolic metric $b$ is the orthogonal projection onto $T\bH^n$ of the usual differentiation on $\bR^{n, 1}$, that is
\[
 D_X \xi(x) = d\xi(X) - b(X, \xi) x.
\]

Equation \eqref{eqPullback} can equivalently be written as follows:
\begin{equation}\label{eqPullback2}
\begin{aligned}
\btil
& =
c^2 b + s^2 \left(b(D_{\cdot} \xi, D_{\cdot} \xi) + \xi^\flat \otimes \xi^\flat\right)
+ c s~\lie_{\xi} b + \frac{1-s^2}{4|\xi|^2} d |\xi|^2 \otimes d |\xi|^2\\
& \qquad + \frac{1-cs}{2|\xi|^2} \left[\xi^\flat \otimes d|\xi|^2 + d|\xi|^2 \otimes \xi^\flat\right].
\end{aligned}
\end{equation}

Last but not least, we will need a way to transfer vector fields $\xi$ defined on $\bH^n$ seen as the hyperboloid in $\bR^{n, 1}$ to vector fields that are intrinsic to $\bH^n$ keeping track of the regularity of $\xi$. This is the content of the following proposition: 

\begin{proposition}\label{propTransfert}
Assume that $\xi: \bH^n \to \bR^{n, 1}$ is a vector field, $\xi = (\xi^0, \xi^1, \ldots, \xi^n)$, tangent to the hyperboloid such that $\xi^\mu \in W^{l, q}_\epsilon(\bH^n, \bR)$ for all $\mu \in \{0, 1, \ldots, n\}$. Then $\xi \in W^{l, q}_{\epsilon}(\bH^n, T\bH^n)$.
\end{proposition}

We prove this proposition by constructing a well-behaved family of vector fields on $\bH^n$. To achieve this, we observe that $\bH^n$ naturally inherits a Lie group structure, where the hyperbolic metric is left-invariant. The desired vector fields are then the left-invariant ones\footnote{While this approach might seem indirect, it has the advantage of showing that the weights $\epsilon$ for $\xi$, when viewed as a Minkowski-valued function, and for $\xi$ as a vector field on hyperbolic space, must coincide. This clarifies some details that may have been overlooked in \cite{HerzlichMassFormulae}.}. This Lie group structure seems well-known to experts (see e.g. \cite{AllenIsenbergLeeStavrov, AllenLeeMaxwell}) and can be easily described on the upper half space model but we give some details as our intention is to work on the hyperboloid. We recall that the hyperbolic space can be identified with the homogeneous space $\raisebox{0.2ex}{\( SO_0(n, 1) \)}/\raisebox{-0.2ex}{\( SO(n) \)},$
where \(SO_0(n, 1)\) denotes the connected component of the identity in \(O(n, 1)\), defined as
\[
O(n, 1) \definedas \{M \in GL(n+1, \mathbb{R}) \mid M^T H M = H\},
\]
with
\[
H = \begin{pmatrix}
-1 & 0 & \cdots & 0\\
0 & 1 &  & \\
\vdots & & \ddots & \\
0 & & & 1
\end{pmatrix},
\]
where the indices of the rows and columns range from \(0\) to \(n\). The Lie algebra $\mathfrak{so}(n, 1)$ has a Cartan involution given by
\[
    \theta: M \mapsto \theta(M) \definedas - M^T.
\]
This involution is such that $\mathfrak{k} \definedas \{M \in \mathfrak{so}(n, 1) :\theta(M) = M\} \simeq \mathfrak{so}(n)$. Let $\mathfrak{p} \definedas \{M \in \mathfrak{so}(n, 1) : \theta(M) = - M\}$ denote the other eigenspace of $\theta$ so that $\mathfrak{so}(n, 1) = \mathfrak{k} \oplus \mathfrak{p}$ is the Cartan decomposition of $\mathfrak{so}(n, 1)$.
The Iwasawa decomposition of $SO_0(n, 1)$ reads
\[
    SO_0(n, 1) = K A N,
\]
where $K = SO(n)$ is a maximal compact subgroup of $SO_0(n, 1)$, $A = \exp(\mathfrak{a})$ is an abelian Lie subgroup of $SO_0(n, 1)$ whose Lie algebra $\mathfrak{a}$ is a maximal (abelian) subalgebra contained in $\mathfrak{p}$ and $N = \exp(\mathfrak{n}_0$) is a nilpotent Lie subgroup of $SO_0(n, 1)$ whose (nilpotent) Lie algebra is obtained from the restricted root decomposition of $\mathfrak{so}(n, 1)$. We refer the reader to \cite{KnappLieGroups} for a general introduction to the Iwasawa decomposition and to \cite{HudsonLiMarkMerinoReyes} for the particular case of the Lorentz group.

We choose for $\mathfrak{a}$ the 1-dimensional Lie subalgebra generated by
\[
a =
\left(
\begin{array}{cc|ccc}
0 & 1 & 0 & \cdots & 0\\
1 & 0 & 0 & \cdots & 0\\
\hline
0 & 0 & \multicolumn{3}{c}{\multirow{4}{*}{\huge 0}} \\
\vdots & \vdots & & &\\
0 & 0 & & &
\end{array}
\right).
\]
The Lie algebra $\mathfrak{so}(n, 1)$ then decomposes into eigenspaces for the adjoint action of $a$. The corresponding eigenvalues are easily found to be $-1, 0, +1$. It follows from Jacobi's identity that $\mathfrak{n}_0 \definedas E_{+1}(a)$ is an abelian (in particular nilpotent) Lie subalgebra as the bracket of two elements of $\mathfrak{n}_0$, if non-zero, should be an eigenvector for the eigenvalue $2$ of $\mathrm{ad}(a)$. Further, $\mathfrak{a} \oplus \mathfrak{n}_0$ is a solvable Lie subalgebra of $\mathfrak{so}(n, 1)$. The algebra $\mathfrak{n}_0$ is the span of the following matrices:
\[
N_i =
\left(
\begin{array}{cc|ccccc}
\multicolumn{2}{c|}{\multirow{2}{*}{\huge 0}} & 0 & \cdots & 1 & \cdots & 0 \\
\multicolumn{2}{c|}{} & 0 & \cdots & 1 & \cdots & 0 \\
\hline
0 & 0 & & & & & \\
\vdots & \vdots & & \multicolumn{4}{c}{\multirow{4}{*}{\huge 0}} \\
1 & -1 & & & & & \\
\vdots & \vdots & & & & & \\
0 & 0 & & & & & \\
\end{array}
\right)
\]
where the non-zero column and non-zero line are the $i$-th ones ($i= 2 \ldots, n$).

The exponential for a matrix $M$ in $\mathfrak{a} \oplus \mathfrak{n}_0$ is easily computed. Indeed, if
\[
M = x_1 a + \sum_{i \geq 2} x^i N_i,
\]
we have $M^3 = (x_1)^2 M$ so
\[
\exp(M) = I_{n+1} + \frac{\sinh(x_1)}{x_1} M + \frac{\cosh(x_1)-1}{(x_1)^2} M^2.
\]
Yet it appears easier to keep two exponentials and write $\exp(\mathfrak{a} \oplus \mathfrak{n}_0)$ as $AN = \exp(\mathfrak{a}) \exp(\mathfrak{n}_0)$.

As $K = SO(n)$ is the stabilizer of $(1, 0, \ldots, 0)$, the Iwasawa decomposition $SO_0(n,1) \simeq K A N \simeq A N K$ shows that the mapping $h \mapsto h \cdot (1, 0, \ldots, 0)$ is a diffeomorphism from $AN$ onto $\bH^n$. It is given by
\[
X(x_1, \vec{x}) = \exp(x_1 a) \exp(\vec{x} \cdot \vec{N})
\begin{pmatrix} 1\\ 0\\ \vec{0}\end{pmatrix}
=
\renewcommand{\arraystretch}{1.3} 
\begin{pmatrix}
\cosh(x_1) + \frac{e^{x_1}}{2} |\vec{x}|^2\\
\sinh(x_1) + \frac{e^{x_1}}{2} |\vec{x}|^2\\
\vec{x}
\end{pmatrix}.
\]
So we see that the coordinates $(x_1, \vec{x})$ representing a vector $X \in \bH^n$ are given by
\[
x_1 = - \log(X^0 - X^1),\quad \vec{x} = (X^2, \ldots, X^n).
\]

The associated left invariant vector fields are easily computed by remarking that, if $V \in \mathfrak{a} \oplus \mathfrak{n}_0$ is an element of the Lie algebra, we can associate to it the corresponding vector at $(1, 0, \ldots, 0)$, by $V (1, 0, \ldots, 0)^T$ and act by an element of $AN$ on the left to move to an arbitrary point of $\bH^n$.
We find the following vectors:
\[
\left\lbrace
\begin{aligned}
\cI^1 & \definedas \frac{1}{X^0 - X^1} \left[\partial_0 + \partial_1\right] - X^\mu \partial_\mu, \\
\cI^A & \definedas \frac{X^A}{X^0 - X^1} (\partial_0 + \partial_1) + \partial_A.  
\end{aligned}
\right.
\]
for $A=2,\ldots, n$ and $\mu=0,\ldots, n$. 

It is easily checked that $(\cI^1, \cI^2, \ldots, \cI^n)$ is an orthonormal basis at all points of $\bH^n$. By construction, the hyperbolic metric is left-invariant as $AN$ is a subgroup of $SO_0(n, 1)$. Indeed, for any $h \in AN$ and any vector field $v$, we have $dL_h(v) = h \cdot v$, where we view $v$ as an element of $\bR^{n, 1}$ and $h$ as the corresponding linear map on $\bR^{n, 1}$. Hence, for any pair $w, v$ of vector fields we have
\[
b(dL_h v, dL_h w) = \eta(h \cdot v, h \cdot w) = \eta(v, w) = b(v, w).
\]
The covariant derivative of the left-invariant vector fields are given by Koszul formula:
\begin{align*}
2 b\left(D_{\cI^j} \cI^i, \cI^k  \right)
&= \cI^i \left(b(\cI^j, \cI^k)\right) + \cI^j \left(b(\cI^i, \cI^k)\right)
- \cI^k \left(b(\cI^i, \cI^j)\right)\\
&\qquad - b\left([\cI^i, \cI^j], \cI^k\right) - b\left([\cI^j, \cI^k], \cI^i\right) + b\left([\cI^k, \cI^i], \cI^j\right).
\end{align*}
The first three terms in the left hand side are zero because $b(\cI^i, \cI^j) = \delta^{ij}$ is constant. Hence, we get
\begin{align*}
D_{\cI^j} \cI^i &= - \frac{1}{2} \sum_{k=1}^n \left(b\left([\cI^i, \cI^j], \cI^k\right) + b\left([\cI^j, \cI^k], \cI^i\right) - b\left([\cI^k, \cI^i], \cI^j\right)\right) \cI^k\\
&= - \frac{1}{2} [\cI^i, \cI^j] - \frac{1}{2} \sum_{k=1}^n \left(b\left([\cI^j, \cI^k], \cI^i\right) - b\left([\cI^k, \cI^i], \cI^j\right)\right) \cI^k.
\end{align*} 
Using the fact that the commutators $[\cI^i, \cI^j]$ are given by the Lie algebra constants,
\[
[\cI^i, \cI^j] =
\left\lbrace
\begin{aligned}
0 & \quad \text{if } i, j \geq 2,\\
\cI^j & \quad \text{if } i=1 \text{ and } j \geq 2,\\
- \cI^i & \quad \text{if } j=1 \text{ and } i \geq 2,
\end{aligned}
\right.
\]
we may now obtain the following result:
\[
D_{\cI^j} \cI^i =
\left\lbrace
\begin{aligned}
\delta_{ij} \cI^1 &\quad\text{if } i, j \geq 2,\\
- \cI^j &\quad \text{if } i=1,~j \geq 2,\\
0 &\quad \text{if } j = 1.
\end{aligned}
\right.
\]
In particular, the covariant derivative operator induces a (bounded) linear map
\[
D: \mathfrak{h} \otimes \mathfrak{h} \to \mathfrak{h},
\]
where $\mathfrak{h} = \mathfrak{a} + \mathfrak{n}_0$ is identified with the set of left-invariant vector fields on $\bH^n$.

We will be interested in these vector fields because of the following proposition:

\begin{proposition}\label{propInvariantVectors}
Assume that $\xi \in \Gamma(T\bH^n)$ is a vector field on the hyperbolic space (not necessarily smooth). We have the following equivalence for $q \in [1, \infty]$, $l \in \bN$, $\delta \in \bR$:
\[
\xi \in W^{l, q}_\delta(\bH^n, T\bH^n)
\Leftrightarrow
\forall i \in \ldbrack 1, n\rdbrack,~b(\xi, \cI^i) \in W^{l, q}_\delta(\bH^n, \bR).
\]
\end{proposition}

Before proving this proposition, we state and prove a lemma:
\begin{lemma}\label{lmHigherDerivative}
For all $k \geq 0$, there exists a constant $C_k$ such that, for any index $a$,
\[
\left| D^{(k)} \cI^a \right| \leq C_k,
\]
i.e. $\cI^a \in W^{k, \infty}_0(\bH^n, T\bH^n)$ for all $k \geq 0$.
\end{lemma}

\begin{proof}
The proof goes by induction on $k$. For $k = 0$, this follows from the statement that $\cI^a$ is a unit vector at each point of $\bH^n$.

Assume that the lemma is proven for a given $k \geq 0$. Since $D^{(k)} \cI^a$ is a left-invariant $(1, k)$-tensor so, for any given $k$-tuple $(i_1, \ldots, i_k)$, $(D^{(k)} \cI^a)(\cI^{i_1}, \ldots, \cI^{i_k})$ is a left-invariant vector field. As a consequence, for any index $i_0$, $D_{\cI^{i_0}} \left((D^{(k)} \cI^a)(\cI^{i_1}, \ldots, \cI^{i_k}\right)$ is a left-invariant vector field whose norm is bounded. Now
\begin{align*}
D^{(k+1)}\cI^a (\cI^{i_0}, \cI^{i_1}, \ldots, \cI^{i_k})
&=  D_{\cI^{i_0}} \left((D^{(k)} \cI^a)(\cI^{i_1}, \ldots, \cI^{i_k}\right)\\
&\qquad - (D^{(k)} \cI^a) (D_{\cI^{i_0}} \cI^{i_1}, \ldots, \cI^{i_k}) - \cdots\\
&\qquad - (D^{(k)} \cI^a) (\cI^{i_1}, \ldots, D_{\cI^{i_0}} \cI^{i_k})
\end{align*}
is a left-invariant vector field whose norm is bounded uniformly with respect to the $(k+1)$-tuple $(i_0, \ldots, i_k)$. Hence,
\[
\left|D^{(k+1)} \cI^a\right|^2
= \sum_{(i_0, \ldots, i_k)} \left|D^{(k+1)}\cI^a (\cI^{i_0}, \cI^{i_1}, \ldots, \cI^{i_k})\right|^2
\]
is bounded (actually constant) on $\bH^n$.
\end{proof}

\begin{proof}[Proof of Proposition \ref{propInvariantVectors}]
From Lemma \ref{lmHigherDerivative}, we know that the vector fields $\cI^a$ belong to $W^{l, \infty}_0(\bH^n, T\bH^n)$. Hence, if $\xi \in W^{l, q}_\delta(\bH^n, T\bH^n)$, we have
$b(\xi, \cI^a) \in W^{l, q}_\delta(\bH^n, T\bH^n)$ (this is a slight generalization of \cite[Lemma 3.6 (a)]{LeeFredholm}).  The converse implication is equally easy.
\end{proof}

We write down explicit formulas for the scalar product of a tangent vector $\xi$ at a point $X \in \bH^n$ where we continue to view the hyperbolic space as an upper unit hyperboloid in Minkowski spacetime. Set $\xi=\xi^0\partial_0 + \xi^i \partial_i \in \Gamma(T\bH^n)$. At every point $(X^0,X^1,\ldots, X^n)$ on the hyperboloid we have 
\[
-\xi^0 X^0 + \xi^i X^i=0.
\]
As a consequence, we have
\begin{align} 
b(\xi, \cI^1)&=\eta(\xi, \cI^1)\nonumber\\
&=\frac{1}{X^0-X^1}\left(-\xi^0+\xi^1\right) - X^\mu \xi^\mu \nonumber\\
& = \frac{\xi^1-\xi^0}{X^0-X^1}. \label{eqScalarProductI1}
\end{align}
Similarly,
\begin{equation}\label{eqScalarProductIA}
b(\xi, \cI^A)=\eta(\xi, \cI^A)=\frac{X^A}{X^0-X^1}(\xi^1-\xi^0)+\xi^A.
\end{equation}

\begin{proof}[Proof of Proposition \ref{propTransfert}]
We note that the denominator $X^0 - X^1$ is not bounded from below by any positive constant on the intersection of the hyperboloid with the $X^0X^1$-plane. Indeed, points of the hyperboloid in  this plane can be written as $(\ch(t), \sh(t), 0, \ldots, 0)$ so $X^0 - X^1 = e^{-t}$ tends to zero as $t$ tends to infinity. To remedy this, we remark that we can assume that $\xi$ has support contained in the half space $X^1 \leq 1$. We do so by considering $\chi(X^1) \xi$ instead of $\xi$, where $\chi$ is a smooth cutoff function such that $\chi(x) = 1$ for $x \leq 0$ and $\chi(x) = 0$ for $x \geq 1$.

As the derivative of $\chi$ has compact support, we have that $\chi(X^1) \xi^\mu \in W^{l, q}_{\delta}(\bH^n, \bR)$ and similarly $(1-\chi) \xi^\mu \in W^{l, q}_{\delta}(\bH^n, \bR)$.

If we prove the proposition for $\chi(X^1) \xi$, namely that
\[
\|\chi(X^1) \xi\|_{W^{l, q}_\delta} \lesssim \sum_{\mu= 0}^n \|\chi(X^1) \xi^\mu\|_{W^{l, q}_\delta},
\]
we can then perform a reflection with respect to the $(X^0 X^2 \cdots X^n)$-plane (changing $X^1$ to $-X^1$) and also  conclude that
\[
\|(1-\chi(X^1)) \xi\|_{W^{l, q}_\delta} \lesssim \sum_{\mu= 0}^n \|(1-\chi(X^1)) \xi^\mu\|_{W^{l, q}_\delta}.
\]
Hence,
\begin{align*}
\|\xi\|_{W^{l, q}_\delta}
&\leq \|\chi(X^1) \xi\|_{W^{l, q}_\delta} + \|(1-\chi(X^1)) \xi\|_{W^{l, q}_\delta}\\
& \lesssim \sum_{\mu= 0}^n \left(\|\chi(X^1) \xi^\mu\|_{W^{l, q}_\delta} + \|(1-\chi(X^1)) \xi^\mu\|_{W^{l, q}_\delta}\right)\\
&\lesssim \sum_{\mu= 0}^n \|\xi^\mu\|_{W^{l, q}_\delta},
\end{align*}
concluding the proof of the proposition.

So let us assume that $\xi$ has support in the subset $X^1 \leq 1$. As a consequence, by a straightforward calculation, we have that $\displaystyle \frac{1}{X^0 - X^1}$ and all its derivatives (w.r.t. the hyperbolic metric $b$) are uniformly bounded on the support of $\xi$. From Formula \eqref{eqScalarProductI1}, we conclude that $b(\xi, \cI^1) \in W^{l, q}_\delta(\bH^n, \bR)$.

In the same manner, the fraction $\displaystyle \frac{X^A}{X^0 - X^1}$ and all its derivatives are uniformly bounded on the support of $\xi$ so Formula \eqref{eqScalarProductIA} shows that $b(\xi, \cI^A) \in W^{l, q}_\delta(\bH^n, \bR)$ for all $A = 2, \ldots, n$.

Applying Proposition \ref{propInvariantVectors}, we conclude the proof of Proposition \ref{propTransfert}.
\end{proof}

\begin{remark}
Note that Formula \eqref{eqScalarProductIA} shows that Proposition \ref{propTransfert} is optimal at least for vector fields directed in the $\partial_A$ direction.
\end{remark}

\section{Estimates for eigenfunctions of the Laplacian}\label{secEstimates}
The aim of this section is to construct solutions to the equation $\Delta V = n V$ so that $\rho V$ extends by continuity to the boundary (in the ball model) $\bS_1(0)$ of $\bH^n$ and to get explicit asymptotics for the 2-tensor $\hess V - V b$ as needed in Section \ref{secMassAspect}. This section is rather technical and dedicated to the proof of Proposition \ref{propEstimateEigenfunction2}. The reader willing to accept this result without a proof can safely skip this section.

It should be noted that there is a natural approach to construct solutions of the above equation, see e.g. \cite[Chapter 8]{LeeFredholm} which addresses a more complicated problem. However, there is a fairly delicate issue which arises if we follow this approach in our setting. Namely, given a function $v \in C^2(\bS_1(0), \bR)$ one would first  construct an approximate solution $\Vtil(x)$, for example
\begin{equation}\label{eqApproximateEigenfunction}
\Vtil(x) = V^0(x) v(x/|x|) \chi(|x|),
\end{equation}
where $V^0$ is defined in \eqref{eqLapses}, $\chi$ is a cut-off function that (in the ball model) vanishes near $x = 0$ and equals $1$ in a neighborhood of $|x| = 1$. The function $v$ satisfies $|\hess v - v b| = O(\rho)$ which is exactly what is expected. Yet, if we want to promote $v$ to a genuine solution $V$ of the equation $\Delta V = n V$, we have to find a function $v'$ such that $-\Delta v' + n v' = \Delta v - n v$ so that $V = v + v'$ is an eigenfunction of $\Delta$. Alas, the right hand side belongs to $C^0_1(\bH^n, \bR)$ and elliptic regularity is known to fail in $C^0$ so this approach would require constructing a better approximate solution $\vtil$ by a smoothing argument. Since the smoothed function $\vtil$ should still satisfy $|\hess \vtil - \vtil b| = O(\rho)$, the smoothing procedure has to take into account the geometry of the problem.
As a consequence, we have chosen a different, computational approach that only applies to the hyperbolic space. The advantage is that we work with explicit formulas.

In what follows, calculations are easier to perform in the upper half space model as in this case the solution corresponds to the convolution with an explicitly given function. In the following calculations, $x$ will denote coordinates in the ball model and $y$ coordinates in the upper half space. As the first coordinate will play a special role, we will denote
\[
 x = (x^1, x^2, \ldots, x^n) = (x^1, \vec{x}),
\]
so that $\vec{x} = (x^2, \ldots, x^n)$, and similarly, we denote $y = (y^1, y^2, \ldots, y^n) = (y^1, \vec{y})$. The simplest way to pass from the conformal ball model to the upper half space is by the inversion $i_{(-1, \vec{0}), \sqrt{2}}$ with respect to the circle centered at $(-1, \vec{0})$ and of radius $\sqrt{2}$ (see e.g. \cite{BenedettiPetronio}):
\[
 y = i_{(-1, \vec{0}), \sqrt{2}}(x) \definedas 2 \frac{x - (-1, \vec{0})}{|x - 
(-1, \vec{0})|^2} + (-1, \vec{0}),
\]
i.e.
\[
 \left\lbrace
 \begin{aligned}
  y^1 &=  2 \frac{x^1 + 1}{(x^1 + 1)^2 + |\vec{x}|^2} - 1,\\
  \vec{y} &= 2 \frac{\vec{x}}{(x^1 + 1)^2 + |\vec{x}|^2}.
 \end{aligned}
 \right.
\]
As $i_{(-1, \vec{0}), \sqrt{2}}$ is an involution, we also get
\begin{equation}\label{eqUpperToBall}
 \left\lbrace
 \begin{aligned}
  x^1 &=  2 \frac{y^1 + 1}{(y^1 + 1)^2 + |\vec{y}|^2} - 1,\\
  \vec{x} &= 2 \frac{\vec{y}}{(y^1 + 1)^2 + |\vec{y}|^2}.
 \end{aligned}
 \right.
\end{equation}
A straightforward calculation yields the following formula for $\rho = 
\frac{1-|x|^2}{2}$:
\begin{equation}\label{eqRho2}
 \rho = \frac{2 y^1}{(y^1 + 1)^2 + |\vec{y}|^2}.
\end{equation}
This gives the expressions for the lapse functions in this new coordinate system:
\[
\left\lbrace
\begin{aligned}
 V_0 &= \frac{1}{\rho} - 1 = \frac{1 + |y|^2}{2 y^1},\\
 V_1 &= \frac{x^1}{\rho} = \frac{1 - |y|^2}{2 y^1},\\
 V_i &= \frac{x^i}{\rho} = \frac{y^i}{y^1}, \quad  i=2,\ldots, n.
\end{aligned}
\right.
\]
We also note that the hyperbolic metric $b$ reads
\[
 b = \frac{1}{(y^1)^2} ((dy^1)^2 + (dy^2)^2 + \cdots + (dy^n)^2).
\]

Our first task is to find the Poisson kernel for the asymptotic Dirichlet 
problem. Note that for $f = f(y^1)$ we have
\[
 \Delta f = (y^1)^2 f'' - (n-2) y^1 f'.
\]
This formula is a direct consequence of the conformal transformation 
law of the Laplacian (see e.g. \cite[Equation 1.1.15]{GicquaudThesis}).
In particular, the function $f$ solves $\Delta f = n f$ if and only if $f$ is a linear combination of $(y^1)^{-1}$ and $(y^1)^n$. We discard the first solution as it is nothing but $V_0 + V_1$.

After an inversion with respect to the circle of radius $1$ centered at $(0, \vec{z})$, which is an isometry of the upper half space model, we get a new eigenfunction $f_{\vec{z}}$ of the Laplacian:
\begin{equation}\label{eqPoissonKernel}
 f_{\vec{z}}(y^1, \vec{y}) = (y^1 \circ i_{(0, \vec{z}), 1})^n
 \definedas \frac{(y^1)^n}{\left((y^1)^2 + |\vec{y} - \vec{z}|^2\right)^n}
 = \frac{1}{(y^1)^n\left(1+\left|\frac{\vec{y}-\vec{z}}{y^1}\right|^2\right)^n}.
\end{equation}
Hence, given an arbitrary function $v \in C^0_{-4}(\bR^{n-1}, \bR)$ where
\[
C^0_{-4}(\bR^{n-1}, \bR) = \{ u \in C^0_{loc}(\bR^{n-1}, \bR),~ |u(\vec{y})|\lesssim (1 + |\vec{y}|^2)^2\},
\]
we get a solution
\begin{equation}\label{eqEigenfunction}
 V \definedas \int_{\bR^{n-1}} v(\vec{z}) f_{\vec{z}}(y^1, \vec{y}) d\vec{z}
\end{equation}
to $\Delta V = n V$. Note that, performing the change of variable $\vec{w} = 
(y^1)^{-1} (\vec{z} - \vec{y})$, we get

\begin{align}
 V
  &=
  (y^1)^{-n} \int_{\bR^{n-1}} v(\vec{z}) 
  \frac{1}{\left(1+\left|\frac{\vec{y}-\vec{z}}{y^1}\right|^2\right)^n} 
d\vec{z}\nonumber\\
  &= \frac{1}{y^1}
  \int_{\bR^{n-1}} v(\vec{y} + y^1 \vec{w}) 
  \frac{1}{\left(1+\left|\vec{w}\right|^2\right)^n} d\vec{w}.\label{eqApproxId}
\end{align}
This shows that, up to a factor $(y^1)^{-1}$ and some multiplicative constant that will be made
explicit, $V$ is the convolution of $v$ with an approximation of the identity.
In particular, using \eqref{eqRho2}, we have
\[
 \rho V = \frac{2}{(y^1 + 1)^2 + |\vec{y}|^2}
 \int_{\bR^{n-1}} v(\vec{y} + y^1 \vec{w}) 
  \frac{1}{\left(1+\left|\vec{w}\right|^2\right)^n} d\vec{w}.
\]

Letting $y^1$ tend to zero, we conclude that
\begin{equation}\label{eqAsymptotic}
 \rho V(y^1, \vec{y}) \to_{y^1 \to 0} \frac{2 v(\vec{y})}{1 + |\vec{y}|^2}
  \int_{\bR^{n-1}} \frac{d\vec{w}}{\left(1+\left|\vec{w}\right|^2\right)^n}.
\end{equation}
The explicit value of the integral that appears in the limit is provided by the next lemma:

\begin{lemma}\label{lmIntegral}
 The integral
 \[
  I_{n, \beta} \definedas \int_{\bR^{n-1}} \frac{1}{(1+|\vec{w}|^2)^\beta} 
d\vec{w}
 \]
 is convergent provided that $\beta > n/2$ and we have
 \[
  I_{n, \beta} = \pi^{(n-1)/2} \frac{\Gamma\left(\beta - 
\frac{n-1}{2}\right)}{\Gamma(\beta)},
 \]
where $\displaystyle \Gamma(x) = \int_0^\infty t^{x-1} e^{-t} dt$
is the Euler gamma function.
\end{lemma}
We remind the reader that $\Gamma$ satisfies $\Gamma(x+1) = x \Gamma(x)$ for 
any $x > 0$.

\begin{proof}
 As the integrand is spherically symmetric, we have
 \[
  I_{n, \beta} = \omega_{n-2} \int_0^\infty \frac{r^{n-2}}{(1+r^2)^\beta} dr,
 \]
 where $\omega_{n-2} = \frac{2 
\pi^{(n-1)/2}}{\Gamma\left(\frac{n-1}{2}\right)}$ denotes the area of the 
unit sphere in $\bR^{n-1}$. We perform the change of variable $u = 
\frac{1}{1+r^2}$ in this integral so that
\[
 r = \sqrt{\frac{1}{u} - 1} \quad\text{and}\quad dr = - \frac{1}{2 u^2 
\sqrt{\frac{1}{u} - 1}} du.
\]
Thus, we have
 \begin{align*}
  I_{n, \beta}
  &= \frac{\omega_{n-2}}{2} \int_0^1 u^{\beta-2} 
\left(\frac{1}{u}-1\right)^{\frac{n-3}{2}} du\\
  &= \frac{\pi^{(n-1)/2}}{\Gamma\left(\frac{n-1}{2}\right)} \int_0^1 u^{\beta - 
\frac{n}{2} - \frac{1}{2}} \left(1-u\right)^{\frac{n-3}{2}} du\\
  &= \frac{\pi^{(n-1)/2}}{\Gamma\left(\frac{n-1}{2}\right)} B\left(\beta - 
\frac{n}{2} + \frac{1}{2}, \frac{n-1}{2}\right),
 \end{align*}
where $\displaystyle B(x, y) \definedas \int_0^1 t^{x-1} (1-t)^{y-1} dt$ is
the Euler beta function. In particular, we have
$\displaystyle B(x, y) = \frac{\Gamma(x)\Gamma(y)}{\Gamma(x+y)}$ which gives 
 \begin{align*}
  I_{n, \beta}
  &= \frac{\pi^{(n-1)/2}}{\Gamma\left(\frac{n-1}{2}\right)} 
\frac{\Gamma\left(\beta - \frac{n}{2} + 
\frac{1}{2}\right)\Gamma\left(\frac{n-1}{2}\right)}{\Gamma\left(\beta - 
\frac{n}{2} + \frac{1}{2} + \frac{n-1}{2}\right)}\\
  &= \pi^{(n-1)/2} \frac{\Gamma\left(\beta - \frac{n}{2} + 
\frac{1}{2}\right)}{\Gamma(\beta)}.
 \end{align*}
\end{proof}

As a consequence of Equation \eqref{eqAsymptotic}, given $v_0 \in C^0(\bS_1(0), \bR)$ we may now define the solution $V$ of the equation $\Delta V=n V$ such that $(\rho V)\vert_{\bS_1(0)} \equiv v_0$ as follows. We first define $v$ by
\begin{equation}\label{eqRelationvv0}
v(\vec{y}) = \frac{1 + |\vec{y}|^2}{2 I_{n, n}} v_0 \left(\frac{1-|\vec{y}|^2}{1+|\vec{y}|^2}, \frac{2 \vec{y}}{1+|\vec{y}|^2}\right).
\end{equation}
In other words, up to a multiplicative factor, $v$ is $v_0$ in the stereographic projection of $\bS^{n-1}$. Consequently, defining $V$ as above (see \eqref{eqEigenfunction} and \eqref{eqApproxId}) we observe that $\rho V$ extends by continuity within the ball model. This extension results in a continuous function, satisfying $(\rho V)\vert_{\bS_1(0)} \equiv v_0$. Note that the arguments of $v_0$ in this formula are nothing but the limits of $(x_1,\vec{x})$ when $y^1 \to 0$, see \eqref{eqUpperToBall}. As a consequence, we have proven the following result:
\begin{proposition}\label{propEstimateEigenfunction0}
For any function $v_0 \in C^0(\bS_1(0), \bR)$, there exists a unique function $V \in C^2_{-1}(\bH^n, \bR)$ such that
\[
\Delta V = n V
\]
and such that the function $\rho V$ extends by continuity to a function over $\overline{B}(0, 1)$ satisfying $\rho V\vert_{\bS_1(0)} = v_0$. We have
\[
\|V\|_{C^0_{-1}(\bH^n, \bR)} \leq C \|v_0\|_{L^\infty(\bH^n, \bR)},
\]
where $C = \|V^0\|_{C^0_{-1}(\bH^n, \bR)}$ and for any $k \geq 0$, $\alpha \in (0, 1)$, there exists a constant $C_{k, \alpha}$ independent of $v_0$ such that
\[
\|V\|_{C^{k, \alpha}_{-1}(\bH^n, \bR)} \leq C_{k, \alpha} \|v_0\|_{L^\infty(\bH^n, \bR)}.
\]
\end{proposition}

\begin{proof}
We start by proving the uniqueness of the function $V$. Since the mapping from $v_0$ to $V$ is linear, it suffices to show that the zero function is the unique $V \in C^2_{-1}(\bH^n, \bR)$ such that $\rho V$ extends continuously to $\overline{B}(0, 1)$ and vanishes on $\bS_1(0)$. Let $V^0$ be the lapse function defined in \eqref{eqLapses}. As $\Delta V = n V$, we have
\begin{align*}
n V
&= \Delta \left(V^0 \frac{V}{V^0}\right)\\
&= \frac{V}{V^0} \Delta V^0 + 2 b\left(d V^0, d\frac{V}{V^0}\right) + V^0 \Delta \frac{V}{V^0}\\
&= n V^0 \frac{V}{V^0} + 2 b\left(d V^0, d\frac{V}{V^0}\right) + V^0 \Delta \frac{V}{V^0}.
\end{align*}
Hence,
\[
\Delta \frac{V}{V^0} + 2 b\left(\frac{d V^0}{V^0}, d\frac{V}{V^0}\right) = 0.
\]
Since $V/V^0$ vanishes at infinity, it necessarily attains either an interior maximum or an interior minimum. By applying the strong maximum principle \cite[Theorem 3.5]{GilbargTrudinger}, we conclude that $V/V^0$ is constant. As it vanishes at infinity, it must therefore vanish identically.

As a consequence, the only solution $V \in C^2_{-1}(\bH^n, \bR)$ to $\Delta V = n V$ such that $\rho V \equiv v_0$ on $\bS_1(0)$ is the one provided by \eqref{eqEigenfunction}.

Next, note that the Poisson kernel \eqref{eqPoissonKernel} is a positive function. This has the following consequence: Given two eigenfunctions $V_1$ and $V_2$ such that $\rho V_1 \equiv v_1$ (resp. $\rho V_2 \equiv v_2$) on $\bS_1(0)$, if $v_2 \geq v_1$, we have $V_2 \geq V_1$. Indeed, from \eqref{eqEigenfunction}, we have
\[
V_2(y) - V_1(y) = \int_{\bR^{n-1}} (v_2(\vec{z}) - v_1(\vec{z}) f_{\vec{z}}(y) d\vec{z} \geq 0.
\]
In particular, for any function $v \in C^0(\bS_1(0), \bR)$, we have 
\[
-\|v\|_{L^\infty(\bS_1(0), \bR)} \leq v \leq \|v\|_{L^\infty(\bS_1(0), \bR)}. 
\]
As $\rho V^0 = 1 - \rho$ restricts to $1$ on $\bS_1(0)$, and $V^0$ is an eigenfunction of the Laplacian, we conclude from the previous discussion that
\[
-\|v\|_{L^\infty(\bS_1(0), \bR)} V^0 \leq V \leq \|v\|_{L^\infty(\bS_1(0), \bR)} V^0.
\]
This shows that
\[
\|V\|_{C^0_{-1}(\bH^n, \bR)} \leq \|V^0\|_{C^0_{-1}(\bH^n, \bR)} \|v\|_{L^\infty(\bS_1(0), \bR)}.
\]
The last part of the proposition follows from standard arguments in elliptic regularity, see e.g. \cite[Lemma 4.8]{LeeFredholm}.
\end{proof}

We now turn our attention to the estimate of $\hess V - V b$. The conformal
transformation law of the Hessian in the upper half space model yields 
\[
 \hess_{ij} V - V b_{ij}
 = \pdiffm{V}{y^i}{y^j}
   + \frac{1}{y^1}\left(\delta^1_i\pdiff{V}{y^j}+\delta^1_j\pdiff{V}{y^i}\right)
   - \left(\frac{1}{y^1}\pdiff{V}{y^1} + \frac{V}{(y^1)^2}\right) \delta_{ij}.
\]
The formula can be somewhat simplified by setting $V = (y^1)^{-1} \Vbar$. Namely, we get
\[
\hess_{ij} V - V b_{ij}
 = \frac{1}{y^1} \pdiffm{\Vbar}{y^i}{y^j}
   - \frac{1}{(y^1)^2} \pdiff{\Vbar}{y^1} \delta_{ij}.
\]
With this formula at hand, it is straightforward to compute 
\begin{align*}
 \Vbar
 &= \frac{1}{(y^1)^{n-1}}
  \int_{\bR^{n-1}}v(\vec{z})
  \frac{1}{\left(1+\left|\frac{\vec{y}-\vec{z}}{y^1}\right|^2\right)^n} d\vec{z},\\
 \pdiff{\Vbar}{y^1}
 &= \frac{1}{(y^1)^n}
  \int_{\bR^{n-1}}v(\vec{z})\left[
  \frac{ n+1}{\left(1+\left|\frac{\vec{y}-\vec{z}}{y^1}\right|^2\right)^n}
  - \frac{2n}{\left(1+\left|\frac{\vec{y}-\vec{z}}{y^1}\right|^2\right)^{n+1}}
  \right] d\vec{z},\\
 \pdiffs{\Vbar}{(y^1)}
 &= \frac{1}{(y^1)^{n+1}}
  \int_{\bR^{n-1}}v(\vec{z})\left[
  \frac{n(n+1)}{\left(1+\left|\frac{\vec{y}-\vec{z}}{y^1}\right|^2\right)^n}
  - \frac{2n(2n+3)}{\left(1+\left|\frac{\vec{y}-\vec{z}}{y^1}\right|^2\right)^{n+1}}
  + \frac{4n(n+1)}{\left(1+\left|\frac{\vec{y}-\vec{z}}{y^1}\right|^2\right)^{n+2}}
  \right] d\vec{z},\\
 \pdiffm{\Vbar}{y^1}{y^i}
 &= \frac{2n(n+1)}{(y^1)^{n+2}}
  \int_{\bR^{n-1}}v(\vec{z})\left[
  \frac{2}{\left(1+\left|\frac{\vec{y}-\vec{z}}{y^1}\right|^2\right)^{n+2}}
  - \frac{1}{\left(1+\left|\frac{\vec{y}-\vec{z}}{y^1}\right|^2\right)^{n+1}}
  \right](y^i - z^i) d\vec{z},\\
 \pdiffm{\Vbar}{y^i}{y^j}
 &= \frac{2n}{(y^1)^{n+1}}
  \int_{\bR^{n-1}}v(\vec{z})\left[
  \frac{2(n+1)(y^i-z^i)(y^j - z^j)}{(y^1)^2\left(1+\left|\frac{\vec{y}-\vec{z}}{y^1}\right|^2\right)^{n+2}}
  - \frac{\delta_{ij}}{\left(1+\left|\frac{\vec{y}-\vec{z}}{y^1}\right|^2\right)^{n+1}}
  \right] d\vec{z},
\end{align*}
for $i, j > 1$.
As a consequence, we also have
\begin{align}
&(\hess V - V b)_{11}\nonumber\\
  &\qquad = \frac{1}{(y^1)^{n+2}}
  \int_{\bR^{n-1}}v(\vec{z})\left[
  \frac{n^2-1}{\left(1+\left|\frac{\vec{y}-\vec{z}}{y^1}\right|^2\right)^n}
  - \frac{4n(n+1)}{\left(1+\left|\frac{\vec{y}-\vec{z}}{y^1}\right|^2\right)^{n+1}}
  + \frac{4n(n+1)}{\left(1+\left|\frac{\vec{y}-\vec{z}}{y^1}\right|^2\right)^{n+2}}
  \right] d\vec{z},\nonumber\\
  &\qquad = \frac{1}{(y^1)^3}
  \int_{\bR^{n-1}}v(\vec{y} + y^1 \vec{w})\left[
  \frac{n^2-1}{\left(1+\left|\vec{w}\right|^2\right)^n}
  - \frac{4n(n+1)}{\left(1+\left|\vec{w}\right|^2\right)^{n+1}}
  + \frac{4n(n+1)}{\left(1+\left|\vec{w}\right|^2\right)^{n+2}}
  \right] d\vec{w},\label{eqEigenFct11}
\end{align}
\begin{align}
&(\hess V - V b)_{1i}\nonumber\\
  &\qquad = \frac{2n(n+1)}{(y^1)^{n+3}}
  \int_{\bR^{n-1}}v(\vec{z})\left[
  \frac{2}{\left(1+\left|\frac{\vec{y}-\vec{z}}{y^1}\right|^2\right)^{n+2}}
  - \frac{1}{\left(1+\left|\frac{\vec{y}-\vec{z}}{y^1}\right|^2\right)^{n+1}}
  \right](y^i - z^i) d\vec{z},\nonumber\\
  &\qquad = \frac{2n(n+1)}{(y^1)^3}
  \int_{\bR^{n-1}}v(\vec{y} + y^1 \vec{w})\left[
  \frac{2}{\left(1+\left|\vec{w}\right|^2\right)^{n+2}}
  - \frac{1}{\left(1+\left|\vec{w}\right|^2\right)^{n+1}}
  \right]w^i d\vec{w},\label{eqEigenFct1i}
\end{align}
\begin{align}
&(\hess V - V b)_{ij}\nonumber\\
  &\qquad = \frac{n+1}{(y^1)^{n+2}}
  \int_{\bR^{n-1}}v(\vec{z})\left[
  \frac{4n(y^i - z^i)(y^j - z^j)}{(y^1)^2\left(1+\left|\frac{\vec{y}-\vec{z}}{y^1}\right|^2\right)^{n+2}}
  - \frac{\delta_{ij}}{\left(1+\left|\frac{\vec{y}-\vec{z}}{y^1}\right|^2\right)^n}
  \right] d\vec{z}\nonumber\\
  &\qquad = \frac{n+1}{(y^1)^3}
  \int_{\bR^{n-1}}v(\vec{y} + y^1 \vec{w})\left[
  \frac{4n w^i w^j}{\left(1+\left|\vec{w}\right|^2\right)^{n+2}}
  - \frac{\delta_{ij}}{\left(1+\left|\vec{w}\right|^2\right)^n}
  \right] d\vec{w}.\label{eqEigenFctij}
\end{align}

Note that, from Equation \eqref{eqEigenFct11} we have
\begin{align*}
& \lim_{y^1 \to 0} (y^1)^3 (\hess V - V b)_{11}\\
& \qquad = v(\vec{y}) \int_{\bR^{n-1}}\left[
  \frac{n^2-1}{\left(1+\left|\vec{w}\right|^2\right)^n}
  - \frac{4n(n+1)}{\left(1+\left|\vec{w}\right|^2\right)^{n+1}}
  + \frac{4n(n+1)}{\left(1+\left|\vec{w}\right|^2\right)^{n+2}}
  \right] d\vec{w}\\
& \qquad = v(\vec{y}) \left[(n^2-1) I_{n, n} - 4n(n+1) I_{n, n+1} + 4n(n+1) I_{n, n+2}\right]\\
& \qquad = 0
\end{align*}
where $I_{n, \beta}$ is defined in Lemma \ref{lmIntegral} and where we used basic properties of the gamma function. And, similarly,
\[
\lim_{y^1 \to 0} (y^1)^3 (\hess V - V b)_{1i} = 
\lim_{y^1 \to 0} (y^1)^3 (\hess V - V b)_{ij} = 0
\]
for any pair of indices $i, j \geq 2$. As a consequence, we have
\begin{equation}\label{eqEstimateEigenFctC0}
\left|\hess V - V b\right|_b = o((y^1)^{-1}) = o(\rho^{-1}).
\end{equation}

Our next step will be to deduce the asymptotics of $|\hess V - V b|_b$ assuming that $v$ is $C^2$ from these formulas.  We focus on \eqref{eqEigenFctij} as the 
formulas \eqref{eqEigenFct11} and \eqref{eqEigenFct1i} can be addressed in a similar way. We make a Taylor expansion of $v$ in
the vicinity of an arbitrary $\vec{y}$:
\begin{equation}\label{eqTaylorV}
 v(\vec{y} + y^1 \vec{w}) = v(\vec{y}) + y^1 w^k \pdiff{v}{y^k}(\vec{y})
  + (y^1)^2 \Theta_v(\vec{y}, \vec{w}, y^1).
\end{equation}
Note that, from Equation \eqref{eqRelationvv0} and for any $\epsilon > 0$, there 
exists a constant
\[
C = C_\epsilon(\vec{y}) \left(\|v_0\|_{L^\infty(\bS_1(0), \bR)} + 
\|v_0\|_{C^2(B_\epsilon(\vec{y}), \bR)}\right)
\]
such that $|\Theta_v(\vec{y}, \vec{w}, y^1)| \leq C$ with 
$C_{\epsilon}(\vec{y})$ locally bounded with respect to $\vec{y}$.
We now insert the Taylor 
expansion \eqref{eqTaylorV} into \eqref{eqEigenFctij} and obtain
\begin{align}
&(\hess V - V b)_{ij}\nonumber\\
&\qquad = \frac{n+1}{(y^1)^3}v(\vec{y})
  \int_{\bR^{n-1}}\left[
  \frac{4n w^i w^j}{\left(1+\left|\vec{w}\right|^2\right)^{n+2}}
  - \frac{\delta_{ij}}{\left(1+\left|\vec{w}\right|^2\right)^n}
  \right] d\vec{w} \label{eqDummy1}\tag{A}\\
&\qquad\qquad + \frac{n+1}{(y^1)^2} \pdiff{v}{y^k}(\vec{y})
  \int_{\bR^{n-1}} w^k \left[
  \frac{4n w^i w^j}{\left(1+\left|\vec{w}\right|^2\right)^{n+2}}
  - \frac{\delta_{ij}}{\left(1+\left|\vec{w}\right|^2\right)^n}
  \right] d\vec{w} \label{eqDummy2}\tag{B}\\
&\qquad\qquad + \frac{n+1}{y^1}
  \int_{\bR^{n-1}} \Theta_v(\vec{y}, \vec{w}, y^1) \left[
  \frac{4n w^i w^j}{\left(1+\left|\vec{w}\right|^2\right)^{n+2}}
  - \frac{\delta_{ij}}{\left(1+\left|\vec{w}\right|^2\right)^n}
  \right] d\vec{w} \label{eqDummy3}\tag{C}.
\end{align}
We compute each term separately. If $i \neq j$, the term \eqref{eqDummy1} 
vanishes as the integrand is odd with respect to $w^i$ (and $w^j$)
so integrating with respect to $w^i$ yields $0$. Note also that when $i = j$ the first term in the integrand of \eqref{eqDummy1}  gives rise to the integral 
$\displaystyle
\int_{\bR^{n-1}} \frac{4n(w^i)^2}{\left(1+\left|\vec{w}\right|^2\right)^{n+2}}
d\vec{w}
$
which does not depend on a particular value of $i$. More specifically, we have
\begin{align*}
\int_{\bR^{n-1}} \frac{4n(w^i)^2}{\left(1+\left|\vec{w}\right|^2\right)^{n+2}}
d\vec{w}
&= \frac{1}{n-1} \sum_{j=2}^n
\int_{\bR^{n-1}} \frac{4n(w^j)^2}{\left(1+\left|\vec{w}\right|^2\right)^{n+2}}
d\vec{w}\\
&= \frac{1}{n-1}
\int_{\bR^{n-1}} 
\frac{4n|\vec{w}|^2}{\left(1+\left|\vec{w}\right|^2\right)^{n+2}}
d\vec{w}\\
&= \frac{4n}{n-1} (I_{n, n+1} - I_{n, n+2}).
\end{align*}
By properties of the gamma function, we have
$I_{n, n+1} = \frac{n+1}{2n}I_{n,n}$ and $I_{n, n+2} = \frac{n+3}{4n}I_{n, n}$,
hence
\[
 \int_{\bR^{n-1}} \frac{4n(w^i)^2}{\left(1+\left|\vec{w}\right|^2\right)^{n+2}}
d\vec{w} = I_{n, n} =   
\int_{\bR^{n-1}} \frac{1}{\left(1+\left|\vec{w}\right|^2\right)^n} d\vec{w} 
\]
showing that the term \eqref{eqDummy1} vanishes. The reasoning for \eqref{eqDummy2} is simpler. Assuming $i \neq j$, the integrand is antisymmetric with respect to either $w^i$ or $w^j$ (depending on the value of $k$) so the integral vanishes in this case. If $i = j$, the integrand is antisymmetric with respect to $w^k$ no matter if $k = i$ or $k \neq i$.

\begin{remark}
It should be noted that the vanishing of \eqref{eqDummy1} and \eqref{eqDummy2}
is nothing but the fact that $\hess (V_0+V_1) - (V_0+V_1) b = 0$ (for 
\eqref{eqDummy1}) and $\hess V_i - V_i b = 0$ (for \eqref{eqDummy2}).
\end{remark}

As a consequence, we conclude that the only term remaining in $(\hess V - V b)_{ij}$ is \eqref{eqDummy3}:
\begin{align*}
&(\hess V - V b)_{ij}\\
&\qquad= \frac{n+1}{y^1}
  \int_{\bR^{n-1}} \Theta_v(\vec{y}, \vec{w}, y^1) \left[
  \frac{4n w^i w^j}{\left(1+\left|\vec{w}\right|^2\right)^{n+2}}
  - \frac{\delta_{ij}}{\left(1+\left|\vec{w}\right|^2\right)^n}
  \right] d\vec{w}\\
&\qquad= O((y^1)^{-1}).
\end{align*}

We want to show that this integral does not vanish in general. To this 
end, we note that, by the Taylor-Lagrange theorem, there exists a $\xi = 
\xi(\vec{y}, \vec{w}, y^1)$ in the interval $(0, y^1)$ such that
\[
 \Theta_v(\vec{y}, \vec{w}, y^1) = \frac{1}{2} w^k w^l 
\pdiffm{v}{y^k}{y^l}(\vec{y} + \xi \vec{w}).
\]
Letting $y^1$ tend to zero, we have $\displaystyle \Theta_v(\vec{y}, \vec{w}, 
y^1) \to \frac{1}{2} w^k w^l \pdiffm{v}{y^k}{y^l}(\vec{y})$, so
\begin{equation}\label{eqAsymptoticij}
(\hess V - V b)_{ij}
\sim \frac{n+1}{2 y^1} \pdiffm{v}{y^k}{y^l}(\vec{y})
  \int_{\bR^{n-1}} w^k w^l \left[
  \frac{4n w^i w^j}{\left(1+\left|\vec{w}\right|^2\right)^{n+2}}
  - \frac{\delta_{ij}}{\left(1+\left|\vec{w}\right|^2\right)^n}
  \right] d\vec{w}.
\end{equation}
To evaluate this last integral, we make the simplifying assumption that the Hessian of $v$ at the point $\vec{y}$ under consideration is diagonal in the coordinate basis so we need only consider the case $k = l$. Additionally, note that if $i \neq j$, the integral in \eqref{eqAsymptoticij} vanishes because the integrand is odd with respect to $w^i$. This holds regardless of whether $i=k$ or $i \neq k$. Consequently, it remains only to address the case $i = j$. By symmetry, all the integrals
\begin{equation}\label{eqStupidIntegral}
 J_{i, k} = \int_{\bR^{n-1}} (w^k)^2 \left[
  \frac{4n (w^i)^2}{\left(1+\left|\vec{w}\right|^2\right)^{n+2}}
  - \frac{1}{\left(1+\left|\vec{w}\right|^2\right)^n}
  \right] d\vec{w}
\end{equation}
are equal either to
$\displaystyle
 J_= = \int_{\bR^{n-1}} (w^1)^2 \left[
  \frac{4n (w^1)^2}{\left(1+\left|\vec{w}\right|^2\right)^{n+2}}
  - \frac{1}{\left(1+\left|\vec{w}\right|^2\right)^n}
  \right] d\vec{w}
$ if $i = k$ or to
$\displaystyle
 J_{\neq} = \int_{\bR^{n-1}} (w^1)^2 \left[
  \frac{4n (w^2)^2}{\left(1+\left|\vec{w}\right|^2\right)^{n+2}}
  - \frac{1}{\left(1+\left|\vec{w}\right|^2\right)^n}
  \right] d\vec{w}
$ if $i \neq k$. Summing \eqref{eqStupidIntegral} over $i = 2, \ldots, n$ for 
$k=1$, we get
\[
 J_= + (n-2) J_{\neq} = \int_{\bR^{n-1}} (w^1)^2 \left[
  \frac{4n |\vec{w}|^2}{\left(1+\left|\vec{w}\right|^2\right)^{n+2}}
  - \frac{n-1}{\left(1+\left|\vec{w}\right|^2\right)^n}
  \right] d\vec{w}.
\]

To evaluate all those integrals, we need a lemma:

\begin{lemma}
 The integral
 \begin{equation}
 J_{n, \alpha, \beta} \definedas \int_{\bR^{n-1}} \frac{|w^1|^\alpha}{\left(1 + 
|\vec{w}|^2\right)^\beta} d\vec{w}
 \end{equation}
 is convergent provided that $\beta > \frac{n-2}{2}$ and
 $\alpha \in (-1, 2\beta-n+1)$. In this case, we have
 \[
  J_{n, \alpha, \beta} = \pi^{\frac{n-2}{2}}
   \frac{\Gamma\left(\frac{\alpha+1}{2}\right)
   \Gamma\left(\beta-\frac{n+\alpha-1}{2}\right)}{\Gamma(\beta)}.
 \]
\end{lemma}

\begin{proof}
 The proof is nothing but a refined version of that of Lemma \ref{lmIntegral}.
 We decompose $\vec{w} = (w^1, \wtil)$ and rewrite the integral as follows:
 \begin{align*}
  J_{n, \alpha, \beta}
  &= \int_{\bR^{n-1}} \frac{|w^1|^\alpha}{\left(1 + (w^1)^2 + |\wtil|^2\right)^\beta} d\wtil dw^1\\
  &= \int_{\bR^{n-1}} \frac{|w^1|^\alpha}{\left(1 + (w^1)^2\right)^\beta}
    \left(1 + \frac{|\wtil|^2}{{1 + (w^1)^2}}\right)^{-\beta}  d\vec{w}.
 \end{align*}
 As the integrand is spherically symmetric for $\wtil$, setting $r=|\wtil|$ we get
 \begin{align*}
  J_{n, \alpha, \beta}
  &= \omega_{n-3} \int_{-\infty}^{\infty} \int_0^\infty
   \frac{|w^1|^\alpha}{\left(1 + (w^1)^2\right)^\beta}
    \left(1 + \frac{r^2}{{1 + (w^1)^2}}\right)^{-\beta} r^{n-3} dr dw^1.
 \end{align*}
 Setting $\displaystyle u = \left(1+\frac{r^2}{1+(w^1)^2}\right)^{-1}$ so
 $\displaystyle 2 r dr = - \frac{1+(w^1)^2}{u^2} du$, we get
 \begin{align*}
  J_{n, \alpha, \beta}
  &= \frac{\omega_{n-3}}{2} \int_{-\infty}^{\infty} \int_0^1
   \frac{|w^1|^\alpha}{\left(1 + (w^1)^2\right)^{\beta-1}}
    u^{\beta-2} \left(\left(\frac{1}{u}-1\right)(1+(w^1)^2)\right)^{\frac{n}{2}-2} du dw^1\\
  &= \frac{\omega_{n-3}}{2} \int_{-\infty}^{\infty} \int_0^1
   \frac{|w^1|^\alpha}{\left(1 + (w^1)^2\right)^{\beta - \frac{n}{2} + 1}}
    u^{\beta - \frac{n}{2}} (1-u)^{\frac{n}{2}-2} du dw^1\\
  &= \frac{\omega_{n-3}}{2} B\left(\beta - \frac{n}{2} + 1, \frac{n}{2}-1\right)
   \int_{-\infty}^{\infty} \frac{|w^1|^\alpha}{\left(1 + (w^1)^2\right)^{\beta-\frac{n}{2} + 1}} dw^1\\
  &= \omega_{n-3} B\left(\beta - \frac{n}{2} + 1, \frac{n}{2}-1\right)
   \int_0^{\infty} \frac{(w^1)^\alpha}{\left(1 + (w^1)^2\right)^{\beta-\frac{n}{2} + 1}} dw^1
 \end{align*}
where, as before,  
 $\displaystyle B(x, y) = \int_0^1 t^{x-1} (1-t)^{y-1} dt$ is
the Euler beta function.

 As in the proof of Lemma \ref{lmIntegral}, we perform the change of variable
 $\displaystyle v = (1+(w^1)^2)^{-1}$ and get, by calculations that are similar to the
 previous ones
 \begin{align*}
  J_{n, \alpha, \beta}
 &= \frac{\omega_{n-3}}{2} B\left(\beta - \frac{n}{2} + 1, \frac{n}{2}-1\right)
   \int_0^1 (1-v)^{\frac{\alpha-1}{2}} v^{\beta-\frac{n+\alpha-1}{2} - 1} dv\\
 &= \frac{\omega_{n-3}}{2} B\left(\beta - \frac{n}{2} + 1, \frac{n}{2}-1\right)
   B\left(\frac{\alpha+1}{2}, \beta-\frac{n+\alpha-1}{2}\right)\\
 &= \frac{\pi^{\frac{n-2}{2}}}{\Gamma\left(\frac{n}{2}-1\right)}
   \frac{\Gamma\left(\beta - \frac{n}{2} + 1\right)\Gamma\left(\frac{n}{2}-1\right)}{\Gamma(\beta)}
   \frac{\Gamma\left(\frac{\alpha+1}{2}\right)\Gamma\left(\beta-\frac{n+\alpha-1}{2}\right)}{\Gamma\left(\beta+1-\frac{n}{2}\right)}\\
 &= \pi^{\frac{n-2}{2}}
   \frac{\Gamma(\frac{\alpha+1}{2})\Gamma\left(\beta-\frac{n+\alpha-1}{2}\right)}{\Gamma(\beta)}
 \end{align*} 
 \end{proof}
With this formula, we get after some calculations
\begin{align*}
 J_=
 & = 4n J_{n, 4, n+2} - J_{n, 2, n}\\
 & = \pi^{\frac{n-1}{2}}
  \frac{\Gamma\left(\frac{n-1}{2}\right)}{\Gamma(n)} \frac{n-2}{n+1},\\
 J_= + (n-2) J_{\neq}
 & = \int_{\bR^{n-1}} (w^1)^2 \left[ 4n \left(
  \frac{1}{\left(1+\left|\vec{w}\right|^2\right)^{n+1}} - \frac{1}{\left(1+\left|\vec{w}\right|^2\right)^{n+2}}\right)
  - \frac{n-1}{\left(1+\left|\vec{w}\right|^2\right)^n}
  \right] d\vec{w}.\\
 & = 4n \left(J_{n, 2, n+1} - J_{n, 2, n+2}\right) - (n-1) J_{n, 2, n}\\
 & = 0,
\end{align*}
so
\[
 J_{\neq} = -\frac{\pi^{\frac{n-1}{2}}}{n+1}
  \frac{\Gamma\left(\frac{n-1}{2}\right)}{\Gamma(n)}.
\]
The previous formulas for $J_=$ and $J_{\neq}$ show that they are both non-zero.
The argument here is sufficient to prove that the estimate
$\hess_{ij}V - V b_{ij} = O((y^1)^{-1})$ cannot be improved. With some extra
effort, we can get the exact asymptotics for the tensor $\hess V - V b$:
\begin{equation}\label{eqAsymptotic2ij}
\left\lbrace
\begin{aligned}
 (\hess V - V b)_{ij}
&\sim \frac{\pi^{\frac{n-1}{2}}}{2 y^1} \frac{\Gamma\left(\frac{n-1}{2}\right)}{\Gamma(n)}
 \pdiffm{v}{y^k}{y^l}(\vec{y})
 \left[\delta_{ij} \delta_{kl} - \frac{n-1}{2} (\delta_{ik}\delta_{jl} + \delta_{il}\delta_{jk}) \right],\\
(\hess V - V b)_{11}
&= o((y^1)^{-1}),\\
(\hess V - V b)_{1i}
&= o((y^1)^{-1}).
\end{aligned}
\right.
\end{equation}

This yields the following result:
\begin{lemma}\label{lmEstimateEigenfunction}
 For any $v_0 \in C^2(\bS_1(0), \bR)$, there is a unique function $V \in C^\infty_{-1}(\bH^n, \bR)$
 satisfying $-\Delta V + n V = 0$ and such that $\rho V$ extends by continuity to
 a function on $\overline{B}_1(0)$ with $\rho V \vert_{\bS_1(0)} = v_0$. The function $V$
 satisfies $|\hess V - Vb|_b = O(\rho)$.
\end{lemma}
As we saw earlier, if $v_0$ is merely continuous, we get the estimate $|\hess V - V b|_b = o(\rho^{-1})$ (see Equation \eqref{eqEstimateEigenFctC0}). By real interpolation, we get the following:
\begin{proposition}\label{propEstimateEigenfunction2}
 For any $\alpha \in [0, 2]$ and any $v_0 \in C^\alpha(\bS_1(0), \bR)$ there exists a unique
 solution $V \in C^\infty_{-1}(\bH^n, \bR)$ to the equation $-\Delta V + n V = 0$ such that $\rho V$
 extends by continuity to  a function on $\overline{B}_1(0)$ with $\rho V \vert_{\bS_1(0)} = v_0$.
 The function $V$ satisfies $|\hess V - Vb|_b = O(\rho^{\alpha-1})$.
\end{proposition}

As a consequence, we can make the following definition:
\begin{definition}\label{defEigenspace}
For any $\alpha \in [0, 2]$, let $\cE^\alpha$ be the set of functions $V: \bH^n \to \bH^n$ satisfying
\begin{equation}\label{eqEigenfunction0}
-\Delta V + n V = 0
\end{equation}
and such that $\rho V$ restricts to a $C^\alpha$ function on $\bS_1(0)$. We let $\cD$ be the mapping
\[
\cD: C^{\alpha}(\bS_1(0)) \rightarrow \cE^\alpha
\]
associating to a given function $v \in C^\alpha(\bS_1(0), \bR)$ the unique function $V \in C^\infty_{-1}(\bH^n, \bR)$ solving \eqref{eqEigenfunction0} such that $\rho V$ restricts to $v$ on $\bS_1(0)$\footnote{$\cD$ stands here for Dirichlet as we are solving an asymptotic Dirichlet problem.}. It follows from the previous proposition that $\cD$ is well-defined and bijective.
\end{definition}

To prove this proposition, we first recall some basic facts from real interpolation theory.
We use the notations from \cite{LunardiSemigroups,LunardiInterpolation}.
Assume given two Banach spaces  $X$ and $Y$ such that there is a continuous injective map $Y \hookrightarrow X$
allowing us to view $Y$ as a subspace of $X$. For any $u \in X$, we set
\[
K(t, u) \definedas \inf (\|a\|_X + t \|b\|_Y),
\]
where the infimum is taken over all pairs $(a, b) \in X \times Y$ such that $u = a + b$. For any
$\theta \in (0, 1)$, let $(X, Y)_{\theta, \infty}$ be the set of elements $u \in X$ such that
\[
 \|u\|_{\theta, \infty} \definedas \sup_{t \in (0, \infty)} t^{-\theta} K(t, u) < \infty.
\]
Then the space $(X, Y)_{\theta, \infty} \definedas \{u \in X,~\|u\|_{\theta, \infty} < \infty\}$ is a
Banach space when endowed with the norm $\|u\|_{\theta, \infty}$.

\begin{lemma}\label{lmInterpolation}
Let $E$ be a geometric tensor bundle on $\bH^n$. Then for any $\delta_1 < \delta_2$, we have
\[
C^{0}_{\delta} (\bH^n, E) = (C^{0}_{\delta_1}(\bH^n, E), C^{0}_{\delta_2}(\bH^n, E))_{\theta, \infty},
\]
where $\delta \in (\delta_1, \delta_2)$ and $\theta \in (0, 1)$ is given by
\[
 \delta = (1-\theta) \delta_1 + \theta \delta_2.
\]
Further, the norms $\|\cdot\|_{C^{0}_\delta}$ and $\|\cdot\|_{\theta, \infty}$ are equivalent.
\end{lemma}

\begin{proof}
Assume first that
$u \in (C^{0}_{\delta_1}(\bH^n, E), C^{0}_{\delta_2}(\bH^n, E))_{\theta, \infty}$. For any $x \in \bH^n$, select a
pair $(a, b)$ with $a \in C^{0}_{\delta_1}(M, E)$ and $b \in C^{0}_{\delta_2}(M, E)$ such that
$u = a + b$. Then
\[
\rho(x)^{-\delta} |u(x)| \leq \rho(x)^{-\delta}\left(|a(x)| + |b(x)|\right)
 \leq \rho(x)^{\delta_1-\delta} \|a\|_{C^{0}_{\delta_1}}
   + \rho(x)^{\delta_2-\delta} \|b\|_{C^{0}_{\delta_2}}.
\]
Set $t = \rho(x)^{\delta_2-\delta_1}$. We have $\delta_2 - \delta = (1-\theta)(\delta_2 - \delta_1)$ and
$\delta_1 - \delta = - \theta (\delta_2 - \delta_1)$. So, from the previous inequality, we infer
\[
\rho(x)^{-\delta} |u(x)| \leq t^{-\theta} \|a\|_{C^{0}_{\delta_1}}
   + t^{1-\theta} \|b\|_{C^{0}_{\delta_2}}.
\]
As the chosen pair $(a, b)$ was arbitrary, we can pass to the infimum and get that
\[
\rho(x)^{-\delta} |u(x)| \leq t^{-\theta} K(t, u) \leq \|u\|_{\theta, \infty}.
\]
Taking the supremum with respect to $x \in \bH^n$, this shows that
$\|u\|_{C^{0}_\delta} \leq \|u\|_{\theta, \infty} < \infty$. In particular,
$u \in C^{0}_\delta(M, E)$. We have proven that
\[
(C^{0}_{\delta_1}(\bH^n, E), C^{0}_{\delta_2}(\bH^n, E))_{\theta, \infty}
\subset C^{0}_{\delta} (\bH^n, E).
\]

To prove the reverse inclusion, we choose a (smooth) cutoff function $\chi: [0, \infty) \to [0, 1]$ such that
\[
\chi(s) = 
\left\lbrace
\begin{aligned}
1 & \qquad \text{on } [0, 1/2],\\
0 & \qquad \text{on } [1, \infty).
\end{aligned}
\right.
\]
Given an arbitrary $u \in C^{0}_\delta(\bH^n, E)$, we set
\[
a \definedas \chi\left(\frac{\rho}{\rho_0}\right) u, \quad
b \definedas u - a = \left[1 - \chi\left(\frac{\rho}{\rho_0}\right)\right] u,
\]
for some constant $\rho_0 > 0$ to be chosen later. We have
\[
\rho^{-\delta_1}(x) |a(x)|
  \leq \chi\left(\frac{\rho(x)}{\rho_0}\right) \rho(x)^{\delta-\delta_1} \|u\|_{C^{0}_\delta}.
\]
As $\chi(\rho/\rho_0) = 0$ if $\rho/\rho_0 \geq 1$, we conclude that
\[
\rho^{-\delta_1}(x) |a(x)|
  \leq \rho_0^{\delta-\delta_1} \|u\|_{C^{0}_{\delta}}
\]
where we have used the fact that $\delta_1 \leq \delta$.
Thus, $\|a\|_{C^{0}_{\delta_1}} \leq \rho_0^{\delta-\delta_1} \|u\|_{C^{0}_\delta}$.
Similarly, we have
\[
\|b\|_{C^{0}_{\delta_2}} \leq \left(\frac{\rho_0}{2}\right)^{\delta-\delta_2} \|u\|_{C^{0}_\delta}.
\]
This shows that, for all $t \in (0, \infty)$, we have
\begin{align*}
t^{-\theta} K(t, u)
 & \leq t^{-\theta} \left[\|a\|_{C^{0}_{\delta_1}} + t \|b\|_{C^{0}_{\delta_2}}\right]\\
 & \leq t^{-\theta} \left[\rho_0^{\delta-\delta_1} + t \left(\frac{\rho_0}{2}\right)^{\delta-\delta_2}\right] \|u\|_{C^{0}_\delta}.
\end{align*}
We now adjust $\rho_0$ to minimize the right-hand side of this inequality. The optimal value for $\rho_0$ is given by
\[
\rho_0^{\delta_2 - \delta_1} = \frac{t}{2^{\delta-\delta_2}} \frac{\delta_2 - \delta}{\delta - \delta_1}
 = \frac{t}{2^{\delta-\delta_2}} \frac{1 - \theta}{\theta}.
\]
This yields
\[
t^{-\theta} K(t, u) \leq C \|u\|_{C^{0}_\delta}
\]
for some explicit constant $C = C(\delta, \delta_1, \delta_2)$. Taking the supremum over $t$, we conclude
that $\|u\|_{\theta, \infty} \leq C \|u\|_{C^{0}_\delta}$. This concludes the proof of the lemma.
\end{proof}

\begin{proof}[Proof of Proposition \ref{propEstimateEigenfunction2}]
From Proposition \ref{propEstimateEigenfunction0}, there exists a map $T$ from $C^0(\bS_1(0), \bR)$ to $C^0_{-1}(\bH^n, S_2\bH^n)$ mapping a function $v_0$ to the unique function $V$ satisfying
\[
-\Delta V + n V = 0,\quad \rho V \vert_{\bS_1(0)} \equiv v_0
\]
and then to the symmetric tensor $\hess V - V b$. $T$ is a continuous mapping. By Lemma \ref{lmEstimateEigenfunction}, $T$ maps continuously $C^2(\bS_1(0), \bR)$ to $C^0_1(\bH^n, S_2\bH^n)$.

We now use interpolation theory. For any $\theta \in (0, 1)$, it follows from Lemma \ref{lmInterpolation} that $T$ induces a map from $[C^0(\bS_1(0), \bR), C^2(\bS_1(0), \bR)]_{\theta, \infty}$ to $C^0_{-1 + 2 \theta}(\bH^n, S_2\bH^n)$.
It is known that $[C^0(\bS_1(0), \bR), C^2(\bS_1(0), \bR)]_{\theta, \infty}$ corresponds to $C^{2\theta}(\bS_1(0), \bR)$ for all $\theta \neq 1/2$ and that $[C^0(\bS_1(0), \bR), C^2(\bS_1(0), \bR)]_{1/2, \infty}$ is the H\"older-Zygmund space $C^1_*(\bS_1(0), \bR)$ into which $C^1(\bS_1(0), \bR)$ embeds continuously (see e.g. \cite[Section 1.2.4]{LunardiSemigroups}). This ends the proof of Proposition \ref{propEstimateEigenfunction2}.
\end{proof}

\section{Weak definitions of the mass aspect function}\label{secMassAspect}

\subsection{An ADM style definition}\label{secADM}

In this section, we choose once and for all an asymptotically hyperbolic manifold $(M, g)$ and a chart at infinity $\Phi$. To keep the notation short, we will systematically write $g$ instead of $\Phi_* g$.

Our goal is to give a weak definition of the mass aspect function. Following \cite{LeeLeFloch}, we choose to work in a low regularity context. However, the strategy we will adopt is slightly different: we work with cutoff functions. This will allow us to reach even weaker regularity as all we will have to allow is that $g$ satisfies the assumptions of Definition \ref{defAH} together with the assumption that the scalar curvature of $g$ (defined in the sense of distributions) satisfies $\scal^g + n(n-1) \in L^1_{\tau'}(\bH^n\setminus K', \bR)$ for some $\tau, \tau' > 0$ to be made precise later (see also Remark \ref{rkRegScalar} for a discussion of allowed weaker regularity of the scalar curvature). The role of the cutoff function will be to replace the potentially ill-defined boundary terms that are encountered in the usual definition of the mass as $D\chi_k$ can be thought as a replacement of the inward pointing normal vector to the boundary of the domains. To maintain this correspondence, it will be important not to introduce second order derivatives of $\chi_k$.

Let $V$ be a smooth function on the hyperbolic space and $\chi_k$ a family of cutoff functions compactly supported in $\bH^n \setminus K'$,where $K'$ is as in Definition \ref{defAH}. Both $V$ and $\chi_k$ will
be described precisely later. Using the first order variation of the scalar curvature \eqref{eqScalarVar1}, we compute
\begin{align*}
 & \int_{\bH^n} \chi_k V \left(\scal^g + n(n-1)\right) d\mu^b                                                                                                              \\
 & = \int_{\bH^n} \chi_k V \left(D_i \left[g^{jl} g^{im} \left(D_j e_{ml} - D_m e_{jl}\right)\right] + (n-1) \tr e + \cQ_0(e, De)\right) d\mu^b                             \\
 & = \int_{\bH^n} \left(-(D_i (\chi_k V)) \left[g^{jl} g^{im} \left(D_j e_{ml} - D_m e_{jl}\right)\right] + (n-1) \chi_k V \tr e + \chi_k V \cQ_0(e, De)\right) d\mu^b      \\
 & = \int_{\bH^n} \left[(- V D_i \chi_k - \chi_k D_i V) \left(g^{jl} g^{im} - g^{ml} g^{ij}\right) D_j e_{ml} + (n-1) \chi_k V \tr e + \chi_k V \cQ_0(e, De)\right] d\mu^b.
\end{align*}

Next, we perform a second integration by parts noting that

\begin{align*}
& \chi_k (D_i V) \left(g^{jl} g^{im} - g^{ml} g^{ij}\right) D_j e_{ml}                                                                                                           \\
& \qquad = - D_j\left[\chi_k (D_i V) \left(g^{jl} g^{im} - g^{ml} g^{ij}\right)\right] e_{ml} + D_j\left[\chi_k (D_i V) \left(g^{jl} g^{im} - g^{ml} g^{ij}\right) e_{ml}\right] \\
& \qquad = - \chi_k (D_j D_i V) \left(g^{jl} g^{im} - g^{ml} g^{ij}\right) e_{ml} - \chi_k (D_i V) D_j\left(g^{jl} g^{im} - g^{ml} g^{ij}\right) e_{ml}                          \\
& \qquad\qquad - (D_j\chi_k) (D_i V) \left(g^{jl} g^{im} - g^{ml} g^{ij}\right) e_{ml} + D_j\left[\chi_k (D_i V) \left(g^{jl} g^{im} - g^{ml} g^{ij}\right) e_{ml}\right]        \\
& \qquad = - \chi_k (D_j D_i V) \left(g^{jl} g^{im} - g^{ml} g^{ij}\right) e_{ml} + D_j\left[\chi_k (D_i V) \left(g^{jl} g^{im} - g^{ml} g^{ij}\right) e_{ml}\right]             \\
& \qquad\qquad - (D_j\chi_k) (D_i V) \left(g^{jl} g^{im} - g^{ml} g^{ij}\right) e_{ml} + \chi_k \cQ_1(dV, e, De),
\end{align*}
where $\cQ_1(dV, e, De)$ denotes a function that satisfies $|\cQ_1(dV, e, De)| \lesssim |dV| |e| |De|$. Further, we have
\begin{align*}
 & -(D_j D_i V) \left(g^{jl} g^{im} - g^{ml} g^{ij}\right) e_{ml}                              \\
 & \qquad = -(D_j D_i V) \left(b^{jl} b^{im} - b^{ml} b^{ij}\right) e_{ml} + \cQ_2(\hess V, e) \\
 & \qquad = -b(\hess V, e) + (\Delta V) \tr(e) + \cQ_2(\hess V, e),
\end{align*}
where $\cQ_2(\hess V, e)$ denotes a function that satisfies $|\cQ_2(\hess V, e)| \lesssim |\hess V| |e|^2$. So,
\begin{align*}
& \chi_k (D_i V) \left(g^{jl} g^{im} - g^{ml} g^{ij}\right) D_j e_{ml}                                                                         \\
& \qquad = \chi_k\left(-b(\hess V, e) + (\Delta V) \tr(e)\right) + D_j\left[\chi_k (D_i V) \left(g^{jl} g^{ik} - g^{kl} g^{ij}\right) e_{kl}\right] \\
& \qquad\qquad - (D_j\chi_k) (D_i V) \left(g^{jl} g^{im} - g^{ml} g^{ij}\right) e_{ml} + \cQ_1(dV, e, De) + \cQ_2(\hess V, e).
\end{align*}
We have thus proven the following formula:
\begin{equation}\label{eqCharge}
\begin{aligned}
& \int_{\bH^n} \chi_k V \left(\scal^g + n(n-1)\right) d\mu^b                                                                                                                                \\
& \qquad = \int_{\bH^n} \chi_k \left( b(\hess V, e) - (\Delta V) \tr(e) + (n-1) V \tr e + \cQ(V, e, De)\right) d\mu^b                                                                      \\
& \qquad\qquad + \int_{\bH^n} \left[(D_j\chi_k) (D_i V) \left(g^{jl} g^{im} - g^{ml} g^{ij}\right) e_{ml} - V D_i \chi_k \left(g^{jl} g^{im} - g^{ml} g^{ij}\right) D_j e_{ml}\right] d\mu^b
\end{aligned}
\end{equation}
where $\displaystyle \cQ(V, e, De) = V \cQ_0(e, De) - \cQ_1(dV, e, De) - \cQ_2(\hess V, e)$ satisfies
\begin{equation}
\label{eqQTerm}
|\cQ(V, e, De)| \lesssim (|V| + |dV| + |\hess V|) \left(|e|^2 + |De|^2\right).
\end{equation}
The analysis performed in \cite{ChruscielNagy,ChruscielHerzlich} (see also \cite{MichelMass} for a broader point of view) then consists in choosing for $V$ a function that satisfies
\begin{equation}\label{eqDefLapse}
\hess V - \Delta V b + (n-1) V b = 0.
\end{equation}
A straightforward calculation shows that this equation is equivalent to \eqref{eqLapse}, i.e. $V \in \cN$ if and only if \eqref{eqLapse} holds. For such a $V$, \eqref{eqCharge} reduces to
\begin{equation}\label{eqCharge1}
\begin{aligned}
 & \int_{\bH^n} \chi_k V \left(\scal^g + n(n-1)\right) d\mu^b - \int_{\bH^n} \chi_k \cQ(V, e, De) d\mu^b                                                                               \\
 & \qquad = \int_{\bH^n} \left[(D_j\chi_k) (D_i V) \left(g^{jl} g^{im} - g^{ml} g^{ij}\right) e_{ml} - V D_i \chi_k \left(g^{jl} g^{im} - g^{ml} g^{ij}\right) D_j e_{ml}\right] d\mu^b.
\end{aligned}
\end{equation}

Before proceeding with the next proposition, we clarify our choice of notation. Up to this point, we have used the symbol $\bS_1(0)$ to represent the sphere at infinity of hyperbolic space, aligning with its interpretation as the unit sphere in Euclidean space. However, since we will frequently refer to spheres defined in terms of the hyperbolic metric, we introduce a distinction: $S_r(0)$ will denote the geodesic sphere of radius $r$ centered at the origin $0$ in $\bH^n$. For easier reference, we summarize this convention as follows:
\vspace{.2cm}

\noindent\fbox{ \parbox{\textwidth}{%
\begin{itemize}
\item\label{rkNotation} $\bS_1(0)$ denotes the sphere at infinity of hyperbolic space.
\item $S_r(0)$ denotes the hyperbolic geodesic sphere and $B_r(0)$ the hyperbolic open ball of radius $r$ centered at the origin.
\end{itemize}
} }

\begin{proposition}\label{propMassCH}
Assume that $g$ is asymptotically hyperbolic of order $\tau \geq 1/2$ in the sense of Definition \ref{defAH} and that $\scal^g+n(n-1) \in L^1_1(M, \bR)$. Let $(\chibar_k)_k$ be a sequence of compactly supported Lipschitz functions over $\bH^n$ with uniformly bounded $C^{0, 1}_0$-norm and such that the sets $\Omega_k = \chibar_k^{-1}(1)$ form an increasing sequence of compact sets such that $\bH^n = \bigcup_k \Omega_k$. Then, for any $V \in \cN$, the limit
\begin{equation}\label{eqChargeCH}
    p(e, V) = \lim_{k \to \infty} \int_{\bH^n} \left[V(\divg(e) - d\tr(e))(-D\chibar_k) + \tr(e) dV(-D\chibar_k) - e(DV, -D\chibar_k)\right] d\mu^b
\end{equation}
is well defined and independent of the choice 
of the sequence $(\chibar_k)_k$.

Assume further that $e \in C^1_{\tau}(\bH^n, \bR)$ for some $\tau > n/2$. Then we recover the classical formula
\begin{equation}\label{eqChargeCH2}
    p(e, V) = \lim_{r \to \infty} \int_{S_r(0)} \left[V(\divg(e) - d\tr(e))(\nu) + \tr(e) dV(\nu) - e(DV, \nu)\right] d\mu^b.
\end{equation}
\end{proposition}

Before delving into the proof of this proposition we would like to make a few comments. As we can see from Proposition \eqref{propMassCH}, in the case of a smooth metric $g$, the right hand side of Formula \eqref{eqCharge1} can be replaced by the corresponding surface integral \eqref{eqChargeCH2} as is done normally. However, when $g$ is not in $C^{1, 0}_{loc}$, we cannot make sense of the trace of $De$ on a hypersurface, so, for weakly regular metrics, \eqref{eqChargeCH2} does not make sense per se.

\begin{proof}
The proof is based on analysing \eqref{eqCharge1} more closely. Let $\chi_0$ be a smooth function on $\bH^n$ which vanishes near $K'$ (see Definition \ref{defAH}) and equals $1$ outside some compact set. We define  $\chi_k \definedas \chi_0 \chibar_k$.

Note that $V = O(\rho^{-1})$ so, by assumption, $V(\scal^g+n(n-1)) \in L^1(M, \bR)$. As a consequence, it follows from the dominated convergence theorem that
\[
    \int_{\bH^n} \chi_k V \left(\scal^g + n(n-1)\right) d\mu^b
    \to \int_{\bH^n} \chi_0 V \left(\scal^g + n(n-1)\right) d\mu^b < \infty.
\]
Further, a simple calculation shows that
\[
    V + |dV| + |\hess V| = O(\rho^{-1})
\]
(this is obvious for $V$ and $\hess V$ as it follows from \eqref{eqDefLapse} that $\hess V = Vb$ only the estimate for $|dV|$ requires some calculation). As a consequence, from the assumption that $e$ satisfies the estimate \eqref{eqEstimateE}, we conclude that $\cQ(V, e, De) \in L^1(M, \bR)$. Once again, by the dominated convergence theorem, we obtain that
\[
    \int_{\bH^n} \chi_k \cQ(V, e, De) d\mu^b
    \to \int_{\bH^n} \chi_0 \cQ(V, e, De) d\mu^b.
\]
We have proven that the left hand side in \eqref{eqCharge1}
admits a limit as $k$ goes to infinity. Consequently, all we need to
prove is that the difference between the right hand side of
\eqref{eqCharge1} and what appears in \eqref{eqChargeCH} tends
to a constant as $k$ tends to infinity.
Since
\[
    d\chi_k - d\chibar_k = (\chi_0-1) d\chibar_k + \chibar_k d\chi_0,
\]
we will assume that $k$ is large enough so that
$\supp (\chi_0-1) \subset \Omega_k$, in which case we have
\[
    d\chi_k - d\chibar_k = d\chi_0.
\]
As a consequence, we have
\[
\begin{aligned}
 & \int_{\bH^n} \left[V(\divg(e) - d\tr(e))(-D\chibar_k) + \tr(e) dV(-D\chibar_k) - e(DV, -D\chibar_k)\right] d\mu^b     \\
 & \qquad = \int_{\bH^n} \left[V(\divg(e) - d\tr(e))(-D\chi_k) + \tr(e) dV(-D\chi_k) - e(DV, -D\chi_k)\right] d\mu^b     \\
 & \qquad\qquad + \int_{\bH^n} \left[V(\divg(e) - d\tr(e))(D\chi_0) + \tr(e) dV(D\chi_0) - e(DV, D\chi_0)\right] d\mu^b.
\end{aligned}
\]
The second term in the right hand side of this formula is independent of $k$.
We also note that
\[
\begin{aligned}
 & \int_{\bH^n} \left[V(\divg(e) - d\tr(e))(-D\chibar_k) + \tr(e) dV(-D\chibar_k) - e(DV, -D\chibar_k)\right] d\mu^b\\
 & \qquad = \int_{\bH^n} \left[(D_j\chibar_k) (D_i V) \left(b^{jl} b^{im} - b^{ml} b^{ij}\right) e_{ml} - V D_i \chibar_k \left(b^{jl} b^{im} - b^{ml} b^{ij}\right) D_j e_{ml}\right] d\mu^b.
\end{aligned}
\]
To complete the proof of the first part of the proposition, we will show that
\begin{equation}\label{eqDiffCharge0}
\int_{\bH^n} \left((D_j\chibar_k) (D_i V) e_{ml} - V D_i \chibar_k D_j e_{ml}\right) \left[\left(b^{jl} b^{im} - b^{ml} b^{ij}\right) - \left(g^{jl} g^{im} - g^{ml} g^{ij}\right)\right] d\mu^b
\end{equation}
tends to zero. Developing \eqref{eqDiffCharge0}, we get four terms:
\begin{align*}
 \int_{\bH^n} (D_j\chibar_k) (D_i V) e_{ml} \left(b^{jl} b^{im} - g^{jl} g^{im}\right) d\mu^b
& - \int_{\bH^n} (D_j\chibar_k) (D_i V) e_{ml} \left(b^{ml} b^{ij} - g^{ml} g^{ij}\right) d\mu^b\\
\qquad  - \int_{\bH^n} V D_i \chibar_k D_j e_{ml} \left(b^{jl} b^{im} - g^{jl} g^{im}\right) d\mu^b
& + \int_{\bH^n} (D_j\chibar_k) (D_i V) e_{ml} \left(b^{jl} b^{im} - g^{jl} g^{im}\right) d\mu^b.
\end{align*}
Since each of these four terms can be handled in a similar way, we provide a complete argument only for the first one, showing that it tends to zero when $k\to\infty$. First, we note that
\begin{align*}
 & \int_{\bH^n} (D_j\chibar_k) (D_i V) e_{ml} \left(b^{jl} b^{im} - g^{jl} g^{im}\right) d\mu^b                                                      \\
 & \qquad = \int_{\bH^n} (D_j\chibar_k) (D_i V) e_{ml} \left[b^{jl} \left(b^{im} - g^{im}\right) + \left(b^{jl} - g^{jl}\right) g^{im}\right] d\mu^b \\
 & \qquad = -\int_{\bH^n} (D_j\chibar_k) (D_i V) e_{ml} \left[b^{jl} f^{im} + f^{jl} g^{im}\right] d\mu^b.
\end{align*}
We now apply the estimate \eqref{eqEFnorm} and get
\begin{align*}
 & \left|\int_{\bH^n} (D_j\chibar_k) (D_i V) e_{ml} \left(b^{jl} b^{im} - g^{jl} g^{im}\right) d\mu^b\right| \\
 & \qquad\lesssim \left\|D\chibar_k\right\|_{L^\infty} \int_{\supp (D\chibar_k)} |dV| |e|^2 d\mu^b.
\end{align*}
As $|dV| = O(\rho^{-1})$ and $e \in L^2_{1/2}(\bH^n, \bR)$, we have
$|dV| |e|^2 \in L^1(\bH^n, \bR)$. As $\bigcup_k \Omega_k = \bH^n$,
we have that
\[
    \int_{\Omega_k} |dV| |e|^2 d\mu^b \to \int_{\bH^n} |dV| |e|^2 d\mu^b
\]
(this is a consequence of the monotone convergence theorem).
In particular,
\[
    \int_{\supp (D\chibar_k)} |dV| |e|^2 d\mu^b \leq \int_{\bH^n \setminus \Omega_k} |dV| |e|^2 d\mu^b \to_{k \to \infty} 0.
\]
This ends the proof of the first part of the proposition. Note that, in summary, we have shown that
\begin{equation}\label{eqCharge2}
\begin{aligned}
p(e, V)= & \int_{\bH^n} \chi_0 V \left(\scal^g + n(n-1)\right) d\mu^b - \int_{\bH^n} \chi_0 \cQ(V, e, De) d\mu^b                               \\
& - \int_{\bH^n} \left((D_i\chi_0) (D_m V) e_{jl} - V D_i \chi_0  D_j e_{ml}\right) \left(g^{jl} g^{im} - g^{ml} g^{ij}\right) d\mu^b.
\end{aligned}
\end{equation}

We now turn to the second part of the proposition proving the equivalence between \eqref{eqChargeCH} and \eqref{eqChargeCH2}. For this we recall the basic identity giving rise to the mass integral in \cite[Section 2]{ChruscielHerzlich}: 
\begin{equation}\label{eqVScalCH}
V(\scal^g + n(n-1)) = \divg \mathbb{U}(e,V) +\cQ(V,e,De),
\end{equation}
where
\begin{equation*}
\mathbb{U}^i (e,V) = (g^{im}g^{jl} - g^{ij}g^{ml}) (V  D_j e_{ml} - D_ m V e_{jl}),
\end{equation*}
and $\mathcal{Q}(V,e, De)$ is in $L^1(\bH^n\setminus K', \bR)$. In fact, this formula can be derived along the lines of the computation preceding to \eqref{eqCharge} but without performing integration by parts against the compactly supported $\chi_k$. In particular, it is easy to  see that the quadratic terms $\cQ$ in \eqref{eqCharge1} and in \eqref{eqVScalCH} are exactly the same. Multiplying \eqref{eqVScalCH} by $\chi_0$ and integrating by parts over $\bH^n$ we obtain
\begin{equation}\label{eqIntVScalCH}
\begin{split}
& \int_{\bH^n} \chi_0 V(\scal^g + n(n-1)) d\mu^b \\
& \quad = \int_{\bH^n} \chi_0 \cQ(V,e,De) \, d\mu^b + \lim_{r\to \infty} \int_{S_r(0)} (g^{im}g^{jl} - g^{ij}g^{ml}) (V  D_j e_{ml} - D_ m V e_{jl}) \nu_i \, d\mu^b\\ & \qquad + \int_{\bH^n} \left((D_i\chi_0) (D_m V) e_{jl} - V D_i \chi_0  D_j e_{ml}\right) \left(g^{jl} g^{im} - g^{ml} g^{ij}\right) d\mu^b ,
\end{split}
\end{equation}
where we have used the fact that $\chi_0\equiv 1$ near infinity. In the view of our fall off conditions we may replace $g$ by $b$ in the second term in the right hand side of \eqref{eqIntVScalCH}, so this term equals  the right hand side of \eqref{eqChargeCH}. Recalling \eqref{eqCharge2}, we establish the asserted equivalence of the definitions  \eqref{eqChargeCH} and \eqref{eqChargeCH2}, under the given regularity assumptions.
\end{proof}

We will now analyse Formula \eqref{eqCharge} more carefully, in order to allow for more general functions $V$. To begin with, we note that it is not necessary to choose $V \in \cN$ to make sense of \eqref{eqCharge}. In fact, any $V$ such that the first integrand in the right hand side is in $L^1(\bH^n \setminus K', \bR)$ will give a well defined charge along the lines of Proposition \ref{propMassCH}. We explore this condition in more detail building upon the results of Section \ref{secEstimates}.

First, note that, setting $T \definedas \hess V - V b$, we have
\begin{align*}
     & \int_{\bH^n} \chi_k \left( b(\hess V, e) - (\Delta V) \tr(e) + (n-1) V \tr e\right) d\mu^b \\
     & \qquad = \int_{\bH^n} \chi_k b(T - \tr(T) b, e)  d\mu^b.
\end{align*}
As $\displaystyle |T - \tr(T) b|^2 = |T|^2 + (n-2) (\tr(T))^2$, by the inequality $|\tr(T)| \leq \sqrt{n} |T|$, we get $|T| \leq |T - \tr(T) b| \leq (n-1) |T|$. This shows that the asymptotic behavior of $|\hess V - \Delta V b + (n-1) V b|$ is dictated by that of $|\hess V - Vb|$.

In particular, if $|\hess V - Vb| = O(\rho^\mu)$, for $\mu \in [-1, 1]$, the integral
\[
    \int_{\bH^n} \chi_0 \left( b(\hess V, e) - (\Delta V) \tr(e) + (n-1) V \tr e\right) d\mu^b
\]
is well defined provided that $e \in L^1_{-\mu}(\bH^n, S_2\bH^n)$, or, if one insists on imposing $L^2$-estimates for $e$, $e \in L^2_{\tau}(\bH^n, S_2\bH^n)$ for $\tau > -\mu + \frac{n-1}{2}$ (see \cite[Lemma 3.6]{LeeFredholm}). 

In summary, we obtain the following result (the proof follows that of Proposition \ref{propMassCH} almost verbatim, and the only difference has been described above):

\begin{proposition}\label{propMassAspect}
Assume that $g$ is asymptotically hyperbolic of order $\tau \geq 1/2$ in the sense of Definition \ref{defAH}, that $e \in L^1_\delta(\bH^n, S_2\bH^n)$, with $\delta \in [-1, 1]$ and that $\scal^g+n(n-1) \in L^1_1(M, \bR)$.

Let $(\chibar_k)_k$ be a sequence of compactly supported Lipschitz functions defined on $\bH^n$ with uniformly bounded $C^{0, 1}_0$-norm and such that the sets $\Omega_k = \chibar_k^{-1}(1)$ form an increasing sequence of compact sets with $\bH^n = \bigcup_k \Omega_k$. For any arbitrary $V \in \cE^{1-\delta}$ (see Definition \ref{defEigenspace}), the limit
\begin{equation}\label{eqChargeAspect}
    p(e, V) = \lim_{k \to \infty} \int_{\bH^n} \left[V(\divg(e) - d\tr(e))(-D\chibar_k) + \tr(e) dV(-D\chibar_k) - e(DV, -D\chibar_k)\right] d\mu^b
\end{equation}
is well defined and independent of the choice of the sequence $(\chibar_k)_k$.

Assume further that $e \in C^1_{\delta}(\bH^n, \bR)$ for some $\delta > n/2$. Then we recover the classical formula
\[
p(e, V) = \lim_{r \to \infty} \int_{S_r(0)} \left[V(\divg(e) - d\tr(e))(\nu) + \tr(e) dV(\nu) - e(DV, \nu)\right] d\mu^b.
\]
\end{proposition}
Note that the condition $e \in L^1_\delta(\bH^n, S_2\bH^n)$ is fulfilled if e.g. $e \in L^\infty_\mu$ for $\mu > \delta + n-1$ (this is a consequence of \cite[Lemma 3.6]{LeeFredholm}) showing that the required decay of $e$ is weaker than in Wang's context \cite{WangMass} described in the introduction.
In the rest of this section, our focus will be on relating our ADM-style definition \eqref{eqChargeAspect} to the definition of Wang. 

We start with a proposition allowing us to view the charges we defined in \eqref{eqChargeAspect} as distributions on $\bS_1(0)$. To this end, keeping the notations as in Proposition \ref{propMassAspect}, let $P(e, \cdot)$ be the mapping defined as follows. For any $v \in C^{1-\delta}(\bS_1(0), \bR)$, $\delta \in [-1, 1]$,
let $V = \cD(v)$ denote the eigenfunction of the Laplacian such that $\rho V \equiv v$ on $\bS_1(0)$ (see Definition \ref{defEigenspace}). We set
\[
    P(e, v) \definedas p(e, V).
\]
We have the following result:
\begin{proposition}\label{propContinuity}
Under the assumptions of Proposition \ref{propMassAspect}, the mapping $P(e, \cdot): C^{1-\delta}(\bS_1(0), \bR) \to \bR$ is continuous.
\end{proposition}

\begin{proof}
As $P$ is linear in $V$ we have to prove that $|P(e, v)| \lesssim \|v\|_{C^{1-\delta}(\bS_1(0), \bR)}$. From \eqref{eqCharge} and the proof of Proposition \ref{propMassAspect}, we have
\begin{align*}
P(e, v)
& = \int_{\bH^n} \chi_0 V \left(\scal^g + n(n-1)\right) d\mu^b\\
& \quad - \int_{\bH^n} \chi_0 \left( b(\hess V - (\Delta V)b + (n-1) V b, e) + \cQ(V, e, De)\right) d\mu^b                                                                           \\
& \quad - \int_{\bH^n} \left[(D_j\chi_0) (D_i V) \left(g^{jl} g^{ik} - g^{kl} g^{ij}\right) e_{kl} - V D_i \chi_0 \left(g^{jl} g^{ik} - g^{kl} g^{ij}\right) D_j e_{kl}\right] d\mu^b.
\end{align*}
The first and the last terms are easily seen to be controlled by the $C^1_{-1}$-norm of $V$.
Hence, from Proposition \ref{propEstimateEigenfunction0}, by the $L^\infty$-norm of $v$. The second term is controlled by the $C^{1-\delta}$-norm of $v$ as it follows from Proposition \ref{propEstimateEigenfunction2}.
\end{proof}

Next, we show how our definition of the mapping $P$ in Proposition \ref{propContinuity} relates to Wang's definition of the mass aspect function:
\begin{proposition}\label{propMassAspect2}
Assume that $g$ is asymptotically hyperbolic in the sense of Wang: that is, there exists a chart at infinity $\Phi: M \setminus K \to \bH^n \setminus K'$ such that the tensor $\ebar \definedas \rho^{n-2} e$ with $e \definedas \Phi_* g - b$ extends by continuity to $\Bbar(0, 1) \setminus K'$ to a tensor satisfying $\ebar_{ij} x^i = 0$. Then the operator $P(e, \cdot): C^\epsilon(\bS_A(0), \bR) \to \bR$ introduced in Proposition \ref{propContinuity} is well-defined for any $\epsilon > 0$ and we have
\[
P(e, v) = n \int_{\bS_1(0)} m v d\mu^\sigma,
\]
where $m$ is $\tr_\delta(\ebar)$ restricted to $\bS_1(0)$.
\end{proposition}
Note that we cannot go down to $\epsilon = 0$ in this proposition as we do not have $e \not\in L^1_1(\bH^n \setminus K', \bR$, but merely $e \in L^1_{1-\epsilon}(\bH^n\setminus K', \bR)$. Before proving Proposition \ref{propMassAspect2}, we state and prove two preliminary lemmas:
\begin{lemma}\label{lmStrongerDecay}
Under the assumptions of Proposition \ref{propMassAspect2}, we have $\scal^g + n(n-1) = O(\rho^{n+1})$. In particular, $\scal^g+n(n-1) \in L_1^1(M, \bR)$.
\end{lemma}

\begin{proof}
Combining Equations \eqref{eqScalarVar1} and \eqref{eqScalarVar2}, we have
\begin{equation}\label{eqScalarVar3}
\scal^g + n(n-1) = (n-1) \tr(e) + \divg(\divg(e)) - \Delta \tr(e) + \hot
\end{equation}
We now analyze the asymptotic behavior of the leading terms in \eqref{eqScalarVar3}. In the course of the proof, all quantities appearing with an overline are defined with respect to the Euclidean metric $\delta = \rho^2 b$.
First note that $e = \rho^{n-2} \ebar$, hence 
\[
\tr(e) = \rho^n \tr_\delta(\ebar).
\]
Using the conformal transformation law for the Laplacian (see e.g. \cite[Equation 1.1.15] {GicquaudThesis}) and the fact that $|d\rho|_\delta = 1 + O(\rho)$, we thereby obtain
\begin{align*}
\Delta \tr(e)
&= \rho^2 \left[\Deltabar (\rho^n \tr_\delta(\ebar)) - (n-2) \left\<\frac{d\rho}{\rho}, d (\rho^n \tr_\delta(\ebar))\right\>_\delta \right]\\
&= \rho^2 (\Deltabar \rho^n) \tr_\delta(\ebar)- (n-2)\rho^2 \left\<\frac{d\rho}{\rho}, d \rho^n \right\>_\delta \tr_\delta(\ebar) + O(\rho^{n+1})\\
&= \rho^2 \left(n \rho^{n-1} \Deltabar \rho + n(n-1) \rho^{n-2} |d\rho|^2_\delta\right) \tr_\delta(\ebar)\\
&\qquad \qquad - n(n-2) \rho^n |d\rho|^2_\delta \tr_\delta(\ebar) + O(\rho^{n+1})\\
&= n(n-1) \rho^n \tr_\delta(\ebar) - n(n-2) \rho^n \tr_\delta(\ebar) + O(\rho^{n+1})\\
&= n \rho^n \tr_\delta(\ebar) + O(\rho^{n+1}).
\end{align*}
The asymptotics of $\divg(\divg(e))$ is slightly harder to compute as the conformal transformation law for the double divergence is more intricate. Instead, we choose an arbitrary compactly supported function $f$ and integrate by parts to compute 
\begin{align*}
&\int_{\bH^n} f \divg(\divg(e)) d\mu^b\\
&\qquad= \int_{\bH^n} \left\<\hess f, e\right\> d\mu^b\\
&\qquad = \int_{\bH^n} \left\<\hess f, e\right\>_\delta \rho^{4-n} d\mu^\delta\\
&\qquad = \int_{\bH^n} \left\<\hessbar f + \frac{d\rho}{\rho} \otimes df + df \otimes \frac{d\rho}{\rho} - \delta\left(\frac{d\rho}{\rho}, df\right) \delta, \rho^{n-2} \ebar\right\>_\delta \rho^{4-n}d\mu^\delta,
\end{align*}
where we used the conformal transformation law for the Hessian \cite[Equation 1.1.14] {GicquaudThesis}). Note that $\ebar$ satisfies the transversality condition, hence
\[
\left\<\frac{d\rho}{\rho} \otimes df, \ebar\right\>_\delta = 0.
\]
Thus,
\begin{align*}
\int_{\bH^n} f \divg(\divg(e)) d\mu^b
&= \int_{\bH^n} \left\<\hessbar f - \delta\left(\frac{d\rho}{\rho}, df\right) \delta, \ebar\right\>_\delta \rho^2 d\mu^\delta\\
&= \int_{\bH^n} \left[\left\<\hessbar f, \rho^2 \ebar\right\>_\delta - \left\<df, \tr_\delta(\ebar) \rho d\rho\right\> \right] d\mu^\delta\\
&= \int_{\bH^n} f \left[ \overline{\divg}(\overline{\divg} (\rho^2 \ebar)) + \overline{\divg}\left(\tr_\delta(\ebar) \rho d\rho\right) \right] \rho^n  d\mu^b.
\end{align*}
As $f$ is arbitrary, we conclude that
\begin{align*}
\divg(\divg(e))
&= \rho^n \left[\overline{\divg}(\overline{\divg} (\rho^2 \ebar)) + \overline{\divg}\left(\tr_\delta(\ebar) \rho d\rho\right)\right]\\
&= \rho^n \tr_\delta(\ebar) \overline{\divg}\left(\rho d\rho\right) + O(\rho^{n+1})\\
&= \rho^n \tr_\delta(\ebar) |d\rho|^2_\delta + O(\rho^{n+1})\\
&= \rho^n \tr_\delta(\ebar) + O(\rho^{n+1}),
\end{align*}
where we used once again the transversality of $\ebar$ to pass from the first line to the second one. All in all, we see that
\[
\scal^g + n(n-1) = O(\rho^{n+1}),
\]
where we used the fact that the higher order terms are $O(\rho^{2n})$ so they contribute only at a much higher order degree. The conclusion that $\scal^g + n(n-1) \in L^1_1(M, \bR)$ follows from \cite[Lemma 3.6(b)]{LeeFredholm}.
\end{proof}

\begin{lemma}\label{lmChangeOfFunction}
Under the assumptions of Proposition \ref{propMassCH}, we have $p(e, V) = 0$ for all  $V \in C^1_1(\bH^n, \bR)$.
\end{lemma}

\begin{proof}
We note that, using two integrations by parts, we can rewrite the charge integral \eqref{eqChargeAspect} as follows:
\begin{align}
p(e, V)
& = \lim_{k \to \infty} \int_{\bH^n} \left[V(\divg(e) - d\tr(e))(-D\chibar_k) + \tr(e) dV(-D\chibar_k) - e(DV, -D\chibar_k)\right] d\mu^b\nonumber                                  \\
& = \lim_{k \to \infty} \int_{\bH^n} \left[V b(e, \hess \chibar_k) + e(DV, D\chibar_k) - \tr(e)(b(DV, D\chibar_k) + V \Delta \chibar_k)\right.\nonumber                           \\
& \qquad\qquad \left. + \tr(e) dV(-D\chibar_k) - e(DV, -D\chibar_k)\right] d\mu^b\nonumber\\
& = \lim_{k \to \infty} \int_{\bH^n} \left[V b(e, \hess \chibar_k) + 2 e(DV, D\chibar_k) - \tr(e)(2 b(DV, D\chibar_k) + V \Delta \chibar_k)\right] d\mu^b.\label{eqChargeAspect3}
\end{align}
This new formula for the charge integral is interesting as we have no derivatives of $e$ appearing so we can exploit the assumption $e\in L^1_\delta(\bH^n, S_2\bH^n) \subset L^1_{-1}(\bH^n, S_2\bH^n)$. In particular, if we choose the sequence $(\chibar_k)_k$ so that it is bounded in $C^{2, 0}_0(\bH^n, \bR)$, then, for some constant $A > 0$,  we have
\begin{align*}
\left|p(e, V)\right|
& \leq A \|\chibar_k\|_{C^2(\bH^n, \bR)} \lim_{k \to \infty} \int_{\supp(D\chibar_k)} |e| (V + |DV|) d\mu^b                     \\
& \leq A \|\chibar_k\|_{C^2(\bH^n, \bR)} \|V\|_{C^1_1(\bH^n, \bR)} \lim_{k \to \infty} \int_{\supp(D\chibar_k)} \rho |e| d\mu^b.
\end{align*}
Arguing as in the proof of Proposition \ref{propMassCH}, we get that $p(e, V) = 0$.
\end{proof}

This lemma shows that if $V_1$ and $V_2$ are such that $V_1-V_2 \in C^1_1(\bH^n, \bR)$, we have
\[
    p(e, V_1) - p(e, V_2) = p(e, V_1-V_2) = 0
\]
as long as $p(e, V_1)$ and $p(e, V_2)$ are well defined. This allows us to remove the restriction imposed by the fact that we chose to work with eigenfunctions of the Laplacian as ``test functions''.

\begin{proof}[Proof of Proposition \ref{propMassAspect2}]
Note that Lemma \ref{lmStrongerDecay} shows that metrics satisfying Wang's asymptotic conditions fall into the class of asymptotically hyperbolic manifolds described in Propositions \ref{propMassAspect} for any $\delta \in [-1, 1)$. In particular, the mapping $P(e, \cdot)$ introduced in Proposition \ref{propContinuity} is well-defined and independent of the choice of the family of cutoff functions $\chibar_k$. Consequently, we can work with a specific choice of the family $\chibar_k$ that will be covenient for us. Namely, we set $\chibar_k = \chi(r - k)$ for $k$ large enough where $\chi: \bR \to [0, 1]$ is a cutoff function such that $\chi(r) = 1$ for $r \leq 0$ and $\chi(r) = 0$ for $r \geq 1$. Here, as before, $r$ is the geodesic distance to the origin of the hyperbolic space.

Next we choose an arbitrary function $v \in C^{2, \alpha}(\bS_1(0), \bR)$ for some $\alpha \in (0, 1)$, and set $V = \cD(v)$. Note that Lemma \ref{lmChangeOfFunction} shows that
\[
p(e, V) = p(e, \Vtil)
\]
for any function $\Vtil$ such that $\Vtil - V \in C^1_1(\bH, n, \bR)$, so it suffices to evaluate the limit \eqref{eqChargeAspect3} on some ``good enough'' approximation $\Vtil$ of $V$. As suggested at the beginning of Section \ref{secEstimates} (see Equation \eqref{eqApproximateEigenfunction}), we define the function $\Vtil$ in the ball model of the hyperbolic space by 
\[
\Vtil(x) = \chi_0(x) V^0(x) v\left(\frac{x}{|x|}\right),
\]
where $\chi_0$ is as in the proof of Proposition \ref{propMassCH} and $V^0 = \cosh(r)$ is the function introduced in Proposition \ref{propLapse}. In what follows, we shall denote $v(x) = v\left(\frac{x}{|x|}\right)$ as a shorthand. Note that $\Vtil = V^0(x) v(x)$ outside some compact set. Consequently, near infinity we have
\begin{align*}
\Delta \Vtil
&= \Delta \left[V^0(x) v(x)\right]\\
&= v(x) \Delta V^0(x) + 2\<dV^0, dv\> + V^0(x) \Delta v(x)\\
&= n v(x) V^0(x) + O(\rho).
\end{align*}
In order to pass from the second line to the third line, we used the facts that $V^0$ solves $\Delta V^0 = n V^0$, that $V^0 = V^0(r)$ is a radial function while $v$ does not depend on $r$ so $\<dV^0, dv\> = 0$ and, finally, that by the conformal transformation law for the Laplacian we have $\Delta v(x) = \rho^2 \Deltabar v(x) = O(\rho^2)$ as $\<d\rho, dv\> = 0$.
Hence, with some more effort (see \cite[Lemma 3.1]{LeeFredholm}), we see that
\[
- \Delta (\Vtil - V) + n (\Vtil - V) \in C^{0, \alpha}_1(\bH^n, \bR).
\]
From \cite[Theorem (C)]{LeeFredholm}, we conclude that $\Vtil - V \in C^{2, \alpha}_1(\bH^n, \bR)$ showing that 
\[
P(e, v) = p(e, V) = p(e, \Vtil),
\]
as we have claimed. Thus, it only remains to compute $p(e, \Vtil)$.

Recall that in polar coordinates, the hyperbolic metric reads $b = dr^2 + \sinh(r)^2 \sigma$ where $\sigma$ is the round metric on $\bS_1(0)$. For computing $p(e,\Vtil)$, we use Formula \eqref{eqChargeAspect3}. First, we note that, outside some large compact subset,
\[
\begin{aligned}
\hess \chibar_k &= \chi_k'' dr \otimes dr + \coth(r) (b - dr \otimes dr) \chi_k',\\
\Delta \chibar_k &= \chi_k'' + (n-1) \coth(r)(r) \chi_k',
\end{aligned}
\]
where $\chi_k''$ (resp. $\chi_k'$) is a shorthand for $\chi''(r-k)$ (resp. $\chi'(r-k)$) and $\coth(r) (b - dr \otimes dr)$ is the second fundamental form of $S_r(0)$. Hence, recalling that $e$ satisfies the transversality condition $e(Dr, \cdot) = 0$, Equation \eqref{eqChargeAspect3} implies
\begin{align*}
p(e, \Vtil)
&= \lim_{k \to \infty} \int_{\bH^n} \left[\Vtil b(e, \hess \chibar_k) + 2 e(D\Vtil, D\chibar_k) - \tr(e)( 2 b(D\Vtil, D\chibar_k) + \Vtil \Delta \chibar_k)\right] d\mu^b\\
&= \lim_{k \to \infty} \int_{\bH^n} \left[\Vtil \coth(r) \chi_k' - ( 2 v \sinh(r) \chi'_k + \Vtil (\chi_k'' + (n-1) \coth(r) \chi_k'))\right] \tr(e) d\mu^b\\
&= \lim_{k \to \infty} \int_{\bH^n} \left[\frac{\cosh(r)^2}{\sinh(r)} \chi_r' - (2 \sinh(r) \chi'_k + \cosh(r) (\chi_k'' + (n-1) \coth(r) \chi_k'))\right] v \tr(e) d\mu^b\\
&= \lim_{k \to \infty} \int_{\bH^n} \left[\left(-(n-2) \frac{\cosh(r)^2}{\sinh(r)} - 2 \sinh(r)\right) \chi'_k + \cosh(r) \chi_k'' \right] v \tr(e) d\mu^b.
\end{align*}
Note that $\tr(e) = \rho^n \tr_\delta(\ebar) = \rho^n(m + O(\rho))$. In order to compute the previous limit, we now change to polar coordinates, as described above, performing the integration  with respect to $r$ first and then integrating with respect to the spherical coordinates. We note the expression for the volume form, $d\mu^b = \sinh(r)^{n-1} dr d\mu^\sigma$ and recall that $\rho = \frac{1}{\cosh(r)+1}$. This gives
\begin{align*}
& p(e, \Vtil)\\
&= \lim_{k \to \infty} \int_{\bS_1(0)} v \int_{r=k}^{r=k+1} \tr(e) \left[\left(-(n-2) \frac{\cosh(r)^2}{\sinh(r)} - 2 \sinh(r)\right) \chi'_k + \cosh(r) \chi_k'' \right] \sinh(r)^{n-1} dr d\mu^b\\
&= \lim_{k \to \infty} \int_{\bS_1(0)} v \int_{r=k}^{r=k+1} (m + O(\rho)) \left[\left(-(n-2) \frac{\cosh(r)^2}{\sinh(r)} - 2 \sinh(r)\right) \chi'_k + \cosh(r) \chi_k'' \right] \frac{\sinh(r)^{n-1}}{(\cosh(r) + 1)^n} dr d\mu^b\\
&= \int_{\bS_1(0)} m v\left[\lim_{k \to \infty} \int_{r=k}^{r=k+1} \left[\left(-(n-2) \frac{\cosh(r)^2}{\sinh(r)} - 2 \sinh(r)\right) \chi'_k + \cosh(r) \chi_k'' \right] \frac{\sinh(r)^{n-1}}{(\cosh(r) + 1)^n} dr\right] d\mu^b\\
&= \int_{\bS_1(0)} m v \int_{r=k}^{r=k+1} \left(-n\chi'_k + \chi_k'' \right) dr d\mu^b.
\end{align*}
The radial integral can be easily computed:
\begin{align*}
\int_{r=k}^{r=k+1} \left(-n\chi'_k + \chi_k'' \right) dr
= -n [\chi_k(r)]_{r=k}^{r=k+1} + [\chi'_k(r)]_{r=k}^{r=k+1}
= n.
\end{align*}
All in all, we have proven that
\[
P(e, v) = p(e, \Vtil) = n \int_{\bS_1(0)} m d\mu^\sigma.
\]
All that remains is to remove the assumption that $v \in C^{2, \alpha}(\bS_1(0), \bR)$ so that the result is valid for all $v \in C^{0, \epsilon}(\bS_1(0), \bR)$. But this follows at once from Proposition \ref{propContinuity} and the fact that $C^{2, \alpha}(\bS_1(0), \bR)$ is dense in $C^{0, \epsilon}(\bS_1(0), \bR)$.
\end{proof}

We finally make a few remarks on the regularity assumptions we used in this section regarding the scalar curvature of $g$:
\begin{remark}\label{rkRegScalar}
For simplicity, throughout this section we have assumed that $\scal + n(n-1) \in L^1_1(\bH^n, \bR)$. We would like to point out that it is possible to relax this assumption allowing for more general local behavior where the scalar curvature could be ``rough" locally while still behaving like an element of $L^1_1(\bH^n, \bR)$ globally. More precisely, inspecting the above proofs we find that all that is needed for the results of this section to hold are the following assumptions:
\begin{enumerate}
\item Integration against $\scal + n(n-1)$ should be a continuous mapping from $C^l_{-1}(\bH^n, \bR)$ to $\bR$ for some integer $l \geq 0$. (Note that the solutions to $\Delta V = n V$ with continuous boundary data belong to $C^l_{-1}(\bH^n, \bR)$ for any integer $l \geq 0$ by elliptic regularity).
\item For some well-chosen sequence of cutoff functions $\chibar_k$ with increasing support and tending to $1$ (we used such a sequence in Proposition \ref{propMassCH}), we have
\[
    \int_{\bH^n \setminus K'} \chi_0 \chibar_k V (\scal^g + n(n-1)) d\mu^b \to \int_{\bH^n \setminus K'} \chi_0 V (\scal^g + n(n-1)) d\mu^b
\]
for all functions $V \in C^l_{-1}(\bH^n, \bR)$, where $\chi_0: \bH^n \to \bR$ is a cutoff function that vanishes near $K'$ and equals $1$ outside of a larger compact subset (again, see the proof of Proposition \ref{propMassCH}).
\end{enumerate}

A general yet simple description of the natural class of distributions to which $\scal+n(n-1)$ should belong for these two assumptions to be satisfied is not easy to find. In what follows we provide  a non-technical description of some particular cases of interest. The reader is referred to \cite{Friedman} and \cite{KrantzParks} for relevant background concerning measure theory.

For instance, generalizing slightly the assumptions of Lee and LeFloch in \cite{LeeLeFloch} one could require that $\nu = \scal + n(n-1)$ be a signed measure whose weighted total variation $\displaystyle \int_{\bH^n} \rho^{-1} d|\nu|$ is finite (here $|\nu|$ denotes the total variation of $\nu$). In this case the first assumption is fulfilled since integration against $\scal+n(n-1)$ is continuous as a mapping from $C^0_{-1}(\bH^n, \bR)$ to $\bR$. The second assumption is fulfilled as a consequence of the dominated convergence theorem.

However, this is not the most general possible  assumption we can make. For example, we could more generally  assume that $\scal+n(n-1)$ can be written as the sum of a signed measure and the total divergence of a vector-valued Borel measure of finite weighted total variation, that is
\[
    \scal + n(n-1) = \nu^{(0)} + D^i \nu^{(1)}_i,
\]
with
\[
    \int_{\bH^n} \rho^{-1} d|\nu^{(0)}|< \infty \text{ and } \int_{\bH^n} \rho^{-1} d|\nu^{(1)}| < \infty.
\]
In this case for any $V \in C^l_{-1}(\bH^n, \bR)$ and $\chi_k=\chi_0 \chibar_k$, we have
\begin{align}
        & \int_{\bH^n} \chi_k V (\scal + n(n-1)) d\mu^g                                                                                                                                           \\
        & \qquad \definedas \int_{\bH^n} \chi_k V  d\nu^{(0)} - \int_{\bH^n} b\left(D (\chi_k V),  d\nu^{(1)}\right)\nonumber                                                                   \\
        & \qquad = \int_{\bH^n} \chi_k V  d\nu^{(0)} - \int_{\bH^n} b\left( \chi_k D V,  d\nu^{(1)}\right) - \int_{\bH^n} b\left( V D \chi_k,  d\nu^{(1)}\right),\label{eqDefinitionOfMeasure}
\end{align}
where $\displaystyle \int_{\bH^n} b\left(X,  d\nu^{(1)}\right)$ denotes the integration of the vector field $X$ against the measure $\nu^{(1)}$. To be more precise, we can write $\displaystyle \nu^{(1)} = \frac{d\nu^{(1)}}{d|\nu^{(1)}|} |\nu^{(1)}|$, where $\displaystyle \frac{d\nu^{(1)}}{d|\nu^{(1)}|}$ denotes the Radon-Nikodym derivative, and define
\[
    \int_{\bH^n} b\left(X,  d\nu^{(1)}\right)
    = \int_{\bH^n} b\left(X,  \frac{d\nu^{(1)}}{d|\nu^{(1)}|}\right) d|\nu^{(1)}|.
\]
We claim that the above assumptions 1 and 2 are fulfilled in this case. More precisely, the terms that do not involve derivatives of $\chi_k$ in \eqref{eqDefinitionOfMeasure} converge by the dominated convergence theorem:
\[
    \int_{\bH^n} \chi_k V d\nu^{(0)}
    + \int_{\bH^n} b\left(\chi_k D V, d\nu^{(1)}\right)
    \to_{k \to \infty}
    \int_{\bH^n} \chi_0 V  d\nu^{(0)} + \int_{\bH^n} b\left(\chi_0 D V, d\nu^{(1)}\right)
\]
as we have, for example, 
\[\displaystyle \left|b\left(\chi_k D V, \frac{d\nu^{(1)}}{d|\nu^{(1)}|}\right)\right| \leq |D^{(1)} V| \leq \rho^{-1} \|V\|_{C^{1,0}_{-1}}\] 
and
\[
    \int_{\bH^n} \rho^{-1} \|V\|_{C^{1,0}_{-1}} \,  d|\nu^{(1)}| < \infty
\]
by our assumption. By a similar reasoning, 
\[
\int_{\bH^n} b\left( V D \chi_k, d\nu^{(1)}\right) \to \int_{\bH^n} b\left( V D \chi_0, d\nu^{(1)}\right)
\]
as well, as $|D \chibar_k|$ converges to $0$ a.e.. This shows that the second assumption is satisfied. Using similar arguments it is straightforward to verify that the first one is satisfied as well.
\end{remark}

\subsection{A definition based on the Ricci tensor}\label{secRicci}
In this section we discuss how the approach of Section \ref{secADM} applies to a different definition of mass, based on the Ricci tensor and conformal Killing vector fields, yielding an alternative definition of mass aspect function whenever sufficient regularity is assumed. For the background on this definition and its relation to the definition of Chru\'sciel and Herzlich as in \eqref{eqChargeCH2}, see \cite{HerzlichRicciMass}. 

We will assume throughout that $(M,g)$ is asymptotically hyperbolic as in Definition \ref{defAH}. In addition to this, we will assume that $e \in W^{2,2}_{1/2}(\bH^n \setminus K',S_2\bH^n)$ (see Section \ref{secFunctSpaces}). The need for this extra regularity assumption will be clarified later, see in particular \eqref{eqRTerm} below. We also continue to use the conventions introduced in Section \ref{secScalarCurvature}. In particular, let $C>1$ be the constant such that $C^{-1} b \leq g \leq C b$. We remind the reader that the notation $A \lesssim B$ means that there exists a constant $\Lambda > 0$ independent of $A$ and $B$ but depending on the value of $C$ such that $A \leq \Lambda B$.

The modified Einstein tensor of an asymptotically hyperbolic manifold is defined by
\[
\Gtil^g=\ric^g - \tfrac{1}{2} \scal^g g - \tfrac{(n-1)(n-2)}{2}g 
\] 
so that we have $\Gtil^b = 0$. Temporarily ignoring issues related to low regularity of the metric $g$, the definition of the mass using the Ricci tensor is obtained by integrating the second Bianchi identity $\divg^g \Gtil^g = 0$ against conformal Killing vector fields $\fX$ that belong to the span of
\begin{equation}\label{eqCKVF}
\fX^0 = D V^0 = x^i \partial_i \quad \text{and} \quad \fX^i = D V^i = \rho \partial_i + x^i x^j \partial_j, \quad i=1, \ldots, n,
\end{equation}
where $V^\mu$,  $\mu=0,1, \ldots, n$, are as in \eqref{eqLapses}. A straightforward computation shows that $|\fX^\mu|, |D\fX^\mu|=O(\rho^{-1})$, and, as the functions $V^\mu$ are eigenfunctions of the Laplacian, it also follows that $\divg \fX^\mu= n V^\mu$ for $\mu=0,1, \ldots, n$. This fact can be used to relate the definition of mass emerging as the boundary term at infinity in the aformentioned integration by parts,  
\begin{equation}\label{eqChargeRicciSmooth}
 \ptil(e,\fX) = \lim_{r \to \infty} \int_{S_r(0)}  \Gtil^g(\fX, \nu) d\mu^b,
\end{equation}
to the definition \eqref{eqChargeCH2}, see \cite{HerzlichRicciMass}. 

Our first goal in this section will be to obtain the analogue of the formula \eqref{eqChargeRicciSmooth} using cutoff functions.  To begin with,  we let $\fX$ be a smooth vector field on $M \setminus K$ (identified with $\bH^n\setminus K'$ as described above) and we let $\chi_k$ be a family of cutoff functions supported in $\bH^n \setminus K'$, both to be specified later. In this case, we have 
\begin{equation*}
\begin{split}
0 = & \int_{\bH^n} \divg^g  \left( \chi_k\Gtil^g (\fX, \cdot)\right) d\mu^g
\\ = &  \int_{\bH^n} \chi_k \tr ^g \tr ^g (\Gtil^g \otimes \nabla \fX) \, d\mu^g +  \int_{\bH^n} \Gtil^g (\nabla \chi_k, \fX)  \, d\mu^g 
\end{split}
\end{equation*}
so that 
\begin{equation}\label{eqWeakBianchi}
\int_{\bH^n} \Gtil^g (\nabla \chi_k, \fX)  \, d\mu^g = -\int_{\bH^n} \chi_k \tr^g \tr^g(\Gtil^g \otimes \nabla \fX) \, d\mu^g.
\end{equation}

Equation \eqref{eqWeakBianchi} is the only consequence of the second Bianchi identity we shall use in the sequel. To establish it, we required the metric $g$ to be smooth but the reader can check that this formula continues to hold even if $g$ is merely in $W^{2, 2}_{loc}(\bH^n \setminus K',S_2\bH^n) \cap L^\infty(\bH^n \setminus K',S_2\bH^n)$ by considering a sequence of smooth metrics $(g_k)_k$ converging to $g$. See also \cite{BurtscherKiesslingTahvildarZadeh} for more on weak second Bianchi identity. 

We will now compute the Taylor expansions at $b$ of all quantities that are involved in \eqref{eqWeakBianchi}. In the course of the calculations, we will encounter cubic terms that are not, a priori, controlled by the fact that $g \in W^{2, 2}_{1/2}(\bH^n \setminus K', S_2\bH^n)$. The following lemma enables us to control these terms, in the view of our assumption that $g \in L^\infty(\bH^n \setminus K', S_2 \bH^n)$:

\begin{lemma}\label{lmGagliardo}
Under the assumptions $e \in W^{2, 2}_{1/2}(\bH^n \setminus K', S_2\bH^n)$ and $C^{-1} b \leq g \leq C b$, we have $e \in W^{1, 3}_{1/3}(\bH^n \setminus K', S_2\bH^n)$.
\end{lemma}

\begin{proof}
By Sobolev embedding \cite[Lemma 3.6 (c)]{LeeFredholm} we know that
\[
W^{2, 2}_{1/2} (\bH^n \setminus K', S_2\bH^n)\hookrightarrow W^{1,3}_{1/2}(\bH^n \setminus K' \subset W^{1,3}_{1/3}(\bH^n \setminus K', S_2\bH^n)
\]
provided that $n\leq 6$. However, to deal with the case $n\geq 7$ we need a different argument. We remark that the assumption that $C^{-1} b \leq g \leq C b$ in Definition \ref{defAH} implies that $e \in L^\infty(\bH^n \setminus K', S_2\bH^n)$ so we can apply the following Gagliardo-Nirenberg inequality \cite[Equation (2.2)]{Nirenberg} on hyperbolic balls $B_1(x)$ of radius $1$ centered at any point $x \in \bH^n$ (see the fifth comment after the statement of the theorem for the case of bounded domains):
\[
\|D^{(j)}e\|_{L^p(B_1(x))} \leq \Lambda \left(\|D^{(m)}e\|_{L^r(B_1(x))}^a \|e\|_{L^q(B_1(x))}^{1-a} + \|e\|_{L^p(B_1(x))}\right),
\]
where $p, q, r \in [1, \infty]$, $\displaystyle a \in \left[\frac{j}{m}, 1\right]$ are such that
\[
\frac{1}{p} = \frac{j}{n} + a \left(\frac{1}{r} - \frac{m}{n}\right) + (1-a) \frac{1}{q}
\]
and the constant $\Lambda$ is independent of $x$.
Here we choose $p = 3$, $j = 1$, $r = 2$, $m=2$ and $q = \infty$ so that $\displaystyle a = \frac{2(n-3)}{3(n-4)}$ lies in the interval $\left[\frac{j}{m}, 1\right] = \left[\frac{1}{2}, 1\right]$ as long as $n \geq 6$. So we get
\[
\|De\|_{L^3(B_1(x))} \leq \Lambda \left(\|D^2 e\|^a_{L^2(B_1(x))} \|e\|_{L^\infty(B_1(x))}^{1-a} + \|e\|_{L^3(B_1(x))}\right).
\]

We now let $\{x_i\}$ be a countable collection of points such that the sets $\{B_1(x_i)\}$ are contained in $\bH^n \setminus K'$, cover the entire set $\bH^n \setminus K''$ for some compact subset $K'' \supset K'$, and are uniformly locally finite, that is, there exists a fixed integer $N$ such that each $B_1(x_i)$ intersects with at most $N-1$ other balls $B_1(x_j)$ for $j \neq i$. We have (see e.g \cite[Lemma 2.6]{GicquaudSakovich} for a similar argument with more details provided)
\begin{align*}
\int_{\bH^n \setminus K''} \left|\rho^{-1/3} De\right|^3 d\mu^b
& \lesssim \sum_i \rho^{-1}(x_i) \int_{B_1(x_i)} \left|De\right|^3 d\mu^b\\
& \lesssim \sum_i \rho^{-1}(x_i) \left(\|D^2 e\|^{3a}_{L^2(B_1(x_i))} \|e\|_{L^\infty(B_1(x_i))}^{3(1-a)} + \|e\|_{L^3(B_1(x_i))}^3\right)\\
& \lesssim \|e\|_{L^3_{1/3}(\bH^n \setminus K')}^3 + \|e\|_{L^\infty(\bH^n \setminus K')}^{3(1-a)} \sum_i \rho^{-1}(x_i) \|D^2 e\|_{L^2(B_1(x_i))}^{3a}.
\end{align*}
As $3a \geq 2$, we have
\begin{align*}
\sum_i \rho^{-1}(x_i) \|D^2 e\|_{L^2(B_1(x_i))}^{3a}
& \leq \left(\sum_i \rho^{-1}(x_i) \|D^2 e\|_{L^2(B_1(x_i))}^2\right)^{\frac{3a}{2}}\\
& = \left(\sum_i \rho^{-1}(x_i) \int_{B_1(x_i)} |D^2 e|^2 d\mu^b\right)^{\frac{3a}{2}}\\
& \lesssim \|D^2 e\|_{L^2_{1/2}(\bH^n \setminus K')}^{3a}.
\end{align*}
All in all, we have
\begin{align*}
\|De\|_{L^3_{1/3}(\bH^n \setminus K'')}
& = \left(\int_{\bH^n \setminus K''} \left|\rho^{-1/3} De\right|^3 d\mu^b\right)^{1/3}\\
& \lesssim \|e\|_{L^3_{1/3}(\bH^n \setminus K')} + \|e\|_{L^\infty(\bH^n\setminus K')}^{(1-a)} \|D^2 e\|_{L^2_{1/2}(\bH^n \setminus K')}^a.
\end{align*}
All that is left to prove is that $e \in L^3_{1/3}(\bH^n \setminus K', S_2\bH^n)$. This comes at once from the following calculation:
\begin{align*}
\|e\|^3_{L^3_{1/3}(\bH^n \setminus K', S_2 \bH^n)}
& = \int_{\bH^n\setminus K'} \rho^{-1} |e|^3 d\mu^g\\
& \leq \|e\|_{L^\infty(\bH^n \setminus K', S_2 \bH^n)} \int_{\bH^n\setminus K'} \rho^{-1} |e|^2 d\mu^g\\
& = \|e\|_{L^\infty(\bH^n \setminus K', S_2 \bH^n)} \|e\|_{L^2_{1/2}(\bH^n \setminus K', S_2 \bH^n)}^2.
\end{align*}
Hence, $\|e\|_{L^3_{1/3}(\bH^n \setminus K', S_2 \bH^n)} \leq \|e\|_{L^\infty(\bH^n \setminus K', S_2 \bH^n)}^{1/3} \|e\|_{L^2_{1/2}(\bH^n \setminus K', S_2 \bH^n)}^{2/3} < \infty$.

At this point, note that we have proven that $De$ belongs to $L^3_{1/3}$ over $\bH^n \setminus K''$ not over $\bH^n \setminus K'$. Regularity over the relatively compact subset $K''\setminus K'$ can be obtained by the Gagliardo-Nirenberg inequality for bounded open subsets with regular boundary (in agreement with the remark following Definition \ref{defAH}).
\end{proof}

In order to obtain the expansion of the modified Einstein tensor $\Gtil^g$ we note that, computing as in the derivation of \eqref{eqScalarVar} (see also \cite[Lemma 15]{DelayFougeirol}), we obtain
\begin{equation*}
\begin{split}
\ric^g_{ij}  = & \ric_{ij} + D_k \Gamma^k_{ij} - D_i \Gamma^k_{kj} + \Gamma^l_{ij} \Gamma^k_{kl} - \Gamma^l_{kj}\Gamma^k_{il} \\
 = & \ric_{ij} + \tfrac{1}{2} D_k [(g^{kl}(D_i e_{jl} + D_j e_{il} - D_l e_{ij})] -\tfrac{1}{2}D_i(g^{kl} D_j e_{kl})+ \Gamma^l_{ij} \Gamma^k_{kl} - \Gamma^l_{kj}\Gamma^k_{il}.
\end{split}
\end{equation*}
Combining this formula with  \eqref{eqScalarVar1} and the fact that $\Gtil^b = 0$ we obtain 
\begin{align}
\Gtil^g_{ij}  = &  \ric^g_{ij} - \tfrac{1}{2}\scal^g g_{ij}- \tfrac{(n-1)(n-2)}{2} g_{ij}\nonumber  \\
 = & \tfrac{1}{2} b^{kl} (D_k D_i e_{jl} +D_k D_j e_{il} - D_k D_l e_{ij} -D_i D_j e_{kl}) \label{eqLinearizationEinsteinTensor}\\
& - \tfrac{1}{2}  \left(\divg\divg e - \Delta \tr(e) + (n-1) \tr e \right) b_{ij} + (n-1) e_{ij} +\cR_{ij},\nonumber
\end{align}
where the remainder $\cR_{ij}$ is a 2-tensor satisfying
\begin{equation*}
|\cR| \lesssim |De|^2+|e||DDe| \lesssim |e|^2 + |De|^2 + |DDe|^2.
\end{equation*} 
We may now rewrite the left hand side of \eqref{eqWeakBianchi} as  
\begin{equation*}
\int_{\bH^n} \Gtil^g (\nabla \chi_k, \fX)  \, d\mu^g = \int_{\bH^n} \Gtil^g (D \chi_k, \fX)  \, d\mu^b + \int_{\bH^n}  \cR_1(d\chi_k, \fX) \, d\mu^b,
\end{equation*}
where $\cR_1(d\chi_k, \fX)$ denotes a term that is linear in both $d\chi_k$ and $\fX$ and satisfies
\begin{equation}\label{eqRTerm1}
|\cR_1(d\chi_k, \fX)|\lesssim |d\chi_k| |\fX|(|e|^2+|De|^2+|DDe|^2).
\end{equation}
Similarly, recalling that 
\begin{equation*}
\nabla_i \fX^k = D_i \fX^k +  \Gamma^k_{ij} \fX^j 
\end{equation*} 
where $\Gamma^k_{ij} = \frac{1}{2} g^{kl}\left(D_i e_{lj} + D_j e_{il} - D_l e_{ij}\right)$ we may rewrite the right hand side of \eqref{eqWeakBianchi} as 
\begin{equation*}
\begin{split}
-\int_{\bH^n} \chi_k \tr^g \tr^g(\Gtil^g \otimes \nabla \fX) \, d\mu^g = &-\int_{\bH^n} \chi_k \tr \tr(\Gtil^g \otimes D \fX) \, d\mu^b \\ & +\int_{\bH^n} \chi_k\cR_2(\fX, D\fX) \, d\mu^b
\end{split}
\end{equation*}
where $\mathcal{R}_2(\fX, D\fX)$ 
satisfies
\begin{equation}\label{eqRTerm}
|\mathcal{R}_2(\fX, D\fX)|\lesssim ( |\fX|  +  |D\fX|)(|e|^2+|De|^2 + |De|^3 +|DDe|^2) .
\end{equation}
Note also that by the symmetry of $b$ and $\Gtil^g$ we may replace $D\fX$ by $\delta^* \fX = \frac{1}{2}\lie_{\fX} b$, 
the symmetric part of the covariant derivative, in the first term in the right hand side. 
All in all, we find from \eqref{eqWeakBianchi} that
\begin{equation}\label{eqRicciBase}
\begin{split}
\int_{\bH^n} \Gtil^g (D \chi_k, \fX)  \, d\mu^b  = & -\int_{\bH^n} \chi_k \tr \tr(\Gtil^g \otimes \delta^* \fX) \, d\mu^b \\ & +\int_{\bH^n} \left(\cR_1(d\chi_k, \fX)+ \chi_k\cR_2(\fX, D\fX)\right) \, d\mu^b,
\end{split}
\end{equation}
for $\cR_1(d \chi_k, \fX)$ and $\cR_2(\fX, D\fX)$  as described above.

Assume now that $\fX$ is in the linear space spanned by  the conformal Killing vector fields as in \eqref{eqCKVF}. In this case $\fX$ is in the kernel of the conformal Killing operator $\bL$ defined by $\bL \fX = 2 \delta^* \fX - \tfrac{2}{n} (\divg \fX) b$. As noted above, we have $\divg \fX = nV$ with $V\in \mathcal{N}$, hence
\[
\delta^* \fX = \tfrac{1}{n} (\divg \fX) b + \frac{1}{2}\bL \fX = \tfrac{1}{n} (\divg \fX) b =  V b.
\]
This enables us to rewrite the integrand in the first term in the right hand side of \eqref{eqRicciBase} as
\[
\begin{split}
\chi_k \tr \tr(\Gtil^g \otimes \delta^* \fX)  = & \chi_k  V \tr (\Gtil^g) \\  = & -\tfrac{n-2}{2}( \scal^g + n(n-1))\chi_k V -  \chi_k V f^{ij} \Gtil^g_{ij} \\ = & -\tfrac{n-2}{2}(\scal^g + n(n-1))\chi_k V + \chi_k \mathcal{R}_3(D\fX)
\end{split}
\]
where $\mathcal{R}_3(D\fX)$ is linear in $D\fX$ and satisfies 
\begin{equation}\label{eqRTerm2}
|\mathcal{R}_3(D\fX)|\lesssim  |D\fX|(|e|^2+|De|^2+|DDe|^2).
\end{equation} 
Summing up, when $\fX$ is a conformal Killing vector field as described above, with $\divg \fX = nV$,  \eqref{eqRicciBase} becomes
\begin{equation}\label{eqRicciBase1}
\begin{split}
\int_{\bH^n} \Gtil^g (D \chi_k, \fX)  \, d\mu^b  & =   \tfrac{n-2}{2} \int_{\bH^n}( \scal^g + n(n-1))\chi_k V\, d\mu^b  \\ & \qquad + \int_{\bH^n} \left(\cR_1(d\chi_k, \fX)+ \chi_k\cR_2(\fX, D\fX) \right)\, d\mu^b,
\end{split}
\end{equation}
for $\cR_1(d\chi_k, \fX)$ as in \eqref{eqRTerm1} and $\cR_2(\fX, D\fX)$  as in \eqref{eqRTerm}, where we have absorbed the terms of the form \eqref{eqRTerm2} in $\cR_2(\fX, D\fX)$.

The identity \eqref{eqRicciBase1} allows us to prove the following result. 

\begin{proposition}\label{propRicciMass} 
Assume that $g$ is asymptotically hyperbolic  in the sense of Definition \ref{defAH} and such that 
\begin{equation}\label{eqMoreFallOff}
e= \Phi_*g -b \in  W^{2,2}_{1/2} (\bH^n \setminus K', S_2\bH^n).
\end{equation}
Suppose additionally that $\scal^g+n(n-1) \in L^1_1(M, \mathbb{R})$.
Let $(\chibar_k)_k$ be a sequence of compactly
supported functions over $\bH^n$ with uniformly bounded $C^1$-norms and such that the sets $\Omega_k = \chibar_k^{-1}(1)$ form an increasing sequence of compact sets such that $\bH^n = \bigcup_k \Omega_k$. Then, for any hyperbolic conformal Killing vector field $\fX$ that belongs to the $(n+1)$-dimensional linear space spanned by $\fX^0, \fX^1,\ldots, \fX^n$ as in \eqref{eqCKVF}, the limit 
\begin{equation}\label{eqChargeRicci}
 \ptil(e,\fX) \definedas \lim_{k \to \infty} \int_{\bH^n} \Gtil^g (\fX, - D \chibar_k) d\mu^b
\end{equation}
is well defined and independent of the choice 
of the sequence $(\chibar_k)_k$.

Assuming that $g\in C^2_{\delta}(\bH^n \setminus K')$ with $\delta > n/2$ and $\scal^g + n(n-1) \in L^1_1(M, \bR)$, the limit \eqref{eqChargeRicci} coincides with the classical definition of mass using the Ricci tensor, that is
\begin{equation}\label{eqChargeRicciToRicciSmooth}
 \ptil(e,\fX) = \lim_{r \to \infty} \int_{S_r(0)}  \Gtil^g(\fX, \nu) d\mu^b.
\end{equation}

We also have
\begin{equation}\label{eqChargeRicciToADM}
\ptil(e,\fX) = - \tfrac{n-2}{2} p(e, V)
\end{equation}
where $V=\tfrac{1}{n}\divg \fX$ and $p(e,V)$ is as in \eqref{eqChargeCH2}. 
\end{proposition}

\begin{proof}
As in the proof of Proposition \ref{propMassCH}, we let $\chi_0$ be a smooth function on $\bH^n$ which vanishes near $K'$ (see Definition \ref{defAH}) and equals $1$ outside some compact set. 
Then we apply \eqref{eqRicciBase1} with $\chi_k \definedas \chi_0 \chibar_k $  where $\chibar_k$ are as in the formulation of the proposition. Applying the dominated convergence theorem as in the proof of Proposition \ref{propMassCH} we get 
\[
\int_{\bH^n}( \scal^g + n(n-1))\chi_k V\, d\mu^b \rightarrow \int_{\bH^n}( \scal^g + n(n-1))\chi_0 V\, d\mu^b. 
\]
Similarly, since our assumption \eqref{eqMoreFallOff} combined with Lemma \ref{lmGagliardo} guarantees that $\cR_2(\fX, D\fX) \in L^1(\bH^n,\bR)$  we obtain 
\[
\int_{\bH^n}\chi_k \cR_2(\fX, D\fX)\, d\mu^b \rightarrow \int_{\bH^n}\chi_0 \cR_2(\fX, D\fX)\, d\mu^b. 
\]
Next, we note that  $d\chi_k =d\chi_0+d\chibar_k $ for sufficiently large $k$ (again, see the proof of  Proposition \ref{propMassCH}). Consequently, by bilinearity of $\cR_1$ we have 
\[
\begin{split}
  \int_{\bH^n}\cR_1(d\chi_k, \fX)\, d\mu^b = &\int_{\bH^n}\cR_1(d\chi_0, \fX)\, d\mu^b + \int_{\bH^n}\cR_1(d\overline{\chi}_k, \fX)\, d\mu^b \\ & \rightarrow  \int_{\bH^n}\cR_1(d\chi_0, \fX)\, d\mu^b, 
\end{split}
\]
since
\[
\left|\int_{\bH^n}\cR_1(d\overline{\chi}_k, \fX)\, d\mu^b\right| \lesssim \|d\overline{\chi}_k\|_{L^\infty} \int_{\supp d\overline{\chi}_k} |\fX|(|e|^2+|De|^2+|DDe|^2)\, d\mu^b \rightarrow 0
\]
as we have, by assumption, $e \in W^{2, 2}_{1/2}(\bH^n \setminus K', S_2\bH^n)$ (again, a complete argument for this step is provided in the proof of Proposition \ref{propMassCH}). In summary, we find that
\begin{equation}\label{eqRicciFinal}
\begin{split}
\ptil(e, \fX) & = \int_{\bH^n} \Gtil^g (\fX, D \chi_0) d\mu^b - \tfrac{n-2}{2}\int_{\bH^n}( \scal^g + n(n-1))\chi_0 V\, d\mu^b \\ & \qquad +\int_{\bH^n} \left(\cR_1(d\chi_0, \fX) + \chi_0\cR_2(\fX, D\fX) \right)\, d\mu^b,
\end{split}
\end{equation}
where the right hand side is well defined and does not depend on $(\chibar_k)_k$ and the left hand side does not depend on $\chi_0$. We have thus proven the first claim of the proposition.

Next, we establish Formula \eqref{eqChargeRicciToRicciSmooth} in the case when $e \in C^2_\delta(\bH^n\setminus K', S_2\bH^n)$, $\delta > n/2$. The strategy is similar to the one used in the proof of Proposition \ref{propMassCH}. Assuming temporarily once again that $g$ is smooth, and repeating the argument that we used to derive Equations \eqref{eqWeakBianchi} and \eqref{eqRicciBase1}, for large enough $r>0$ we obtain
\[
\begin{split}
\int_{S_r(0)} \Gtil^g(\fX, \nu^g) d\mu^g
& = \int_{B_r(0)} \divg^g (\chi_0 \Gtil^g(\fX, \cdot)) d\mu^g\\
& = \int_{B_r(0)} \left(\Gtil^g(\fX, \nabla \chi_0) + \chi_0 \tr^g \tr^g (\Gtil^g \otimes \nabla \fX) \right)d\mu^g\\
& = \int_{B_r(0)} \left[\Gtil^g(\fX, D\chi_0) - \frac{n-2}{2}(\scal^g + n(n-1)) \chi_0 V\right] d\mu^b\\
& \qquad + \int_{B_r(0)} \left(\cR_1(d\chi_0, \fX) + \chi_0\cR_2(\fX, D\fX) \right)\, d\mu^b,
\end{split}
\]
where $\nu^g$ is the outward pointing normal vector to $S_r(0)$ with respect to the metric $g$. In the above formula, the error terms $\cR_1$ and $\cR_2$ are exactly the same as in \eqref{eqRicciFinal}. Letting $r$ tend to infinity, we conclude that 
\[
\begin{split}
\ptil(e, \fX)
&= \lim_{r \to \infty} \int_{B_r(0)} \left[\Gtil^g(\fX, D\chi_0) - \frac{n-2}{2}(\scal^g + n(n-1)) \chi_0 V\right] d\mu^b\\
& \qquad + \lim_{r \to \infty}\int_{B_r(0)} \left(\cR_1(d\chi_0, \fX) + \chi_0\cR_2(\fX, D\fX) \right)\, d\mu^b\\
&= \lim_{r \to \infty} \int_{S_r(0)} \Gtil^g(\fX, \nu^g) d\mu^g.
\end{split}
\]
All that remains to be done in order to verify  \eqref{eqChargeRicciToRicciSmooth}  is to show that we can replace $\nu^g$ and $d\mu^g$ by $\nu$ and $d\mu^b$ in the last line. As $e \in C^2_\delta(\bH^n\setminus K, S_2\bH^n)$, we have $\Gtil^g \in C^0_\delta(\bH^n\setminus K, S_2\bH^n)$. Furthermore, we have $\nu^g - \nu \in C^0_\delta(\bH^n\setminus K, T\bH^n)$ and $\displaystyle \frac{d\mu^g}{d\mu^b} - 1 \in  C^0_\delta(\bH^n\setminus K, \bR)$. As a consequence,
\[
\Gtil^g(\fX, \nu^g) \frac{d\mu^g}{d\mu^b} - \Gtil^g(\fX, \nu) \in C^0_{2\delta - 1}(\bH^n\setminus K, \bR).
\]
As $2\delta-1 > n-1$ and $\vol(S_r(0)) = O(e^{(n-1)r})$, we have
\[
\int_{S_r(0)} \left[\Gtil^g(\fX, \nu^g) \frac{d\mu^g}{d\mu^b} - \Gtil^g(\fX, \nu)\right] d\mu^b \to_{r \to \infty} 0.
\]
This concludes the proof of the fact that
\[
\ptil(e, \fX)
= \lim_{r \to \infty} \int_{S_r(0)} \Gtil^g(\fX, \nu) d\mu^b.
\]

We will now establish the equivalence of the two definitions, \eqref{eqChargeCH} and \eqref{eqChargeRicci}.
At this point, we choose a new sequence of cutoff functions $(\theta_k)_k$ that are $1$ outside a compact set containing $K'$ and $\supp d\chi_0$ and such that $\theta_k^{-1}(0)$ is increasing with respect to $k$ with $\bH^n = \bigcup_{k=0}^\infty \theta_k^{-1}(0)$. We assume further that the functions $\theta_k$ have uniformly bounded $C^2$-norm. We set
\[
g_k \definedas \theta_k g + (1-\theta_k) b = b + \theta_k e.
\]
As each $g_k$ coincides with the metric $g$ outside a compact set, we have, for any $k \geq 0$,
\[
\ptil(\theta_k e, \fX) = \ptil(e, \fX).
\]
At the same time, $g_k$ coincides with the hyperbolic metric $b$ on the support of $d\chi_0$, so Equation \eqref{eqRicciFinal} becomes
\begin{equation}\label{eqRicciFinal2}
\begin{split}
\ptil(e, \fX) & = \ptil(\theta_k e, \fX)\\
 & = - \tfrac{n-2}{2}\int_{\bH^n}( \scal^{g_k} + n(n-1))\chi_0 V\, d\mu^b + \int_{\bH^n} \chi_0\cR_2(e_k, \fX, D\fX)\, d\mu^b,
\end{split}
\end{equation}
where we have indicated that $\cR_2$ is computed with respect to the metric $g_k$, that is $e_k \definedas g_k-b$. Recalling the estimate for $\cR_2$ given in Equation \eqref{eqRTerm},
\[
|\mathcal{R}_2(e_k, \fX, D\fX)|\lesssim (|\fX| + |D\fX|)(|e_k|^2+|De_k|^2 + |De_k|^3 + |DDe_k|^2),
\]
we claim that
\[
\int_{\bH^n} \chi_0\cR_2(e_k, \fX, D\fX)\, d\mu^b \to_{k \to \infty} 0.
\]
Indeed, using the formula $e_k = \theta_k e$ in the previous estimate together with Leibniz' formula, we have
\begin{equation*}
|\mathcal{R}_2(e_k, \fX, D\fX)|\lesssim (|\fX| + |D\fX|) (|e|^2+|De|^2 + |De|^3 +|DDe|^2) \varphi_k,
\end{equation*}
where $\varphi_k$ is a sequence of uniformly bounded functions whose supports coincide with that of $\theta_k$. For instance, we have
\[
\begin{split}
|DDe_k|
& = |\theta_k DDe + D\theta_k \otimes De + r_{12} (D\theta_k \otimes De) + DD\theta_k \otimes e|\\
& \leq \theta_k |DDe| + 2|D\theta_k| |De| + |DD\theta_k| |e|\\
& \leq \max\{\theta_k, |D\theta_k|, |DD\theta_k|\} (|DDe| + 2 |De| + |e|),
\end{split}
\]
where $r_{12}$ is the operator permuting the first and second indices. Hence,
\[
|DDe_k|^2 \lesssim (\max\{\theta_k, |D\theta_k|, |DD\theta_k|\})^2 (|DDe|^2 + |De|^2 + |e|^2).
\]
Note that our assumptions on $e$ guarantee that $(|\fX| + |D\fX|) (|e|^2+|De|^2 + |De|^3 +|DDe|^2) \in L^1(\bH^n \setminus K')$. Since the functions $\varphi_k$ are uniformly bounded in $L^\infty(\bH^n \setminus K', \bR)$ and tend to zero a.e. as the supports of $\varphi_k$ coincide with those of $\theta_k$, we conclude by the dominated convergence theorem that
\[
\begin{split}
\left|\int_{\bH^n} \chi_0\cR_2(e_k, \fX, D\fX)\, d\mu^b\right|
& \lesssim \int_{\bH^n} (|\fX| + |D\fX|) (|e|^2+|De|^2 + |De|^3 +|DDe|^2) \varphi_k d\mu^b\\
& \to_{k \to \infty} 0.
\end{split}
\]
This concludes the proof of the claim. Returning to Equation \eqref{eqRicciFinal2}, we have shown that
\begin{equation}\label{eqRicciFinal3}
\ptil(e, \fX)
 = - \tfrac{n-2}{2} \lim_{k \to \infty} \int_{\bH^n}( \scal^{g_k} + n(n-1))\chi_0 V\, d\mu^b.
\end{equation}
To complete the proof of equivalence of the two definitions, we now repeat the argument that we used to derive Equation \eqref{eqCharge2} but replacing the metric $g$ by the sequence $g_k$. As before, we note that $g$ and $g_k$ coincide outside a compact set and that $g_k=b$ on the support of $D\chi_0$, which gives
\[
\begin{aligned}
p(e, V)
& = p(e_k, V)\\
& = \int_{\bH^n} \chi_0 V \left(\scal(g_k) + n(n-1)\right) d\mu^b - \int_{\bH^n} \chi_0 \cQ(V, e_k, De_k) d\mu^b,
\end{aligned}
\]
where the quadratic term $\cQ(V, e_k, De_k)$ satisfies the estimate given in Equation \eqref{eqQTerm}:
\[
|\cQ(V, e_k, De_k)| \lesssim (|V| + |dV| + |\hess V|) (|e_k|^2 + |De_k|^2).
\]
By the same argument as before, we conclude that
\[
\int_{\bH^n} \chi_0 \cQ(V, e_k, De_k) d\mu^b \to_{k \to \infty} 0.
\]
Hence,
\[
p(e, V) = \lim_{k \to \infty} \int_{\bH^n} \chi_0 V \left(\scal(g_k) + n(n-1)\right) d\mu^b.
\]
Combining with Equation \eqref{eqRicciFinal3}, we obtain the desired formula:
\[
\ptil(e, \fX) = - \frac{n-2}{2} p(e, V).
\]
\end{proof}

Similar to Proposition \ref{propMassAspect} we may allow for more general vector fields $\fX$ in \eqref{eqRicciFinal}, thereby obtaining a definition of the mass aspect function, at the cost of imposing an additional but rather mild fall-off assumption.

\begin{proposition}\label{propRicciDistribution}
Assume that $g$ is asymptotically hyperbolic in the sense of Definition \ref{defAH} and such that
\[
e= \Phi_*g -b \in (W^{2,2}_{1/2}\cap W^{2,1}_{\delta}) (\bH^n,S_2\bH^n).
\] 
for $\delta\in [-1,1]$. Suppose additionally that $\scal^g+n(n-1) \in L^1_1(M, \bR)$.
Let $(\chibar_k)_k$ be a sequence of compactly supported functions over $\bH^n$ with uniformly bounded $C^1$-norms and such that the sets $\Omega_k = \chibar_k^{-1}(1)$ form an increasing sequence of compact sets such that $\bH^n = \bigcup_k \Omega_k$.  For an arbitrary $C^{1-\delta}$ function $v$ on $\bS_1(0)$, let $V = \cD(v)$ be the unique function $V\in C^{\infty}_{-1} (\bH^n, \bR)$ satisfying $\Delta V = nV$ such that $\rho V \equiv v $ on ${\bS_1(0)}$. Then the limit 
\begin{equation}\label{eqRicciMassAspect}
 \widetilde{P}(e, v) \definedas \lim_{k \to \infty} \int_{\bH^n} \Gtil^g (DV, - D \chibar_k) d\mu^b
\end{equation}
is well defined and independent of the choice of the
sequence $(\chibar_k)_k$, and we have
\begin{equation*}
\widetilde{P}(e, v) = -\frac{n-2}{2} P(e,v).  
\end{equation*}
\end{proposition}
\begin{proof}
The assumptions of the proposition imply that 
\[
\bL (DV) = 2 \delta^* (DV) - \tfrac{2}{n} (\divg (DV)) b = 2(\hess V -  V b) = O(\rho^{-\delta}) 
\]
see Proposition \ref{propEstimateEigenfunction2}, so that
\[
\delta^* (DV)= \tfrac{1}{n} (\divg (DV)) b + \tfrac{1}{2}\bL (DV)  = V b + (\hess V -  V b).
\]
Computing the integrand in the first term in the right hand side of \eqref{eqRicciBase} for $\fX = DV$ where $V$ is as above, we thereby obtain
\[
\begin{split}
\chi_k \tr \tr(\Gtil^g \otimes \delta^* (DV)) = & \chi_k \tr \tr(\Gtil^g \otimes (Vb + (\hess V -  V b))) \\ = &   \chi_k V \tr \Gtil^g +  \chi_k \tr \tr(\Gtil^g \otimes (\hess V -  V b)),
\end{split}
\]
where $\chi_k \definedas \chi_0 \chibar_k$ is as in the proof of Proposition \ref{propMassCH}.
We conclude that the counterpart of Equation \eqref{eqRicciBase1} in the case when $\fX= DV$ is 
\begin{equation*}
\begin{split}
\int_{\bH^n} \Gtil^g (D \chi_k, DV)  \, d\mu^b  & =   \tfrac{n-2}{2} \int_{\bH^n}( \scal^g + n(n-1))\chi_k V\, d\mu^b  \\ & \qquad + \int_{\bH^n} \left(\cR_1(d\chi_k, \fX)+ \chi_k\cR_2(\fX, D\fX) \right)\, d\mu^b \\ &  \qquad - \int_{\bH^n}\chi_k \tr \tr(\Gtil^g \otimes (\hess V -  V b))\, d\mu^b
\end{split}
\end{equation*}
for $\cR_1(d\chi_k, \fX)$ and $\cR_2(\fX, D\fX)$ as in \eqref{eqRTerm1} and \eqref{eqRTerm}. In the view of the assumptions made we have $\hess V -  V b \in L^\infty_{-\delta}\bH^n \setminus K'', S_2 \bH^n)$ and $\Gtil^g \in L^1_{\delta}(\bH^n \setminus K'', S_2\bH^n)$ so it follows that
\[
\tr \tr(\Gtil^g \otimes (\hess V -  V b))) \in L^1(\bH^n,\mathbb{R}).
\]
The proof of Proposition \ref{propRicciMass} can now be straightforwardly adjusted to show that the limit \eqref{eqRicciMassAspect} is well defined and independent of the choice of $\chi_0$ and of the
sequence $(\chibar_k)_k$. 

The only point that remains to be clarified is the equivalence of the definition of the mass aspect function using Ricci tensor \eqref{eqRicciMassAspect} and the ADM style definition as in Proposition \ref{propMassAspect2}.  
Repeating the proof of \eqref{eqChargeRicciToADM} from Proposition \ref{propRicciMass} but with $V$ that only satisfies $\Delta V=n V$ and $\hess V -  V b \in L^\infty_{-\delta}(\bH^n \setminus K'', S_2 \bH^n)$, one encounters two terms that were not present in the case when $\hess V -  V b=0$. The first one comes from the ADM style definition of the mass: 
\[
\int_M \left\<\hess V - Vb, \theta_k e\right\> d\mu^b.
\]
In the view of our assumptions, it will converge to zero by the dominated convergence theorem as $k\to \infty$. The second term is from the Ricci tensor definition:
\[
\int_{\bH^n}\chi_k \tr \tr(\Gtil^{g_k} \otimes (\hess V -  V b))\, d\mu^b.
\]
To show that it converges to zero when $k\to \infty$, we use the Taylor formula  \eqref{eqLinearizationEinsteinTensor} for  $\Gtil^{g_k}$. The contribution of the quadratic remainder term to the above integral will converge to zero as $k\to \infty$, by the same type of argument as in the proof of \eqref{eqChargeRicciToADM}. As for the linearization  part of $\Gtil^{g_k}$, its  contribution to the limit will also be zero which can be seen from our assumption $e\in W^{2,1}_{\delta} (\bH^n, S_2\bH^n)$ combined with $\hess V -  V b \in L^\infty_{-\delta}(\bH^n \setminus K'', S_2 \bH^n)$. For example, in the case when $e=e_k$, the terms in the second line of \eqref{eqLinearizationEinsteinTensor} are of the form $\Gtil^{g_k} = \theta_k D^{(2)} e + D\theta_k \star De + D^{(2)} \theta_k \star e$, where $D^{(2)}e$ denotes some contraction of the second order derivatives of $e$ with the metric $b$ and $D\theta_k \star De$ (resp. $D^{(2)} \theta_k \star e$) is some contraction of the tensor $D\theta_k \otimes De$ (resp. $D^{(2)} \theta_k \star e$) with respect to $b$.
When contracted with $\hess V - V b$ all these terms are integrable and bounded from above independently of $\theta_k$. As $\theta_k$, $D\theta_k$ and $D^{(2)} \theta_k$ all tend to zero a.e., we conclude that the respective integrals will converge to zero by dominated convergence theorem. 
\end{proof}

\begin{remarks}
\begin{enumerate}
\item The assumption 
$e \in W^{2,1}_\delta(\bH^n\setminus K'', S_2\bH^n)$ can be obtained by assuming that $e \in W^{2, 2}_\alpha(\bH^n\setminus K'', S_2\bH^n)$ with
\[
\alpha + \frac{n-1}{2} > \delta + (n-1)
\]
as it follows from \cite[Lemma 3.6 (b)]{LeeFredholm}. We have chosen to keep the assumptions on $e$ in the statement of the proposition as it is slightly more general and to highlight the link with the hypotheses made in Proposition \ref{propMassAspect}.
\item
Assume that $g$ has Wang's asymptotics, i.e. $e=\rho^{n-2}\overline{e}$, where $\overline{e}$ is a smooth 2-tensor defined in a neighborhood of infinity of $\bH^n$. For $v$ and $V = \cD(v)$ as in Proposition \ref{propRicciDistribution}, in the view of \eqref{eqChargeRicciToADM} and Proposition \ref{propMassAspect2} we obtain
\[
\widetilde{P}(e, v)  = -\tfrac{n-2}{2} P(e,v) = -\tfrac{n-2}{2} P(e,V) = - \tfrac{n(n-2)}{2} \int_{\bS_1(0)} m \,v \, d\mu^\sigma,
\]
where $m=\tr \overline{e}_{|_{\bS_1(0)}}$.
\end{enumerate}
\end{remarks}

\section{Covariance under coordinates change}\label{secCoordinates}
\subsection{A Fredholm theorem}\label{secFredholm}

The aim of this section is to prove the following technical result that will be needed in the sequel:

\begin{proposition}\label{propFredholm}
Assume that $e \in W^{1, p}_\tau(\bH^n \setminus K', S_2\bH^n)$ for some $p > n$ and $\tau > 1$ such that $\tau + \frac{n-1}{p} \geq 2$. Then for any constants $q \in (1, p]$ and $\tau'$ such that
\begin{equation}\label{eqFredholmRange}
    \left|\tau' + \frac{n-1}{q} - \frac{n-1}{2}\right| < \frac{n+1}{2},
\end{equation}
the operator
\[
    P: W^{2, q}_{\tau'}(M, \bR) \to L^q_{\tau'}(M, \bR)
\]
defined by $P: u \mapsto -\Delta^g u + nu$ is an isomorphism. 
\end{proposition}

This proposition is not a direct consequence of \cite[Theorem C]{LeeFredholm} because of the low regularity assumptions that we are working with. It could nevertheless be derived using the results of \cite{AllenLeeMaxwell}. But most of the technicality of this paper is not needed here so we provide a self-contained proof: due to the fact that the asymptotic geometry is that of the hyperbolic space, we only need to work with a single M\"obius chart. So the argument is less delicate than the ones in \cite{AllenLeeMaxwell,LeeFredholm}. As we will recurrently need to address asymptotic estimates, we choose once and for all a cutoff function $\chi_0: \bH^n \to \bR$ such that $\chi_0(x) = 0$ for $x$ in a neighborhood of $K'$ and $\chi_0(x) \equiv 1$ outside some compact set.

We let  $\check{P}: W^{2, p}_{\tau'}(\bH^n, \bR) \to L^p_{\tau'}(\bH^n, \bR)$ denote the analog of the operator $P$ for the hyperbolic space: $\check{P}(u) = -\Delta u + n u$. Then, provided that the assumptions of the proposition for $\tau'$ and $q$ are fulfilled, $\check{P}$ is an isomorphism (see e.g. \cite[Section 7]{LeeFredholm}). We first perform a formal calculation. For any function $u \in W^{2, q}_{\tau'}(M, \bR)$, we write
\[
    \chi_0 \check{P}^{-1} P (\chi_0 u) = \chi_0^2 u + \chi_0 \check{P}^{-1} (P - \check{P})  \chi_0 u,
\]
so that
\begin{equation}\label{eqDiscrepancy}
\begin{aligned}
u & = (1 - \chi_0^2) u + \chi_0^2 u\\
  & = (1 - \chi_0^2) u + \chi_0 \check{P}^{-1} P (\chi_0 u) - \chi_0 \check{P}^{-1} (P - \check{P})  \chi_0 u,
\end{aligned}
\end{equation}
where we purposefully do not mention the asymptotic diffeomorphism $\Phi$ to make the notation less cluttered. The role of the cutoff function $\chi_0$ is to restrict the support of the functions to $M \setminus K \simeq \bH^n \setminus K'$. The difference $(P - \check{P})(\chi_0 u)$ can be computed (cf. the proof of Lemma \ref{lmHarmonic}) so we obtain
\begin{equation}\label{eqDiffP}
    (P - \check{P}) (\chi_0 u)
    = (b^{ij} - g^{ij}) \hess_{ij}^b (\chi_0 u) + g^{ij} \Gamma_{ij}^k D_k (\chi_0 u).
\end{equation}
Since $g-b \in L^\infty_\tau(M \setminus K, S_2M) \subset L^\infty_1(M \setminus K, S_2M)$, we have
\[
    \left\|(g^{ij} - b^{ij}) \hess_{ij} (\chi_0 u)\right\|_{L^q_{\tau' + 1}} \lesssim \|u\|_{W^{2, q}_{\tau'}}.
\]
The second term in \eqref{eqDiffP} is slightly trickier to deal with. We write
\[
    \left\|g^{ij} \Gamma_{ij}^k D_k (\chi_0 u)\right\|_{L^q_{\tau'+1}}
    \lesssim \|De\|_{L^p_{\tau}} \|u\|_{W^{1, r}_{\tau'+1-\tau}},
\]
with $\frac{1}{r} = \frac{1}{q} - \frac{1}{p}$. From the Sobolev embedding theorem \cite[Lemma 3.6 (c)]{LeeFredholm}, we deduce
\[
    \left\|g^{ij} \Gamma_{ij}^k D_k (\chi_0 u)\right\|_{L^q_{\tau'+1}}
    \lesssim \|De\|_{L^p_{\tau}} \|u\|_{W^{2, s}_{\tau'+1-\tau}},
\]
where $s$ is such that $\frac{1}{r} = \frac{1}{s} - \frac{1}{n}$. Finally, note that the assumption
$\tau + \frac{n-1}{p} \geq 2$ combined with the above relations for $r,p,q$ and $s$ implies that
\[
    \tau' + 1 - \tau + \frac{n-1}{s} < \tau' + \frac{n-1}{q},
\]
so from \cite[Lemma 3.6b]{LeeFredholm}, we have
\[
    \|u\|_{W^{2, s}_{\tau'+1-\tau}} \lesssim \|u\|_{W^{2, q}_{\tau'}}.
\]
As a consequence, the difference $(P - \check{P}) (\chi_0 u)$ in \eqref{eqDiffP} can be estimated as follows
\begin{equation}\label{eqDiscrepancy2}
    \left\|(P - \check{P}) (\chi_0 u)\right\|_{L^q_{\tau'+1}} \lesssim \|u\|_{W^{2, q}_{\tau'}}.
\end{equation}

We will now explain how to deduce  Proposition \ref{propFredholm} from \eqref{eqDiscrepancy} and  \eqref{eqDiscrepancy2}. Note that $P$ is a continuous linear map, hence, by the open mapping theorem, it suffices to prove that $P$ is bijective.

\begin{claim2}\label{cl1p}
Under the assumptions of Proposition \ref{propFredholm}, the operator $P$ is injective.
\end{claim2}

\begin{proof}
Let $u \in W^{2, q}_{\tau'}(M, \bR)$ be such that $Pu \equiv 0$. It follows from local elliptic regularity that $u \in W^{2, q}_{loc}(M, \bR) \subset L^\infty_{loc}(M, \bR)$. Now note that the terms $(1 - \chi_0^2) u$ and $P (\chi_0 u)$ in $\chi_0 \check{P}^{-1} P (\chi_0 u)$ in the right hand side of \eqref{eqDiscrepancy} have compact support. Consequently, the first two terms in the right hand side of \eqref{eqDiscrepancy} belong to $W^{2, q}_{\tau''}(M, \bR)$ for any $\tau''$, in particular for any $\tau''$ such that
\begin{equation}\label{eqInterval}
    \left|\tau'' + \frac{n-1}{q} - \frac{n-1}{2}\right| < \frac{n+1}{2}.
\end{equation}
If we additionally  assume that $\tau'' \leq \tau'+1$, the last term in the right hand side of \eqref{eqDiscrepancy} also belongs to $W^{2, q}_{\tau''}(M, \bR)$. So we conclude that $u \in W^{2, q}_{\tau''}(M, \bR)$. Repeating the above argument a finite number of times, we get that $u \in W^{2, q}_{\tau''}(M, \bR)$ for any $\tau''$ satisfying \eqref{eqInterval}. In particular, regardless of the value of $q \in (1, p)$, we can select a positive $\tau''$ satisfying \eqref{eqInterval}. Elliptic regularity on geodesic balls for the hyperbolic metric then implies $u \in L^\infty_{\tau''}(M, \bR)$, i.e. $|u|$ decays towards zero at infinity. From the maximum principle \cite[Theorem 8.1]{GilbargTrudinger}, we conclude that $u \equiv 0$, i.e. $P$ is injective.
\end{proof}

\begin{claim2}\label{cl2p}
The range of $P$ is closed.
\end{claim2}

\begin{proof}
We remark that the estimate
\[
\left\|P u - \check{P} u\right\|_{L^q_{\tau'}} \lesssim \|u\|_{W^{2, q}_{\tau'-1}},
\]
can be derived along the lines of \eqref{eqDiscrepancy2}, even though the condition \eqref{eqFredholmRange} is not necessarily satisfied with $\tau'-1$ replacing $\tau'$, as the proof does not require us to use the invertibility of $\check{P}$. Hence,
Equation \eqref{eqDiscrepancy} implies
\begin{equation*}
    \|\chi_0^2 u \|_{W^{2, q}_{\tau'}} \lesssim \|P(\chi_0 u)\|_{L^q_{\tau'}} + \|u\|_{W^{2, q}_{\tau'-1}}.
\end{equation*}
Applying interior elliptic regularity, we can promote this inequality to the following one:
\begin{equation}\label{eqInitialEllipticRegularity}
    \|u \|_{W^{2, q}_{\tau'}} \lesssim \|P u\|_{L^q_{\tau'}} + \|u\|_{L^q_{\tau'-1}}.
\end{equation}
We now prove that this estimate can be improved to the following one:
\begin{equation}\label{eqImprovedRegularity}
    \|u\|_{W^{2, q}_{\tau'}} \lesssim \|P u\|_{L^q_{\tau'}}.
\end{equation}
Assume, by contradiction that, for any integer $k$, there exists a non-zero function $u_k \in W^{2, q}_{\tau'}(M, \bR)$ such that
\begin{equation}\label{eqUpperBound}
    \|u_k\|_{W^{2, q}_{\tau'}} \geq k \|P u_k\|_{L^q_{\tau'}}.
\end{equation}
Combining this with \eqref{eqInitialEllipticRegularity} for $u_k$ we get that for $k$ large enough we have
\[
    \|u_k\|_{W^{2, q}_{\tau'}} \lesssim \|u_k\|_{L^q_{\tau'-1}}.
\]
Without loss of generality, we can rescale the functions $u_k$ so that $\|u_k\|_{L^q_{\tau'-1}} = 1$ making the sequence of functions $(u_k)_k$ bounded in $W^{2, q}_{\tau'}(M, \bR)$. The embedding $W^{2, q}_{\tau'}(M, \bR) \hookrightarrow L^q_{\tau'-1}(M, \bR)$ being compact, we can assume further that the sequence $(u_k)_k$ converges weakly in $W^{2, q}_{\tau'}(M, \bR)$ and strongly in $L^q_{\tau'-1}(M, \bR)$ to a function $u_\infty \in W^{2, q}_{\tau'}(M, \bR)$. Due to the strong convergence in $L^q_{\tau'-1}(M, \bR)$, we have $\|u_\infty\|_{L^q_{\tau'-1}} = 1$. On the other hand, Equation \eqref{eqUpperBound} shows that, for any $v \in C^\infty_c(M, \bR)$, we have
\begin{align*}
\int_M v (P u_\infty) d\mu^g
& = \int_M (P v) u_\infty d\mu^g\\
& = \lim_{k \to \infty} \int_M (P v) u_k d\mu^g\\
& = \lim_{k \to \infty} \int_M v (P u_k) d\mu^g.
\end{align*}
As $\|P u_k\|_{L^q_{\tau'}} \lesssim 1/k$, we have that 
\[
    \lim_{k \to \infty} \int_M v (P u_k) d\mu^g = 0,
\]
from which we deduce that $P u_\infty \equiv 0$ a.e.. From the fact that $P$ is injective, we conclude that $u_\infty \equiv 0$. This yields a contradiction with the fact that $\|u_\infty\|_{L^q_{\tau'-1}} = 1$. Thus we have proven that Estimate \eqref{eqImprovedRegularity} holds for any $u \in W^{2, q}_{\tau'}(M, \bR)$.

Assume now that $f_i = P u_i$ is a sequence of elements in the range of $P$ converging to $f \in L^q_{\tau'}(M, \bR)$. The sequence $(f_i)_i$ is Cauchy in $L^q_{\tau'}(M, \bR)$ so the estimate \eqref{eqImprovedRegularity} applied to $u_i - u_j$ shows that $(u_i)_i$ is also Cauchy in $W^{2, q}_{\tau'}(M, \bR)$. As $W^{2, q}_{\tau'}(M, \bR)$ is a Banach space, this implies that the sequence $(u_i)_i$ converges to some $u \in W^{2, q}_{\tau'}(M, \bR)$. By continuity of $P$, we have $Pu = f$ showing that $f$ belongs to the range of $P$.
\end{proof}

We now establish the claim that concludes the proof of Proposition \ref{propFredholm}.

\begin{claim2}\label{cl3p}
The operator $P$ is surjective.
\end{claim2}

\begin{proof}
From the previous claim, we know that $P$ has a closed range. Assuming that the range of $P$ is not the whole of $L^q_{\tau'}(M, \bR)$, there would exist a non-zero element $v$ in the dual of $L^q_{\tau'}(M, \bR)$, i.e. $v \in L^r_{-\tau'}(M, \bR)$ where $r$ satisfies $\frac{1}{r} + \frac{1}{q} = 1$, such that, for all $u \in W^{2, q}_{\tau'}(M, \bR)$, we have
\[
0 = \int_M v \left(-\Delta^g u + n u\right) d\mu^g.
\]
In particular, for all smooth compactly supported functions $v$, we have
\[
    0 = \int_M v (-\Delta^g u + n u) d\mu^g = \int_M u (-\Delta^g v + n v) d\mu^g.
\]
This means that $v$ is a solution to $Pv = 0$ in the sense of distributions. From elliptic regularity, we then know that $v \in W^{2, r}_{-\tau'}(M, \bR)$. However, note that 
\[
    \left|-\tau' + \frac{n-1}{r} - \frac{n-1}{2}\right| = \left|\tau' + \frac{n-1}{q} - \frac{n-1}{2}\right| < \frac{n+1}{2}.
\]
So Claim \ref{cl1p} shows that $v \equiv 0$. This proves that $P$ is onto.
\end{proof}

\subsection{Asymptotic rigidity for weakly regular metrics}\label{secRigidity}

The aim of this section is to prove the following theorem:

\begin{theorem}\label{thmRigidity}
Assume that $\Phi_1: M \setminus K_1 \to \bH^n \setminus K'_1$ and $\Phi_2: M \setminus K_2 \to \bH^n \setminus K'_2$ are two charts at infinity such that
\[
    e_1 \coloneqq \Phi_{1*} g - b \quad\text{and}\quad  e_2 \coloneqq \Phi_{2*} g - b
\]
with $e_1 \in W^{1, p}_\tau(\bH^n \setminus K'_1, S_2 \bH^n)$ and  $e_2 \in W^{1, p}_\tau(\bH^n \setminus K'_2, S_2 \bH^n)$, where $p > n$, and $\tau > 3/2$ satisfies
\[
    2 \leq \tau + \frac{n-1}{p} < n.
\]
Then $\Phi_2 \circ \Phi_1^{-1}: \Phi_1(M \setminus (K_1 \cup \Phi_1^{-1}(K'_2))) \to \Phi_2(M \setminus (K_2 \cup \Phi_2^{-1}(K'_1)))$ can be written as $B \circ \Phi_0$, where $B \in O_\uparrow(n, 1)$ is an isometry of the hyperbolic space and $\Phi_0(x) = \exp_b(\zeta(x))$ is a diffeomorphism ``asymptotic to the identity'', where $\exp_b$ denotes the exponential map with respect to the hyperbolic metric and $\zeta \in W^{2, p}_\tau(\Phi_1(M \setminus (K_1 \cup \Phi_1^{-1}(K'_2))), T \bH^n)$.
\end{theorem}

The argument is similar to the one used by Bartnik to prove the asymptotic rigidity of asymptotically Euclidean manifolds, see \cite{BartnikMass}, and it is a strengthening of an earlier result by Herzlich \cite[Proposition 2.4]{HerzlichMassFormulae} for which we provide a detailed proof. The restriction $\displaystyle \tau + \frac{n-1}{p} \geq 2$ will be explained after Proposition~\ref{propApproximateLapse} as it has an intrinsic interpretation in this approach. The assumption $\tau > 3/2$ is a consequence of the fact that we are dealing with weighted Sobolev spaces that are not so well suited for non-linear elliptic PDEs over asymptotically hyperbolic manifolds. This was already noted in~\cite{GicquaudStability} and was the main motivation behind introducing weighted local Sobolev spaces in~\cite{GicquaudSakovich}. However, we have chosen to stay with the standard Sobolev spaces as they are the function spaces that are typically used in the definition of the mass.

The asymptotic rigidity can be proven in the context of a more general conformal infinity, see~\cite{ChruscielNagy, ChruscielHerzlich}, and also~\cite{CortierDahlGicquaud}, but at the cost of imposing stronger regularity conditions on the metric. In a future work, we intend to adapt the strategy developed below to this more general context.

We start the proof of Theorem~\ref{thmRigidity} with a couple of lemmas. In the first one we construct an analogue of Bartnik's harmonic coordinates on asymptotically Euclidean manifolds that is better suited to the asymptotically hyperbolic setting.

In what follows, we set $X^\mu \definedas V^\mu \circ \Phi_1$ and $Y^\mu \definedas V^\mu \circ \Phi_2$. The first thing we want to do is to transform these coordinate functions to ones that are defined on the whole of $M$ and satisfy the equation $\Delta^g \cV = n \cV$. We will state everything in terms of the diffeomorphism $\Phi_1$ keeping in mind that the same results will be true for $\Phi_2$.

\begin{lemma}\label{lmHarmonic}
Given $V \in \cN^b$, there exists a unique function $\cV$ on $M$ such that
\[
\Delta^g \cV = n \cV, \quad \cV - V\circ \Phi_1 \in W^{2, p}_{\tau-1}(M, \bR).
\]
Furthermore, we have $\hess^g \cV - \cV g \in L^p_{\tau-1}(M \setminus K_1, S_2 M)$.
\end{lemma}

Following the notations of the lemma, a cursive capital letter (e.g. $\cX^\mu$)
will denote the eigenfunction of the Laplacian of $g$ asymptotic to the corresponding straight capital letter ($X^\mu$ in this example).

\begin{proof}
We choose a smooth cutoff function $\chibar$ supported on $M \setminus K_1$ such that $\chibar \equiv 1$ outside a compact subset of $M$. This allows us to ``extend'' $V \circ \Phi_1$ to a smooth function $\Vtil: M\to \mathbb{R}$ defined by
\[
\Vtil(x) = \left\{
\begin{aligned}
\chibar(x) V \circ\Phi_1(x) & \quad\text{if } x \in M \setminus K_1, \\
0                          & \quad\text{if } x \in K_1.
\end{aligned}
\right.
\]
We have
\[
-\Delta^g \Vtil + n\Vtil \in W^{1, p}_{\tau-1}(M, \bR).
\]
Further, due to our assumptions on $\tau$, we have
\[
\left|(\tau-1) + \frac{n-1}{p} - \frac{n-1}{2}\right| < \frac{n+1}{2}.
\]
Hence, Proposition \ref{propFredholm} applies to provide a unique function $w \in W^{2, p}_{\tau-1}(M, \bR)$ such that
\[
-\Delta^g w + n w = \Delta^g \Vtil - n\Vtil.
\]
Setting $\cV \definedas \Vtil + w$, we see that $\cV$ is an eigenfunction of the Laplacian. The estimate for the Hessian of $\cV$ follows by writing
\begin{align*}
\hess^g \cV - \cV g
& = \left(\hess^b \Vtil - \Vtil b\right) + \left(\hess^g \Vtil - \hess^b \Vtil\right) \\
& \qquad + \Vtil (b - g) + (\hess^g w - w g)
\end{align*}
and estimating each of the parentheses; the first term is zero outside a compact set since $\Vtil(x) = \chibar(x) (V \circ \Phi_1)(x)$. Uniqueness follows by noting that if $\cV$ and $\cV'$ are two eigenfunctions of the Laplacian such that $\cV - V \circ\Phi_1 \in W^{2, p}_{\tau-1}(M, \bR)$ and  $\cV' - V \circ\Phi_1 \in W^{2, p}_{\tau-1}(M, \bR)$, then 
$\cV - \cV' \in W^{2, p}_{\tau-1}(M, \bR)$ satisfies $-\Delta^g (\cV - \cV') + n(\cV - \cV') = 0$. By Proposition \ref{propFredholm}, we conclude that $\cV - \cV' \equiv 0$.
\end{proof}

We now prove that the span of the $\cX^\mu$ (or equivalently the $\cY^\mu$) is intrinsic to the manifold $(M, g)$: 

\begin{proposition}\label{propApproximateLapse} Given $g$ such that $(\Phi_1)_* g - b \in W^{1, p}_\tau(\bH^n \setminus K'_1, S_2 \bH^n)$ with $p > n$ and $\tau > 0$ satisfying
\[
    2 \leq \tau + \frac{n-1}{p} < n,
\]
the set $\cN^g$ defined as
\[
    \cN^g \definedas \left\{\cV \in W^{2, p}_{loc}(M, \bR)~\middle\vert~\Delta^g \cV = n \cV~\text{and}~ \hess^g \cV - \cV g \in L^p_{\tau-1}(M, S_2M) \right\}
\]
is $(n+1)$-dimensional and spanned by the functions $\cX^\mu$:
\[
    \cN^g = \mathrm{span}\{\cX^0, \cX^1, \ldots, \cX^n\}.
\]
\end{proposition}

\begin{remark}
We pause here and comment on the assumption $\displaystyle \tau + \frac{n-1}{p} \geq 2$. Lemma~\ref{lmHarmonic} can be generalized to show that any arbitrary $V$ satisfying $\Delta^b V = n V$ such that $\rho V$ extends to a $C^2$ function to $\bS_1(0)$ can be perturbed to a function $\cV$ such that $\Delta^g \cV = n \cV$ with the same asymptotic. The tensor
\begin{equation}\label{eqDefT}
    \cT \definedas \hess^g \cV - \cV g
\end{equation}
has contributions coming from the following two facts:
\begin{itemize}
    \item $V$ does not (in general) satisfy $\hess^b V = V b$,
    \item $g$ is not the hyperbolic metric.
\end{itemize}
To distinguish functions $\cV$ asymptotic to the lapse functions (see the discussion after Proposition \ref{propLapse}), our assumptions on the metric must be strong enough to ensure that the second type of correction terms does not dominate the first so we can still single out functions asymptotic to hyperbolic lapse functions. Recall that the optimal estimate of $\hess^b V - V b$ for an arbitrary eigenfunction $V$ is given by \eqref{eqAsymptotic2ij}:
\[
|\hess^b V - V b| = O(\rho).
\]
In particular, we do not have $\hess^b V - V b \in W^{1, p}_\tau(\bH^n, \bR)$ for any $\tau$ such that $\displaystyle \tau + \frac{n-1}{p} \geq 2$.
\end{remark}

Proposition \ref{propApproximateLapse} will allow us to identify the leading term (i.e. the hyperbolic isometry) in the composition $\Phi_2 \circ \Phi_1^{-1}$. It is by far the most challenging and lengthy intermediate result in proving Theorem \ref{thmRigidity}. Due to the low regularity we are working with, we must employ arguments based on integrals to achieve our goal. To make the proof of Proposition \ref{propApproximateLapse} easier to read, we decompose it in a series of claims.  Note that we will work with $\Phi_1$ as a fixed chart at infinity so we will temporarily write $g$ instead of $(\Phi_1)_* g$ to simplify the notation. 

From Lemma \ref{lmHarmonic}, we know that the set $\cN^g$ is at least $(n+1)$-dimensional as it contains the span of $\cX^0, \ldots, \cX^n$. All we have to do is to prove that any $\cV \in \cN^g$ belong to this span.

Assume that $\cV \in \cN^g$ is given. As before, we let $r$ denote the hyperbolic distance function from a given origin in $\bH^n$ and remind the reader that the hyperbolic metric in normal geodesic coordinates reads
\[
b = dr^2 + \sinh(r)^2 \sigma,
\]
where $\sigma$ is the round metric on $\bS_1(0)$. Denoting a point $x$ in $\bH^n$ as $(r, \theta)$, where $r$ is its distance from the origin and $\theta$ is a coordinate system on $\bS_1(0)$, we first introduce the following two functions:
\begin{align}
f_0(s) & \definedas \int_{s-1}^{s+1} \int_{\bS_1(0)} \cV(r, \theta)^p e^{-pr} d\theta dr,\label{eqDefF0}\\
f_1(s) & \definedas \int_{s-1}^{s+1} \int_{\bS_1(0)} |d\cV|^p_g(r, \theta) e^{-pr} d\theta dr,\nonumber
\end{align}
defined for $s \geq r_0+1$, where $r_0$ is such that $K'_1 \subset B_{r_0}(0)$, i.e. so that the image of $\Phi_1$ contains the image of $[r_0, \infty) \times \bS_1(0)$ under the exponential map with respect to the origin $0$ of $\bH^n$. Hence $f_0$ and $f_1$ are integrals over the annuli of increasing radius, $s-1 \leq r \leq s+1$, and the underlying idea is that they should capture the average behavior of $\cV$ on these annuli. In what follows the index $i = 1$ will denote the radial direction and upper case Latin indices will refer to directions tangent to $\bS_1(0)$.

Note that the functions $V^\mu$, $\mu=0,\ldots, n$, are given in hyperbolic normal geodesic coordinates by the following formulas:
\begin{equation}\label{eqLapseGeodesic}
V^0 = \cosh(r), \quad V^i = x^i \sinh(r) , \quad i=1,\ldots, n,
\end{equation}
where $x^1, \ldots, x^n$ denote the restrictions of the coordinate functions of $\bR^n$ to the unit sphere $\bS_1(0)$. Thus, we expect that $\cV$ and $|d\cV|$ are $O(e^r)$. The first claim shows that this is true in some sense on average:

\setcounter{claim}{0}
\begin{claim}\label{cl1}
Assume that $\cV \in \cN^g$, then the functions $f_0$ and $f_1$ are bounded functions of $s$.
\end{claim}

\begin{proof}
By the change of variable $r = t+s$, $t \in [-1, 1]$, we have
\[
f_0(s) = \int_{-1}^{1} \int_{\bS_1(0)} \cV(t+s, \theta)^p e^{-p(t+s)} d\theta dt.
\]
Hence,
\[
\frac{d}{ds} f_0(s)
= \int_{-1}^1 \int_{\bS_1(0)} \partial_1\left(\cV(t+s, \theta)^p e^{-p(t+s)}\right) d\theta dt
= \int_{s-1}^{s+1} \int_{\bS_1(0)} \partial_1\left(\cV(r, \theta)^p e^{-pr}\right) d\theta dr
\]
and, thus,
\begin{align*}
\frac{d}{ds} f_0(s)                                                                                                     
&= \int_{s-1}^{s+1} \int_{\bS_1(0)} \partial_1 \left[\cV(r, \theta)^p e^{-pr}\right] d\theta dr        \\
&= p \int_{s-1}^{s+1} \int_{\bS_1(0)} \left[\cV^{p-1} \partial_1\cV - \cV^p \right] e^{-pr} d\theta dr \\
&\leq p \int_{s-1}^{s+1} \int_{\bS_1(0)} \cV^{p-1} |d\cV|_b e^{-pr} d\theta dr - p f_0(s)              \\
&\leq p \left(\int_{s-1}^{s+1} \int_{\bS_1(0)} \cV^p  e^{-pr} d\theta dr\right)^{\frac{p-1}{p}}
\left(\int_{s-1}^{s+1} \int_{\bS_1(0)} |d\cV|_b^p  e^{-pr} d\theta dr\right)^{1/p} - p f_0(s).
\end{align*}
Note that we could estimate $\partial_1 \cV$ from above by $|d\cV|_b$ as $\partial_1$ is a unit vector field with respect to the metric $b$. The last line is almost
$p f_0(s)^{\frac{p-1}{p}} f_1(s)^{1/p} - p f_0(s)$, except that the norm of $d\cV$ in the second integral is with respect to $b$ and not to $g$ so we cannot immediately replace this second integral by $f_1(s)^{1/p}$. However, as $g - b \in W^{1, p}_\tau$ with $\tau > 1$, we have the (very rough) estimate $g - b = O(\rho^\tau) = O(e^{-\tau r})$. In particular, there exists a constant $c > 0$ such that, for any 1-form $\psi$, we have
\[
(1 - c e^{-\tau r}) |\psi|^2_g \leq |\psi|^2_b \leq (1 + c e^{-\tau r}) |\psi|^2_g.
\]
As a consequence, we get that, for some new constant $c$, we have
\begin{align*}
\int_{s-1}^{s+1} \int_{\bS_1(0)} |d\cV|_b^p  e^{-pr} d\theta dr
& \leq (1 + c e^{-\tau s}) \int_{s-1}^{s+1} \int_{\bS_1(0)} |d\cV|_g^p  e^{-pr} d\theta dr\\
& = (1 + c e^{-\tau s}) f_1(s).
\end{align*}
We have thereby proven the following differential inequality for $f_0$:
\begin{equation}\label{eqDiffF0}
\begin{aligned}
\frac{d}{ds} f_0(s)
& \leq p (1 + c e^{-\tau s}) f_0(s)^{\frac{p-1}{p}} f_1(s)^{1/p} - p f_0(s)     \\
& \leq (1 + c e^{-\tau s}) \left[(p-1) f_0(s) + f_1(s)\right] - p f_0(s),
\end{aligned}
\end{equation}
where we used Young's inequality to pass from the first line to the second one.

Similarly, recalling the notation $\mathcal{T}=\hess^g \mathcal{V} - \mathcal{V}g$ as in \eqref{eqDefT}, we can estimate the derivative of $f_1$ as follows:
\begin{align*}
\frac{d}{ds} f_1(s)
&= \int_{s-1}^{s+1} \int_{\bS_1(0)} \partial_1 \left[|d\cV|_g^p(r, \theta) e^{-pr}\right] d\theta dr                                                   \\
&= p \int_{s-1}^{s+1} \int_{\bS_1(0)} \left[ |d\cV|_g^{p-2} g^{ij} (\partial_i \cV) (\hess^g_{j1} \cV) - |d\cV|_g^p \right] e^{-pr} d\theta dr           \\
&= p \int_{s-1}^{s+1} \int_{\bS_1(0)} \left[|d\cV|_g^{p-2} g^{ij} (\partial_i \cV) (g_{j1} \cV + \cT_{j1})\right] e^{-pr} d\theta dr - p f_1(s)        \\
&= p \int_{s-1}^{s+1} \int_{\bS_1(0)} |d\cV|_g^{p-2} \left[\cV \partial_1 \cV + g^{ij} (\partial_i \cV) \cT_{j1}\right] e^{-pr} d\theta dr  - p f_1(s) \\
&\leq p (1 + c e^{-\tau s}) f_0(s)^{1/p} f_1(s)^{\frac{p-1}{p}} - p f_1(s)\\
&\qquad + p f_1(s)^{\frac{p-1}{p}}\left(\int_{s-1}^{s+1} \int_{\bS_1(0)} |\cT|^p_g e^{-pr} d\theta dr\right)^{1/p}.
\end{align*}
By the assumption that we have made, we have $\cT \in L^p_{\tau-1}(\bH^n \setminus K'_1, S_2\bH^n)$. Written in terms of the exponential coordinates, this means
\[
\int_{r_0}^\infty \int_{\bS_1(0)} |\cT|^p_g \rho^{-p(\tau-1)}\sinh(r)^{n-1} d\theta dr < \infty.
\]
Since $\rho = \rho(r)$ is given by $\displaystyle \rho = \frac{1}{\cosh(r)+1} \sim 2 e^{-r}$, we conclude that
\[
\int_{r_0}^\infty \int_{\bS_1(0)} |\cT|^p_g e^{(n-1+p(\tau-1)) r} d\theta dr < \infty.
\]
As a consequence, we have
\begin{align*}
& \int_{s-1}^{s+1} \int_{\bS_1(0)} |\cT|^p_g e^{-pr} d\theta dr\\
& = \int_{s-1}^{s+1} \int_{\bS_1(0)} |\cT|^p_g e^{(n-1+p(\tau-1)) r} e^{-(n-1+p\tau) r} d\theta dr\\
& \leq \int_{s-1}^{s+1} \int_{\bS_1(0)} |\cT|^p_g e^{(n-1+p(\tau-1)) r}  d\theta dr~\left(\sup_{r \in [s-1, s+1]} e^{-(n-1+p\tau) r}\right)\\
& = e^{-(n-1+p\tau) (s-1)} \int_{s-1}^{s+1} \int_{\bS_1(0)} |\cT|^p_g e^{(n-1+p(\tau-1)) r}  d\theta dr.
\end{align*}
All in all, we have proven that there exists a constant $C > 0$ such that
\begin{equation}\label{eqEstimateT}
\left(\int_{s-1}^{s+1} \int_{\bS_1(0)} |\cT|^p_g e^{-pr} d\theta dr\right)^{1/p} \leq C \|\cT\|_{L^p_{\tau-1}} e^{-\lambda s},
\end{equation}
where $\lambda \definedas \tau + \frac{n-1}{p} \geq 2$. Inserting this in the above formula for $\displaystyle \frac{d}{ds} f_1(s)$, and using Young's inequality once again, we conclude that
\begin{equation}\label{eqDiffF1}
\begin{aligned}
\frac{d}{ds} f_1(s)
& \leq p (1 + c e^{-\tau r}) f_0(s)^{1/p} f_1(s)^{\frac{p-1}{p}} - p f_1(s) + C \|\cT\|_{L^p_{\tau-1}} e^{-\lambda s} f_1(s)^{\frac{p-1}{p}}     \\
& \leq (1 + c e^{-\tau r}) \left(f_0(s) + (p-1) f_1(s)\right) - p f_1(s) + C \|\cT\|_{L^p_{\tau-1}} e^{-\lambda s} f_1(s)^{\frac{p-1}{p}}.
\end{aligned}
\end{equation}
Adding Equations \eqref{eqDiffF0} and \eqref{eqDiffF1}, we find
\begin{align*}
& \frac{d}{ds} (f_0(s) + f_1(s))\\
& \qquad = p (1 + c e^{-\tau s})(f_0(s)+ f_1(s)) - p (f_0(s) + f_1(s)) + C \|\cT\|_{L^p_{\tau-1}} e^{-\lambda s} f_1(s)^{\frac{p-1}{p}} \\
& \qquad \leq p c e^{-\tau s} (f_0(s) + f_1(s)) + C \|\cT\|_{L^p_{\tau-1}} e^{-\lambda s} (f_0(s) + f_1(s))^{\frac{p-1}{p}}.
\end{align*}
Using this inequality, we can finally prove that $f_0 + f_1$ is bounded as follows. First, we notice that, setting $w(s) \definedas \exp \left(\frac{pc}{\tau} e^{-\tau s}\right)$, the last formula can be rewritten as
\[
    \frac{d}{ds} \left[w(s) (f_0(s) + f_1(s))\right] \leq C \|\cT\|_{L^p_{\tau-1}} w(s)^{1/p} e^{-\lambda s} \left[w(s) (f_0(s) + f_1(s))\right]^{\frac{p-1}{p}},
\]
or, equivalently,
\[
    \frac{d}{ds} \left(w(s) (f_0(s) + f_1(s))\right)^{1/p} \leq \frac{C}{p} \|\cT\|_{L^p_{\tau-1}} w(s)^{1/p} e^{-\lambda s}.
\]
This concludes the proof of the claim by noticing that the right-hand side of this inequality is integrable over $[r_0, \infty)$.
\end{proof}

\begin{claim}\label{cl1.5}
We have 
\[
|\cV| = O(e^r)\quad\text{and}\quad |d(e^{-r} \cV)|_g = O(1).
\]
\end{claim}

\begin{proof}
The proof is based on applying the Sobolev embedding theorem to the family of functions
\[
\cU_s: (t, \theta) \mapsto e^{-(t+s)}\cV(t + s, \theta)
\]
on the annulus $\cA \definedas [-1, 1] \times \bS_1(0)$ where $\cA$ is endowed with the product metric $h =dt^2 + \sigma$. To distinguish these norms from the usual $L^p$-norms over $\bH^n \setminus K'$, we will indicate the domain we are considering in the notation. So $\|\cdot\|_{L^p(\bH^n \setminus K')}$ will be the $L^p$-norm on $\bH^n \setminus K'$ defined with respect to the metric $b$, while $\|\cdot\|_{L^p(\cA)}$ will be the $L^p$-norm on the annulus with respect to the metric $h$.

First, note that we have
\begin{align}
\|\cU_s\|_{L^p(\cA)}^p\nonumber
&= \int_{-1}^1 \int_{\bS_1(0)} |e^{-(t+s)}\cV(t + s, \theta)|^p d\theta dt\nonumber\\
&= \int_{s-1}^{s+1} \int_{\bS_1(0)} |e^{-r} \cV(r, \theta)|^p d\theta dr\nonumber\\
&= f_0(s)\nonumber,
\end{align}
so that
\[
\|\cU_s\|_{L^p(\cA)} \lesssim 1
\]
by Claim \ref{cl1}.
Next, we shall obtain an estimate for $\|\hess^h\cU_s\|_{L^p(\cA)}$. This is rather intricate as an $L^\infty$ estimate for $d\cV$ will come into play. As we will see, this estimate can be obtained from a bound on $\|\hess^h\cU_s\|_{L^p(\cA)}$.

Note that $h$ and $b$ are both warped product metrics so their Levi-Civita connections only differ by the second fundamental form of the level sets of $r$ (resp. of $t$), as it follows from a direct calculation in Fermi coordinates. More specifically, the geodesic sphere of radius $r$  in $\bH^n$ has a shape operator given by $(S_b)^A_{\phantom{A}B} = \coth(r) \delta^A_B$, while the spheres of constant $t$  in $\cA$ are totally geodesic, that is $(S_h)^A_{\phantom{A}B} = 0$. We will now use these facts to compute the components of $\hess^h\cU_s$. Here and in what follows, we have $r = t+s$ where $r$ corresponds to a coordinate on the hyperbolic space while $t$ is the radial coordinate on the annulus $[-1, 1] \times \cS_1(0)$.

We have 
\begin{align}
\hess^h_{11} \cU_s(t, \theta)
&= (\partial_1)^2 (e^{-r} \cV)\nonumber\\
&= e^{-r} (\partial_1)^2 \cV - e^{-r} \cV - 2 \partial_1 (e^{-r} \cV)\nonumber\\
&= e^{-r} (\hess_{11} \cV - b_{11} \cV) - 2 \partial_1 (e^{-r} \cV)\nonumber\\
&= e^{-r} \left[\cT_{11} + \Gamma_{11}^i \partial_i \cV + (g_{11} - b_{11}) \cV\right] - 2 \partial_1 (e^{-r} \cV)\nonumber\\
&= e^{-r} \cT_{11} + \Gamma_{11}^i \partial_i(e^{-r}\cV) + (e_{11} + \Gamma_{11}^1) e^{-r} \cV - 2 \partial_1 (e^{-r} \cV),\label{eqIdentity1}
\end{align}
\begin{align}
\hess^h_{1A} \cU_s(t, \theta)
&= \partial_1 \partial_A (e^{-r} \cV)\nonumber\\
&= e^{-r} (\partial_1 \partial_A \cV - \partial_A \cV)\nonumber\\
&= e^{-r} \left[\hess_{1A} \cV - \partial_A \cV + \coth(r) \partial_A \cV\right]\nonumber\\
&= e^{-r} \left[\cT_{1A} + \Gamma_{1A}^i \partial_i \cV + g_{1A} \cV + (\coth(r)-1) \partial_A \cV\right]\nonumber\\
&= e^{-r} \cT_{1A} + \Gamma_{1A}^i \partial_i (e^{-r}\cV) + e^{-r}(\Gamma_{1A}^1 + g_{1A}) \cV + (\coth(r)-1) \partial_A (e^{-r}\cV),\label{eqIdentity2}
\end{align}
and
\begin{align}
\hess^h_{AB} \cU_s(t, \theta)
&= e^{-t+s} \hess^h_{AB} \cV(t+s, \theta)\nonumber\\
&= e^{-r} \left(\hess_{AB} \cV - \coth(r) b_{AB} \partial_1 \cV\right)\nonumber\\
&= e^{-r} \hess_{AB} \cV - \coth(r) b_{AB} e^{-r}\cV - \coth(r) b_{AB} \partial_1 (e^{-r}\cV)\nonumber\\
&= e^{-r}\cT_{AB} + \Gamma^i_{AB}\partial_i (e^{-r} \cV)\label{eqIdentity3}\\
&\qquad  + (\Gamma^1_{AB} + g_{AB} - \coth(r) b_{AB}) e^{-r} \cV - \coth(r) b_{AB} \partial_1 (e^{-r}\cV)\nonumber.
\end{align}
Next, we remark that
\begin{align*}
\left|\hess^h \cU_s\right|_h^2
&= \left(\hess^h_{11} \cU_s\right)^2 + 2 \sigma^{AB} (\hess^h_{1A} \cU_s)(\hess^h_{1B} \cU_s)\\
&\qquad + \sigma^{AB} \sigma^{A'B'} (\hess^h_{AA'} \cU_s)(\hess^h_{BB'} \cU_s).
\end{align*}
Hence,
\begin{equation}\label{eqNormHessian}
\begin{aligned}
\left|\hess^h \cU_s\right|_h^p
&\leq 2^{p-2} \left[\left(\hess^h_{11} \cU_s\right)^p + 2 \left(\sigma^{AB} (\hess^h_{1A} \cU_s)(\hess^h_{1B} \cU_s)\right)^{p/2}\right.\\
&\qquad + \left.\left(\sigma^{AB} \sigma^{A'B'} (\hess^h_{AA'} \cU_s)(\hess^h_{BB'} \cU_s)\right)^{p/2}\right],
\end{aligned}
\end{equation}
where the (irrelevant) constant $2^{p-2}$ is obtained from the fact that the mapping $x\mapsto x^{p/2}$ is convex. 
We can now estimate all three terms in \eqref{eqNormHessian} using the formulas obtained in \eqref{eqIdentity1}, \eqref{eqIdentity2} and \eqref{eqIdentity3}:
\begin{align}
\left(\hess^h_{11} \cU_s\right)^p
&\leq 5^{p-1} \left[e^{-pr} |\cT|_g^p |\partial_1|^{2p}_g + e^{-pr} |\Gamma|^p_g |\partial_1|_g^{2p} |d\cV|_g^p\right.\nonumber\\
&\qquad\qquad \left.+ e^{-pr} (|e|_b^p + |\Gamma|_b^p) |\cV|^p + 2^p |\partial_1 (e^{-r} \cV)|_b^p\right]\nonumber\\
&\lesssim e^{-pr} |\cT|^p_g + |De|^p_b |d(e^{-r}\cV)|^p_b + e^{-pr} (|e|_b^p + |De|_b^p) |\cV|^p + |d(e^{-r} \cV)|_b^p,\label{eqIdentity1p}
\end{align}
where $|\Gamma|_g$ denotes the norm of the (pseudo) tensor field $\Gamma$:
\[
|\Gamma|_g^2 = g^{ii'} g^{jj'} g_{kk'} \Gamma_{ij}^k \Gamma_{i'j'}^{k'}.
\]
\begin{align}
\left(\sigma^{AB} (\hess^h_{1A} \cU_s)(\hess^h_{1B} \cU_s)\right)^{p/2}
& = \sinh(r)^p \left(b^{AB} (\hess^h_{1A} \cU_s)(\hess^h_{1B} \cU_s)\right)^{p/2}\nonumber\\
& \lesssim |\cT|^p_b + e^{pr} |De|_b^p |d(e^{-r}\cV)|_b^p + (|e|_b^p + |De|_b^p) |\cV|^p\nonumber\\
& \qquad + e^{pr} (\coth(r)-1)^p |d(e^{-r})\cV|_b^p\nonumber\\
& \lesssim |\cT|^p_b + (e^{pr} |De|_b^p + 1) |d(e^{-r}\cV)|_b^p\nonumber\\
&\qquad  + (|e|_b^p + |De|_b^p) |\cV|^p\label{eqIdentity2p}.
\end{align}
And, finally,
\begin{align}
&\left(\sigma^{AB} \sigma^{A'B'} (\hess^h_{AA'} \cU_s)(\hess^h_{BB'} \cU_s)\right)^{p/2}\nonumber\\
&\qquad = \sinh(r)^{2p} \left(b^{AB} b^{A'B'} (\hess^h_{AA'} \cU_s)(\hess^h_{BB'} \cU_s)\right)^{p/2}\nonumber\\
&\qquad \leq e^{2pr} \left(b^{AB} b^{A'B'} (\hess^h_{AA'} \cU_s)(\hess^h_{BB'} \cU_s)\right)^{p/2}\nonumber\\
&\qquad\lesssim e^{pr} |\cT|_g^p + e ^{2pr} |De|_b^p |d(e^{-r}\cV)|_b^p + (|e|_b^p + |De|_b^p + e^{-2pr}) e^{pr} |\cV|^p + \left|d(e^{-r}\cV)\right|_b^p\label{eqIdentity3p}.
\end{align}

Inserting the three estimates \eqref{eqIdentity1p}, \eqref{eqIdentity2p} and \eqref{eqIdentity3p} into \eqref{eqNormHessian}, we obtain
\begin{align*}
&\left|\hess^h \cU_s\right|_h^p\\
&\qquad \lesssim e^{pr} |\cT|_g^p + (e^{2pr} |De|_b^p + 1) |d(e^{-r}\cV)|^p_b + e^{pr}(|e|_b^p + |De|_b^p + e^{-2pr}) |\cV|^p.
\end{align*}
We now integrate this last estimate with respect to $s \in [-1, 1]$ and $\theta \in \bS_1(0)$:
\begin{equation}\label{eqFinalEstimate}
\begin{aligned}
&\int_{-1}^1 \int_{\bS_1(0)} \left|\hess^h \cU_s\right|_h^p d\theta dt\\
&\qquad \lesssim \int_{s-1}^{s+1} \int_{\bS_1(0)} \left(e^{pr} |\cT|_b^p + e^{-pr}|\cV|^p\right) d\theta dr
+ \int_{s-1}^{s+1} \int_{\bS_1(0)} e^{2pr} |De|_b^p |d(e^{-r}\cV)|^p_b d\theta dr\\
&\qquad\qquad + \int_{s-1}^{s+1} \int_{\bS_1(0)} e^{2pr}(|e|_b^p + |De|_b^p) |e^{-r}\cV|^p d\theta dr + \int_{s-1}^{s+1} \int_{\bS_1(0)} |d(e^{-r}\cV)|^p_b d\theta dr.
\end{aligned}
\end{equation}
Note that the first integral in the right hand side of \eqref{eqFinalEstimate} is finite, as it is the sum of $f_0(s)$ and the integral 
\[
\int_{s-1}^{s+1} \int_{\bS_1(0)} e^{pr} |\cT|_b^p d\theta dr = O(e^{(2-\lambda)p r}),
\]
where $\lambda \geq 2$ is as in \eqref{eqEstimateT}.
Note also that, for any $1$-form $\psi$, we have
\[
|\psi|_b^2 = (\psi(\partial_1))^2 + b^{AB} \psi_A \psi_B \leq (\psi(\partial_1))^2 + \sigma^{AB} \psi_A \psi_B = |\psi|_h^2,
\]
so we can replace $|d(e^{-r}\cV)|^p_b$ by $|d(e^{-r}\cV)|^p_h = |d\cU_s|_h^p$ in the second and fourth terms in \eqref{eqFinalEstimate}. As for the third term, we simply note that $e^{-r} \cV = \cU_s$. Taking supremum over $\cA$ we thereby obtain:
\[
\begin{aligned}
&\int_{-1}^1 \int_{\bS_1(0)} \left|\hess^h \cU_s\right|_h^p d\theta dt\\
&\qquad \lesssim \int_{s-1}^{s+1} \int_{\bS_1(0)} \left(e^{pr} |\cT|_b^p + e^{-pr}|\cV|^p\right) d\theta dr + \|\cU_s\|_{L^\infty(\cA)}^p \int_{s-1}^{s+1} \int_{\bS_1(0)} e^{2pr}(|e|_b^p + |De|_b^p) d\theta dr\\
&\qquad\qquad + \|d\cU_s\|_{L^\infty(\cA)}^p \int_{s-1}^{s+1} \int_{\bS_1(0)} \left(e^{2pr} |De|_b^p +1\right)d\theta dr.
\end{aligned}
\]
Arguing as when deriving \eqref{eqEstimateT}, one can check that the integrals
\[
\int_{s-1}^{s+1} \int_{\bS_1(0)} e^{2pr}(|e|_b^p + |De|_b^p) d\theta dr\quad\text{and}\quad \int_{s-1}^{s+1} \int_{\bS_1(0)} \left(e^{2pr} |De|_b^p +1\right)d\theta dr
\]
are bounded independently of $s \geq r_0+1$. Hence, we obtain the following inequality:
\[
\left\|\hess^h \cU_s\right\|_{L^p(\cA)}^p
\lesssim 1 + \|d\cU_s\|_{L^\infty(\cA)}^p + \|\cU_s\|_{L^\infty(\cA)}^p,
\]
where the constant hidden in the $\lesssim$ symbol depends on $\cV$ but not on $s$.
By the Gagliardo-Niremberg inequality~\cite[Theorem 3.1]{SOUDSKY2018160}, we have
\[
\|d\cU_s\|_{L^\infty(\cA)} \leq \Lambda \|\cU_s\|_{L^p(\cA)}^{1-\mu} \|\hess^h \cU_s\|_{L^p(\cA)}^\mu + \|\cU_s\|_{L^p(\cA)},
\]
where $\displaystyle \mu = \frac{1}{2} \left(1 + \frac{n}{p}\right)$ and $\Lambda$ some large constant, see the remark in \cite{Nirenberg} about the addition of $\|\cU_s\|_{L^p(\cA)}$. Using the $\epsilon$-Young inequality, this estimate can be transformed into the following estimate valid for any $\epsilon$:
\[
\|d\cU_s\|_{L^\infty(\cA)} \leq \epsilon \|\hess^h \cU_s\|_{L^p(\cA)} + C_\epsilon \|\cU_s\|_{L^p(\cA)},
\]
A similar reasoning can be applied to $\|\cU_s\|_{L^\infty(\cA)}$. In total, we get the following estimate for $\left\|\hess^h \cU_s\right\|_{L^p(\cA)}$, for some new constant $C_\epsilon$:
\[
\left\|\hess^h \cU_s\right\|_{L^p(\cA)}^p
\lesssim \epsilon \left\|\hess^h \cU_s\right\|_{L^p(\cA)}^p + 1 + C_\epsilon\|\cU_s\|_{L^p(\cA)}^p.
\]
Choosing $\epsilon$ small enough, and using the fact that $\|\cU_s\|_{L^p(\cA)}$ is bounded independently of $s$, we conclude that $\left\|\hess^h \cU_s\right\|_{L^p(\cA)}$ is bounded. From the Sobolev embedding theorem, we conclude that both $\|\cU_s\|_{L^\infty(\cA)}$ and $\|d\cU_s\|_{L^\infty(\cA)}$ are bounded independently of $s$.
This concludes the proof of the claim.
\end{proof}

Now that we have established the growth rate of $\cV$, we can prove that $\rho \cV$ induces, in some sense, a function on the sphere at infinity. For this, given a bounded function $u$ on $\bS_1(0)$, we introduce the function $h_u(s)$ defined as follows:
\begin{equation}\label{eqDefHU}
    h_u(s) \definedas \int_{s-1}^{s+1} \int_{\bS_1(0)}  \cV(r, \theta) u(\theta) e^{-r} d\theta dr.
\end{equation}

\begin{claim}\label{cl2}
The function $h_u$ admits a limit as $s$ tends to infinity. The mapping
\[\ell: u \mapsto \lim_{s\to \infty} h_u(s)\]
satisfies $|\ell(u)| \leq C \|u\|_{L^c(\bS_1(0), \bR)}$, where $c$ is such that $\displaystyle \frac{1}{c} + \frac{1}{p} = 1$.
\end{claim}

\begin{proof}
We compute the first and second order derivatives of $h_u$ as we did in the proof of Claim \ref{cl1}:
\begin{align*}
\frac{d}{ds} h_u(s)
& = \int_{s-1}^{s+1} \int_{\bS_1(0)}  \partial_1\left(\cV(r, \theta) u(\theta) e^{-r}\right) d\theta dr,                  \\
& = \int_{s-1}^{s+1} \int_{\bS_1(0)}  \left(\partial_1 \cV - \cV\right) u(\theta) e^{-r} d\theta dr  ,                     \\
\frac{d^2}{ds^2} h_u(s)
& = \int_{s-1}^{s+1} \int_{\bS_1(0)}  \left(\partial_1^2 \cV - 2 \partial_1\cV + \cV \right) u(\theta) e^{-r} d\theta dr.
\end{align*}
Note that $\partial_1^2 \cV = \hess^g_{11} \cV + \Gamma_{11}^k \partial_k \cV$ where the Christoffel symbols $\Gamma^i_{jk}$ are defined through \eqref{eqChristoffel} as in the rest of the paper\footnote{Note that there is no mistake here: a  calculation shows that $\Gamma_{11}^k$ defined by \eqref{eqChristoffel} actually do coincide with the ``usual'' Christoffel symbols.  This is due to the fact that the ``usual'' Christoffel symbols ${}^b \Gamma^k_{11}$ of the hyperbolic metric are zero.}. Hence,
\begin{align*}
    \partial_1^2 \cV
        & = g_{11} \cV + \Gamma_{11}^k \partial_k \cV + \cT_{11}        \\
        & = (1 + e_{11}) \cV + \Gamma_{11}^k \partial_k \cV + \cT_{11}.
\end{align*}
Inserting this equality in the expression for the second order derivative of $h_u$, we get
\begin{align}
        & \frac{d^2}{ds^2} h_u(s) \nonumber\\
        & \qquad = \int_{s-1}^{s+1} \int_{\bS_1(0)}  \left(2 \cV - 2 \partial_1\cV + e_{11} \cV + \Gamma_{11}^k \partial_k \cV + \cT_{11} \right) u(\theta) e^{-r} d\theta dr \nonumber                 \\
        & \qquad = -2 \frac{d}{ds} h_u(s) + \int_{s-1}^{s+1} \int_{\bS_1(0)}  \left(e_{11} \cV + \Gamma_{11}^k \partial_k \cV + \cT_{11} \right) u(\theta) e^{-r} d\theta dr.\label{eqEstimateHUSecond}
\end{align}
The first term in the remaining integral in \eqref{eqEstimateHUSecond} can be estimated using Claim 2 as follows:
\begin{align*}
& \left|\int_{s-1}^{s+1} \int_{\bS_1(0)}  e_{11} \cV~u(\theta) e^{-r} d\theta dr\right|                                                                                          \\
& \qquad \leq \left(\sup_{(r, \theta) \in [s-1, s+1]\times \bS_1(0) } \cV(r, \theta) e^{-r}\right) \left(\int_{s-1}^{s+1} \int_{\bS_1(0)}  |e|_b^p d\theta dr\right)^{1/p} \\
& \qquad\qquad \times \left(\int_{s-1}^{s+1} \int_{\bS_1(0)}  |u(\theta)|^c d\theta dr\right)^{1/c}                                                                              \\
& \qquad \leq \mu \|u\|_{L^c(\bS_1(0), \bR)} \left(\int_{s-1}^{s+1} \int_{\bS_1(0)}  |e|_b^p d\theta dr\right)^{1/p},
\end{align*}
for some constant $\mu > 0$ independent of $s$. In what follows $\mu$ might vary from line to line. Note that we can write
\begin{align}
& \left(\int_{s-1}^{s+1} \int_{\bS_1(0)}  |e|_b^p d\theta dr\right)^{1/p}\nonumber                                                                                     \\
& \qquad\leq \left(e^{-(p \tau + (n-1))(s-1)}\int_{r_0}^\infty \int_{\bS_1(0)} |e|_b^p e^{(p \tau + (n-1))r} d\theta dr\right)^{1/p}\nonumber \\
& \qquad \leq \mu~e^{-\lambda s} \|e\|_{L^p_\tau},\label{eqEstimateError}
\end{align}
where $\displaystyle \lambda = \tau + \frac{n-1}{p}$ as in the proof of Claim 1. Hence,
\[
\left|\int_{s-1}^{s+1} \int_{\bS_1(0)}  e_{11} \cV~u(\theta) e^{-r} d\theta dr\right| \leq \mu e^{-\lambda s} \|e\|_{L^p_\tau} \left(\sup_{(s, \theta) \in [r_0, \infty) \times \bS_1(0)} e^{-r}\cV(r, \theta)\right) \|u\|_{L^c(\bS_1(0), \bR)}.
\]
Similarly, for the second term in the integral in the last line of  \eqref{eqEstimateHUSecond}, we get
\[
\left|\int_{s-1}^{s+1} \int_{\bS_1(0)}  \Gamma_{11}^k (\partial_k \cV)~u(\theta) e^{-r} d\theta dr\right| \leq \mu e^{-\lambda s} \|De\|_{L^p_\tau} \left(\sup_{s \in [r_0, \infty)}e^{-r} |d\cV|\right) \|u\|_{L^c(\bS_1(0), \bR)}.
\]
As for the last term, we proceed as in the proof of Claim 1 to obtain
\[
    \left|\int_{s-1}^{s+1} \int_{\bS_1(0)}  \cT_{11} u(\theta) e^{-r} d\theta dr \right|
    \leq c \|\cT\|_{L^p_{\tau-1}} \|u\|_{L^c(\bS_1(0), \bR)} e^{-\lambda s}.
\]
As a consequence, if we set
\[
    \theta(s) \definedas \int_{s-1}^{s+1} \int_{\bS_1(0)}  \left(e_{11} \cV + \Gamma_{11}^k \partial_k \cV + \cT_{11} \right) u(\theta) e^{-r} d\theta dr,
\]
we have shown that there exists a large constant $\mu > 0$ such that
\begin{equation}\label{eqTheta}
    |\theta(s)| \leq \mu e^{-\lambda s} \|u\|_{L^c(\bS_1(0), \bR)}.
\end{equation}
All in all, we may rewrite \eqref{eqEstimateHUSecond} as the differential equation 
\[
    \frac{d^2 h_u}{ds^2} + 2 \frac{d h_u}{ds} = \theta(s),
\]
integrating which we get
\[
    h_u(s) = h_u(r_0) + \frac{1}{2} \left(1 - e^{-2(s-r_0)}\right) h'_u(r_0)
    + \int_{r_0}^s \left(e^{-2t} \int_{r_0}^t \theta(x) e^{2x} dx\right) dt.
\]
Using \eqref{eqTheta} it is straightforward to check that $h_u$ has a limit $\ell(u)$ as $s\to \infty$ and that $|\ell(u)| \lesssim \|u\|_{L^c(\bS_1(0), \bR)}$
as claimed.
\end{proof}

From Claim \ref{cl2}, we see that $\ell$ is a continuous linear form over $L^c(\bS_1(0), \bR)$. As as consequence, there exists a function $v \in L^p(\bS_1(0), \bR)$, such that, for any smooth function $u$ over $\bS_1(0)$,
\begin{equation}\label{eqDefL}
    \ell(u) = \int_{\bS_1(0)} u(x) v(x) dx.
\end{equation}
We now prove that the function $v$ is actually smooth and satisfies $\mathring{\hess}^\sigma v = 0$, where $\mathring{\hess}~v$ denotes the traceless part of the Hessian of $v$. This is the content of the next two claims:
\begin{claim}\label{cl3}
The function $v$ satisfies the equation
\begin{equation}\label{eqTracelessHessian}
    \mathring{\hess}^\sigma v = 0
\end{equation}
in the sense of distributions.
\end{claim}

\begin{proof}
Let $T^{AB}$ be an arbitrary symmetric traceless 2-tensor on $\bS_1(0)$. We are to show that
\[
    \ell({}^\sigma\!\! {D}_A {}^\sigma\!\! {D}_B T^{AB}) = \int_{\bS_1(0)} v(x) {}^\sigma\!\! {D}_A {}^\sigma\!\! {D}_B T^{AB} dx = 0,
\]
where indices $A, B \in \{2, \ldots, n\}$ denote directions tangent to the geodesic sphere and ${}^\sigma\!\! {D}$ is the Levi-Civita connection of the round metric $\sigma$. To this end, we return to the definition of $\ell(u)$: $\ell(u) = \lim_{s \to \infty} h_{u}(s)$ where $h_u$ is as in \eqref{eqDefHU}. Setting $u \equiv {}^\sigma\!\! {D}_A {}^\sigma\!\! {D}_B T^{AB}$ and integrating by parts over $\bS_1(0)$ we obtain
\begin{align*}
    h_u(s)
        & = \int_{s-1}^{s+1} \int_{\bS_1(0)}  \cV(r, \theta) \left({}^\sigma\!\! {D}_A {}^\sigma\!\! {D}_B T^{AB}\right) e^{-r} d\theta dr \\
        & = \int_{s-1}^{s+1} \int_{\bS_1(0)}  \hess^\sigma_{AB} \cV(r, \theta) T^{AB} e^{-r} d\theta dr.
\end{align*}
Using \eqref{eqChristoffel} and \eqref{eqDefT} we can rewrite
\begin{align*}
    \hess^\sigma_{AB} \cV
        & = \hess^b_{AB} \cV + \coth(r) \partial_1 \cV \sigma_{AB}                                                 \\
        & = \hess^g_{AB} \cV + \Gamma^i_{AB} \partial_i \cV + \coth(r) \partial_1 \cV \sigma_{AB}                  \\
        & = \cV g_{AB} + \cT_{AB} + \Gamma^i_{AB} \partial_i \cV + \coth(r) \partial_1 \cV \sigma_{AB}             \\
        & = \cV (b_{AB} + e_{AB}) + \cT_{AB} + \Gamma^i_{AB} \partial_i \cV + \coth(r) \partial_1 \cV \sigma_{AB}.
\end{align*}
Contracting with $T^{AB}$ and using the fact that it is traceless, we get:
\[
    (\hess^\sigma_{AB} \cV) T^{AB} = \cV e_{AB} T^{AB} + T^{AB} \cT_{AB} + \Gamma^i_{AB} \partial_i \cV T^{AB}.
\]
Hence,
\begin{equation}\label{eqHUHessian}
    h_u(s)
    = \int_{s-1}^{s+1} \int_{\bS_1(0)}  \left(\cV e_{AB} T^{AB} + T^{AB} \cT_{AB} + \Gamma^i_{AB} \partial_i \cV T^{AB}\right) e^{-r} d\theta dr.
\end{equation}
We now prove that all terms in this expression vanish when $s$ tends to infinity so $\lim_{s \to \infty} h_u(s) = 0$. First note that
\[
|e_{AB} T^{AB}| \leq |e|_\sigma |T|_\sigma = \sinh(r)^2 |e|_b |T|_\sigma \leq e^{2r} |e|_b |T|_\sigma.
\]
So
\begin{align*}
& \left|\int_{s-1}^{s+1} \int_{\bS_1(0)}  \cV e_{AB} T^{AB} e^{-r} d\theta dr\right| \\
& \qquad \leq \int_{s-1}^{s+1} \int_{\bS_1(0)}  \cV |e|_b |T|_\sigma e^{r} d\theta dr \\
& \qquad \leq \left(\int_{s-1}^{s+1} \int_{\bS_1(0)}  \cV^p e^{-pr} d\theta dr\right)^{1/p} \left(\int_{s-1}^{s+1} \int_{\bS_1(0)}  |e|_b^p e^{2pr} d\theta dr\right)^{1/p} \\
& \qquad \qquad \times\left(\int_{s-1}^{s+1} \int_{\bS_1(0)}  |T|_\sigma^c d\theta dr\right)^{1/c}.
\end{align*}
From Claim \ref{cl1}, we know that the first term in the product is bounded. The third one is bounded as well. The second one is more complex. Since $p\tau + n-1 = p\lambda \geq 2p$ we have
\[
\int_{s-1}^{s+1} \int_{\bS_1(0)}  |e|_b^p e^{2pr} d\theta dr
\leq \int_{s-1}^{s+1} \int_{\bS_1(0)}  |e|_b^p e^{(n-1 + p \tau) r} d\theta dr.
\]
Assume that the right hand side does not tend to zero as $s$ tends to infinity. Then there exists an $\epsilon > 0$ and a sequence $(s_k)_k$ such that
\[
\int_{s_k-1}^{s_k+1} \int_{\bS_1(0)} |e|_b^p e^{(n-1 + p \tau) r} d\theta dr \geq \epsilon.
\]
Upon extracting a subsequence, we can assume that $s_{k+1} > s_k + 2$ so that the intervals $(s_k-1, s_k+1)$, $k \geq 0$, do not overlap.
By summing over $k$, we get
\[
\int_{r_0}^\infty \int_{\bS_1(0)} |e|_b^p e^{(n-1 + p \tau) r} d\theta dr \geq \sum_k \int_{s_k-1}^{s_k+1} \int_{\bS_1(0)} |e|_b^p e^{(n-1 + p \tau) r} d\theta dr = \infty,
\]
a contradiction. Hence,
\[
\left(\int_{s-1}^{s+1} \int_{\bS_1(0)}  |e|_b^p e^{2pr} d\theta dr\right)^{1/p} \to_{s \to \infty} 0.
\]
In summary, we have shown that the first term in \eqref{eqHUHessian} satisfies
\[
\int_{s-1}^{s+1} \int_{\bS_1(0)}  \cV e_{AB} T^{AB} e^{-r} d\theta dr \to_{s \to \infty} 0.
\]
A similar analysis applies to the last term in \eqref{eqHUHessian} as $|\Gamma|_b$ respectively $|\partial \cV|_b$ satisfy the same type of estimates as $|e|_b$ respectively $\cV$. The argument for the second term is a bit different. We have
\begin{align*}
& \left|
\int_{s-1}^{s+1} \int_{\bS_1(0)}  \cT_{AB} T^{AB} e^{-r} d\theta dr \right|\\
& \qquad \leq \int_{s-1}^{s+1} \int_{\bS_1(0)}  \left|\cT\right|_\sigma \left|T\right|_\sigma e^{-r} d\theta dr\\
& \qquad \leq \int_{s-1}^{s+1} \int_{\bS_1(0)}  \left|\cT\right|_b \left|T\right|_\sigma e^r d\theta dr\\
& \qquad \leq \left(\int_{s-1}^{s+1} \int_{\bS_1(0)}  \left|\cT\right|_b^p e^{p r} d\theta dr\right)^{1/p} \left(\int_{s-1}^{s+1} \int_{\bS_1(0)}  \left|T\right|_\sigma^{p/(p-1)} d\theta dr\right)^{\frac{p-1}{p}},
\end{align*}
where we used the fact that $|\cT|_\sigma \leq \sinh(r)^2 |\cT|_b \leq e^{2r} |\cT|_b$. However, as $\cT \in L^p_{\tau-1}(\bH^n \setminus K, S_2\bH^n)$, we have
\[
    \int_{r_0}^\infty \int_{\bS_1(0)} |\cT|_b^p e^{p(\tau-1)r + (n-1)r} d\theta dr < \infty
\]
which implies that
\[
    \int_{r_0}^\infty \int_{\bS_1(0)} |\cT|_b^p e^{p r} d\theta dr < \infty
\]
due to the condition $\displaystyle \tau + \frac{n-1}{p} \geq 2$. Arguing as above, we conclude that
\[
    \int_{s-1}^{s+1} \int_{\bS_1(0)}  \cT_{AB} T^{AB} e^{-r} d\theta dr \to_{s \to \infty} 0.
\]
\end{proof}

\begin{claim}\label{cl4}
The function $v$ is in $C^\infty(\bS_1(0), \bR)$ and satisfies
Equation \eqref{eqTracelessHessian} in the strong sense.
\end{claim}

For proving this claim, we will need the following result:
\begin{proposition}\label{propTracelessHessian}
The vector space $\cH$ of functions $v \in C^\infty(\bS_1(0), \bR)$ satisfying Equation \eqref{eqTracelessHessian} is $(n+1)$-dimensional and is spanned by the functions $1$ and $x^i$, $i \in \ldbrack 1, n\rdbrack$.
\end{proposition}

\begin{proof}
Since $v$ is smooth, $v$ admits a critical point $x_0 \in \bS_1(0)$. Due to rotational symmetry of the problem, we can assume that $x_0$ is the North pole, $(0, \ldots, 0, 1)$. Let $\tau$ denote the round metric on $\bS^{n-2}$. Then, in geodesic normal coordinates centered at $x_0$, the metric $\sigma$ reads
\[
    \sigma = d\theta^2 + (\sin^2 \theta) \tau.
\]
Let $(\theta, \phi^1, \ldots, \phi^{n-2})$ be a corresponding Fermi coordinate system, i.e. so that $(\phi^1, \ldots, \phi^{n-2})$ are coordinates on $\bS^{n-2}$ (for example the polar coordinate system). Within this proof we will denote $\phi^0 = \theta$ and our convention will be that lower case Latin indices run from $0$ to $n-2$ while upper case Latin indices run from $1$ to $n-2$.
Since the vectors $\partial_0 = \pdiff{~}{\theta}$ and $\partial_A = \pdiff{~}{\phi^A}$ are orthogonal, Equation \eqref{eqTracelessHessian} implies
\[
    0 = \hess^\sigma_{0 A} v = \partial_0 \partial_A v + (\cot \theta) \partial_A v
\]
which gives $\partial_A v = \sin(\theta) X_A$ for some vector field $X$ on $\bS^{n-2}$. We actually have, for any $\theta_0 \in (0, \pi)$, that
\[
\partial_A v(\theta_0, \phi^1, \ldots, \phi^{n-2}) = \sin(\theta_0) X_A,
\]
so $X_A = \partial_A f$ with $\displaystyle f \definedas \frac{1}{\sin(\theta_0)} v$. As a consequence, $v = h(\theta) + \sin(\theta) f(\phi^1, \ldots, \phi^{n-2})$ for some function $h$. As $v$ is smooth at $x_0$ and $x_0$ is a critical point for $v$, we have $h(\theta) = h(0) + O(\theta^2)$ and, since $\sin(\theta) = \theta + O(\theta^3)$, $f \equiv 0$.

We have proven that $f \equiv 0$ and $v = h(\theta)$. Next we note that Equation \eqref{eqTracelessHessian} can be rewritten as $\hess^\sigma h = \psi \sigma$ for some unknown function $\psi$. The radial component of this equation yields
\[
    \frac{\partial^2 h}{\partial \theta^2} = \hess^\sigma h = \psi \sigma_{00} = \psi.
\]
Similarly, the tangential components read
\[
    \frac{\partial^2 h}{\partial \phi^A \partial \phi^B} + (\cos(\theta) \sin (\theta)) \frac{\partial h}{\partial \theta} \tau_{AB} = \psi \sin(\theta)^2 \tau_{AB}.
\]
From these two equations and the fact that $\displaystyle \frac{\partial^2 h}{\partial \phi^A \partial \phi^B} \equiv 0$, we find that $h$ satisfies
\[
    \cos(\theta) h'(\theta) = \sin(\theta) h''(\theta).
\]
The general solution to this equation is $h = c_1 + c_2 \cos(\theta)$ for arbitrary constants $c_1$ and $c_2$. In terms of coordinates on $\bR^n$ restricted to $\bS_1(0)$ this means that $v = c_1 + c_2 x^n$. By rotational invariance of the problem, we conclude that the set of solutions to \eqref{eqTracelessHessian} is the span of $1$ and the $x^i$'s.
\end{proof}

We can now return to the proof of Claim \ref{cl4}:
\begin{proof}[Proof of Claim \ref{cl4}]
The idea of the proof is to take advantage of the fact that \[\bS_1(0) \simeq SO(n)/SO(n-1)\] is a homogeneous space to define mollification operators using some sort of convolution.
Let $(\psi_k)_k$ be a sequence of nonnegative smooth functions on $SO(n)$ such that
\[
\int_{SO(n)} \psi_k(g) d\mu(g) = 1 \quad\text{for all } k,
\]
where $d\mu(g)$ denotes the Haar measure on $SO(n)$\footnote{In this proof we use the letter $g$ to denote elements of the Lie group $SO(n)$ as the tradition dictates. This should not be confused with the metric $g$ used elsewhere as no reference to it will be made in here.},
and such that for any neighborhood $U$ of the identity element $\id$, there exists an integer $K$ such that for all $k \geq K$, $\supp \psi_k \subset U$. We set 
\[
v_k(x) \definedas \int_{SO(n)} \psi_k(g^{-1}) v(g\cdot x) d\mu(g).
\]
First of all, we note that $\|v_k\|_{L^c(\bS_1(0), \bR)} \leq \|v\|_{L^c(\bS_1(0), \bR)}$. Indeed, as $c \geq 1$, the function $t \mapsto |t|^c$ is convex so it follows from Jensen's inequality that
\[
    \left|\int_{SO(n)} \psi_k(g^{-1}) v(g\cdot x) d\mu(g)\right|^c
    \leq \int_{SO(n)} \psi_k(g^{-1}) |v(g\cdot x)|^c d\mu(g).
\]
Hence,
\begin{align*}
    \|v_k\|_{L^c(\bS_1(0), \bR)}
        & = \left[\int_{\bS_1(0)} \left|\int_{SO(n)} \psi_k(g^{-1}) v(g\cdot x) d\mu(g)\right|^c d\mu^\sigma(x)\right]^{1/c} \\
        & \leq \left[\int_{\bS_1(0)} \int_{SO(n)} \psi_k(g^{-1}) |v(g\cdot x)|^c d\mu(g) d\mu^\sigma(x)\right]^{1/c}         \\
        & = \left[\int_{SO(n)} \psi_k(g^{-1}) \int_{\bS_1(0)} |v(g\cdot x)|^c d\mu^\sigma(x) d\mu(g) \right]^{1/c}           \\
        & = \left[\int_{SO(n)} \psi_k(g^{-1}) \int_{\bS_1(0)} |v(x)|^c d\mu^\sigma(x) d\mu(g) \right]^{1/c}                  \\
        & = \|v\|_{L^c(\bS_1(0), \bR)},
\end{align*}
where we used the fact that the measure $d\mu^\sigma$ on $\bS_1(0)$ is invariant under the action of $SO(n)$.
Next, we claim that $v_k \in C^\infty(\bS_1(0), \bR)$ and that $v_k \to v$ in $L^c(\bS_1(0), \bR)$. For any $h \in SO(n)$, let $L_h$ denote the map from $\bS_1(0)$ onto itself defined by $x \mapsto h \cdot x$. Then we have
\begin{align*}
    (L_h)^* v_k(x)
        & = v_k(h \cdot x)                                  \\
        & = \int_{SO(n)} \psi_k(g^{-1}) v(g h\cdot x) d\mu(g)      \\
        & = \int_{SO(n)} \psi_k(g^{-1}h^{-1}) v(g\cdot x) d\mu(g),
\end{align*}
by the right invariance of the Haar measure. In particular, given $X \in \mathfrak{so(n)}$ and $t \in \bR$, we can choose $h = \exp(t X)$. If $\Phi^X_t$ denotes the flow associated to the left invariant vector field $X$ on $G$, we know that $\Phi^X_{-t}(g) = g \exp(-t X) = g h^{-1}$ (see e.g. \cite[Proof of Theorem 9.16]{LeeSmoothManifolds}. Hence,
\begin{align*}
    (L_{\exp(t X)})^* v_k(x)
        & = \int_{SO(n)} \psi_k(\Phi^X_{-t}(g^{-1})) v(g\cdot x) d\mu(g)     \\
        & = \int_{SO(n)} ((\Phi^X_{-t})^*\psi_k)(g^{-1}) v(g\cdot x) d\mu(g).
\end{align*}
As a consequence,
\begin{align*}
    X v_k(x)
        & = \lim_{t \to 0} \frac{(L_{\exp(t X)})^* v_k(x) - v_k(x)}{t}                                              \\
        & = \int_{SO(n)} \left[\lim_{t \to 0} \frac{((\Phi^X_{-t})^*\psi_k)(g^{-1}) - \psi_k(g^{-1})}{t}\right] v(g\cdot x) d\mu(g) \\
        & = \int_{SO(n)} (- X \cdot \psi_{k})(g^{-1}) v(g\cdot x) d\mu(g)
\end{align*}
is well-defined and continuous. Note that we can continue the argument adding more and more vector fields on the left as the only property of $\psi_k$ used here is that it is a smooth function. The convergence $v_k \to v$ in $L^c(\bS_1(0), \bR)$ follows now by standard arguments, using the fact that continuous functions are dense in $L^c(\bS_1(0), \bR)$ and are uniformly continuous as $\bS_1(0)$ is compact (see e.g. \cite[Theorem 4.22]{BrezisFunctionalAnalysis}).\\

We now claim that for any traceless symmetric tensor $T \in C^\infty(\bS_1(0), \mathring{\mathrm{Sym}}_2 \bS_1(0))$, we have
\[
    0 = \int_{\bS_1(0)} v_k(x) {}^\sigma\!\! D ^A {}^\sigma\!\! {D}^B T_{AB} d\mu^\sigma(x).
\]
Indeed, from the fact that $SO(n)$ acts by isometries on $\bS_1(0)$, for $u = {}^\sigma\!\! D ^A {}^\sigma\!\! {D}^B T_{AB}$, we have
\begin{align*}
\int_{\bS_1(0)} u(x) v_k(x) d\mu^\sigma(x)
&= \int_{\bS_1(0)} u(x) \left(\int_{SO(n)} \psi_k(g^{-1}) v(g\cdot x) d\mu(g)\right) d\mu^\sigma(x)\\
&= \int_{SO(n)} \psi_k(g^{-1}) \int_{\bS_1(0)} u(x) v(g\cdot x) d\mu^\sigma(x) d\mu(g)      \\
&= \int_{SO(n)} \psi_k(g^{-1}) \int_{\bS_1(0)} u(g^{-1}\cdot x) v(x) d\mu^\sigma(x) d\mu(g) \\
&= \int_{SO(n)} \psi_k(g^{-1}) \int_{\bS_1(0)} ((L_g)_* u)(x) v(x) d\mu^\sigma(x) d\mu(g).
\end{align*}
Again from the fact that $SO(n)$ acts by isometries on $\bS_1(0)$, we have  that
\[
(L_g)_* u = (L_g)_* ({}^\sigma\!\! D ^A {}^\sigma\!\! {D}^B T_{AB})
= {}^\sigma\!\! D ^A {}^\sigma\!\! {D}^B ((L_g)_* T)_{AB}
\]
is the double divergence of a symmetric trace free 1-tensor. Thus,
\[
\int_{\bS_1(0)} ((L_g)_* u)(x) v(x) d\mu^\sigma(x) = 0.
\]
This shows that for any $T \in C^\infty(\bS_1(0), \mathring{\mathrm{Sym}}_2 \bS_1(0))$, we have
\[
\int_{\bS_1(0)} v_k(x) \left({}^\sigma\!\! D ^A {}^\sigma\!\! {D}^B T_{AB}\right) d\mu^\sigma(x) = 0.
\]
Integrating by parts twice, this yields
\[
\int_{\bS_1(0)} \hess^\sigma_{AB}(v_k) T^{AB} d\mu^\sigma(x) = 0.
\]
As $T$ is arbitrary, we conclude that $\mathring{\hess}^\sigma v_k \equiv 0$.
From Proposition \ref{propTracelessHessian}, we know that all $v_k$ live in the finite dimensional subspace $\cH \subset L^c(\bS_1(0), \bR)$. As $\cH$ is closed, we can pass to the limit and conclude that $v \in \cH$, showing that $v$ is smooth and satisfies \eqref{eqTracelessHessian} in the strong sense.
\end{proof}

Note that Claim \ref{cl4} implies that given $\cV \in \cN^g$, we can associate to it a unique function $v \in \cH$. We denote this linear map by $\cB: \cN^g \to \cH$.  Recall (see a comment after Lemma \ref{lmHarmonic}) that $\cX^\mu$ denotes the eignfunction of the Laplacian $\Delta^g$ that is asymptotic to the coordinate function $X^\mu$, $\mu=0,1,\ldots n$.

\begin{claim}\label{cl5}
We have $\cB(\cX^0) = 1$ and, for all $i \in \ldbrack 1, n\rdbrack$, $\cB(\cX^i) = x^i$.
In particular, the map $\cB$ is surjective.
\end{claim}

\begin{proof}
We write $\cX^i = X^i + (\cX^i - X^i)$ and compute $h_u(r)$ with $\cV = \cX^i$. As $\cX^i - X^i \in L^p_{\tau-1}(\bH^n \setminus K, \bR)$ and $X^i = \sinh(r) x^i$ (see \eqref{eqLapseGeodesic}), we conclude that
\begin{align*}
\lim_{s \to \infty} h_u(s)
& = \lim_{s \to \infty} \int_{s-1}^{s+1} \int_{\bS_1(0)} X^i(r, \theta) u(\theta) e^{-r} d\theta dr           \\
& \qquad\qquad+ \lim_{s \to \infty}\int_{s-1}^{s+1} \int_{\bS_1(0)} (\cX^i - X^i) u(\theta) e^{-r} d\theta dr \\
& = \lim_{s \to \infty} \int_{s-1}^{s+1} \int_{\bS_1(0)} x^i(\theta) \sinh(r) u(\theta) e^{-r} d\theta dr     \\
& = \int_{\bS_1(0)} x^i(\theta) u(\theta) d\theta \lim_{s \to \infty}\int_{s-1}^{s+1} \sinh(r)  e^{-r} dr     \\
& = \int_{\bS_1(0)} x^i(\theta) u(\theta) d\theta,
\end{align*}
observing that the limit in the second line is zero (see the proof of Claim \ref{cl3} for a similar argument). The proof for $\cX^0$ is identical.
\end{proof}

Before proving injectivity of $\cB$ (which will be Claim \ref{cl8}), we need to obtain  better estimates for the asymptotic behavior of functions in $\cN^g$. We do this in the following two claims.

\begin{claim}\label{cl6}
Let \( f_2(s) \) be the function defined by
\[
f_2(s) \coloneqq \int_{s-1}^{s+1} \int_{\mathbb{S}^{n-1}} \left( \partial_1 (\mathcal{V} e^{-r}) \right)^p \, d\theta \, dr.
\]
Then, $f_2(s)=O(s^p e^{-2ps})$.
\end{claim}

\begin{proof}
We compute the derivative of $f_2$ with respect to $s$:
\[
\frac{d}{ds} f_2(s)
= p \int_{s-1}^{s+1} \int_{\bS_1(0)} \left(\partial_1 (\cV e^{-r})\right)^{p-1}  \left(\partial_1^2 (\cV e^{-r})\right) d\theta dr.
\]
Now note that
\begin{align*}
\partial_1^2 (\cV e^{-r})
& = (\partial_1^2 \cV) e^{-r} - 2 (\partial_1 \cV) e^{-r} + \cV e^{-r}                                             \\
& = e^{-r} \left( \hess_{11} \cV + \Gamma_{11}^i \partial_i \cV + \cV \right) - 2 (\partial_1 \cV) e^{-r}        \\
& = e^{-r} \left( \cV g_{11} + \cT_{11} + \Gamma_{11}^i \partial_i \cV + \cV \right) - 2 (\partial_1 \cV) e^{-r}   \\
& = e^{-r} \left( 2 \cV + \cT_{11} + \Gamma_{11}^i \partial_i \cV + \cV e_{11} \right) - 2 (\partial_1 \cV) e^{-r} \\
& = - 2 \partial_1 (\cV e^{-r}) + e^{-r} \left(\cT_{11} + \Gamma_{11}^i \partial_i \cV + \cV e_{11}\right).
\end{align*}
Inserting this formula in that of the derivative of $f_2$, we get 
\begin{align}
\frac{d}{ds} f_2(s)
& = - 2 pf_2(s)\nonumber\\
& \qquad + p \int_{s-1}^{s+1} \int_{\bS_1(0)} \left(\partial_1 (\cV e^{-r})\right)^{p-1} e^{-r}\left(\cT_{11} + \Gamma_{11}^i \partial_i \cV + \cV e_{11}\right) d\theta dr\nonumber\\
& \leq - 2 p f_2(s)
+ p f_2(s)^{\frac{p-1}{p}} \left[\left(\int_{s-1}^{s+1} \int_{\bS_1(0)} \left|e^{-r} \cT_{11}\right|^p d\theta dr\right)^{\frac{1}{p}} \right.\nonumber\\
& \qquad + \left(\int_{s-1}^{s+1} \int_{\bS_1(0)} \left|\Gamma_{11}^{\cdot}\right|^p_g |e^{-r} d\cV|_g^p d\theta dr\right)^{\frac{1}{p}} \nonumber\\
& \qquad \left. + \left(\int_{s-1}^{s+1} \int_{\bS_1(0)} e^{-pr} \cV^p |e|_b^p d\theta dr\right)^{\frac{1}{p}}\right]. \label{eqNotYetTheWorst}
\end{align}
From the proof of Claim \ref{cl1} (see Equation \eqref{eqEstimateT}), we have
\[
\left(\int_{s-1}^{s+1} \int_{\bS_1(0)} \left|e^{-r} \cT_{11}\right|^p d\theta dr\right)^{\frac{1}{p}} \leq \Lambda \| \cT\|_{L^p_{\tau-1}} e^{-\lambda s}.
\]
Also, from Claim \ref{cl1.5}, we have
\begin{align*}
\left(\int_{s-1}^{s+1} \int_{\bS_1(0)} e^{-pr} \cV^p |e|_b^p d\theta dr\right)^{\frac{1}{p}}
& \lesssim \left(\int_{s-1}^{s+1} \int_{\bS_1(0)}  |e|_b^p d\theta dr\right)^{\frac{1}{p}}\\
& \lesssim e^{-\lambda s}.
\end{align*}
So the first and the third integrals in the right hand side of \eqref{eqNotYetTheWorst} are $O(e^{-\lambda s})$. Finally, for the second integral in \eqref{eqNotYetTheWorst}, we remark that, from Claim \ref{cl1.5}, we have
\[
|e^{-r} d\cV|_g \leq |d(e^{-r} \cV)|_g + |\cV|~|de^{-r}|_g
\]
and the right-hand side is bounded independently of $(r, \theta)$. So,
\begin{align*}
\left(\int_{s-1}^{s+1} \int_{\bS_1(0)} \left|\Gamma_{11}^{\cdot}\right|^p_g |e^{-r} d\cV|_g^p d\theta dr\right)^{\frac{1}{p}}
&\lesssim \left(\int_{s-1}^{s+1} \int_{\bS_1(0)} \left|\Gamma_{11}^{\cdot}\right|^p_g d\theta dr\right)^{\frac{1}{p}}\\
&\lesssim e^{-\lambda s}.
\end{align*}
From \eqref{eqNotYetTheWorst}, we conclude that, for some large constant $\Lambda$, we have
\[
\frac{d}{ds} f_2(s)
\leq - 2 p f_2(s) + \Lambda e^{-\lambda s} f_2(s)^{\frac{p-1}{p}}.
\]
Equivalently, we have
\[
\frac{d}{ds} \left(f_2(s)\right)^{1/p}
\leq - 2 \left(f_2(s)\right)^{1/p} + \frac{\Lambda}{p} e^{-\lambda s}.
\]
As $\lambda \geq 2$, the estimate follows.
\end{proof}

We now conclude the proof of Proposition \ref{propApproximateLapse} with the following claim:

\begin{claim}\label{cl8}
The map $\cB$ is injective.
\end{claim}

Indeed, this claim combined with Claim \ref{cl5} shows that $\dim (\cN^g) = \dim (\cH) = n+1$. Hence $\cN^g = \mathrm{span}\{\cX^0, \ldots, \cX^n\}$.

\begin{proof}[Proof of Claim \ref{cl8}]
As $\cB$ is linear, it suffices to study its kernel and show that it reduces to $\{0\}$. If $\cV \in \ker \cB$, we have that the function $v$ introduced in \eqref{eqDefL} is identically equal to zero.

From Claim \ref{cl1.5}, we have that the sequence of functions $\cU_k(t, \theta) \definedas \cV(k+t, \theta) e^{-(k+t)}$ is uniformly bounded in $W^{1, p}(\cA, \bR)$.

Passing to a subsequence $(\cU_{\phi(k)})_k$, where $\phi$ is an increasing function, we can assume that $(\cU_{\phi(k)})_k$ converges strongly in $L^p(\cA, \bR)$ and weakly in $W^{1, p}(\cA, \bR)$ to some function $\cU_\infty$. But note that, from Claim \ref{cl6}, we have $\left\|\partial_1 \cU_{\phi(k)}\right\|_{L^p(\cA)} \to 0$. So $\partial_1 \cU_\infty = 0$ in the sense of distributions, meaning that $\cU_\infty$ does not depend on $t$. Returning to the definition of $h_u(s)$ (Equation \eqref{eqDefHU}), we have that, for any function $u \in L^c(\bS_1(0), \bR)$,
\[
0 = \lim_{k \to \infty} h_{u}(\phi(k)) = \int_{[-1, 1]} \int_{\bS_1(0)} \cU_\infty(t, \theta) u(\theta) d\theta dt.
\]
So we conclude that $\cU_\infty \equiv 0$, showing that
\[
\lim_{k \to \infty} f_0(\phi(k)) = \lim_{k \to \infty} \int_{-1}^1 \int_{\bS_1(0)} \left|\cU_{\phi(k)}(t, \theta)\right|^p d\theta dt = 0,
\]
where the function $f_0$ was defined in Equation \eqref{eqDefF0}. We now show that $\cV$ enjoys better decay than what we obtained in Claim \ref{cl1}. Note that the estimate we obtained in the proof of Claim \ref{cl1} was fairly crude. Now we can improve it to the following:
\begin{align*}
\left|\frac{d}{ds} f_0(s)\right|
& = \left|\int_{s-1}^{s+1} \int_{\bS_1(0)}  \partial_1 \left[\cV(r, \theta)^p e^{-pr}\right] d\theta dr\right|\\
& = p \left|\int_{s-1}^{s+1} \int_{\bS_1(0)}  (e^{-r} \cV)^{p-1} \partial_1 (\cV e^{-r}) d\theta dr\right|\\
& \leq p f_0(s)^{\frac{p-1}{p}} f_2(s)^{\frac{1}{p}}.
\end{align*}
This differential inequality can be rewritten as follows:
\[
    \left|\frac{d}{ds} f_0(s)^{1/p}\right| \leq f_2(s)^{\frac{1}{p}} \lesssim s e^{-2 s}.
\]
From the fact that $f_0(\phi(k)) \to 0$, we conclude that $f_0(s)$ tends to zero as $s$ goes to infinity and further that there exists a constant $\Lambda > 0$ such that
\[
f_0(s) \leq \Lambda s^p e^{- 2p s}.
\]
In particular,
\begin{equation}\label{eqFiniteness}
\int_{r_0+1}^\infty e^{ps} f_0(s) ds < \infty.
\end{equation}
We rewrite this as follows:
\begin{align*}
\int_{r_0+1}^\infty e^{ps} f_0(s) ds
&= \int_{s = r_0+1}^\infty e^{ps} \int_{r = s-1}^{s+1} \int_{\bS_1(0)} |\cV(r, \theta)|^p e^{-pr} d\theta dr ds\\
&= \int_{r = r_0}^\infty \int_{s = \max\{r_0+1, r-1\}}^{r+1} e^{ps} \int_{\bS_1(0)} |\cV(r, \theta)|^p e^{-pr} d\theta ds dr\\
&\geq \int_{r = r_0+2}^\infty \int_{s = r-1}^{r+1} e^{ps} \int_{\bS_1(0)} |\cV(r, \theta)|^p e^{-pr} d\theta ds dr
\end{align*}
where we have exchanged the order of integration with respect to the variables $s$ and $r$. Now the integral with respect to $s$ can be computed:
\[
\int_{s = r-1}^{r+1} e^{ps} ds = \frac{e^p-e^{-p}}{p} e^{pr},
\]
so
\begin{align*}
\int_{r_0+1}^\infty e^{ps} f_0(s) ds
&\geq \frac{e^p-e^{-p}}{p} \int_{r = r_0+2}^\infty \int_{\bS_1(0)} |\cV(r, \theta)|^p d\theta dr\\
&\geq \frac{e^p-e^{-p}}{p} \int_{\bH^n \setminus B_{r_0+2}(0)} |\cV|^p \sinh(r)^{1-n} d\mu^b.
\end{align*}
As a consequence, the estimate \eqref{eqFiniteness} implies that
\[
\int_{\bH^n \setminus B_{r_0+2}(0)} |\rho^{\frac{n-1}{p}}\cV|^p d\mu^b < \infty,
\]
i.e. that $\cV \in L^p_{-\frac{n-1}{p}}(M, \bR)$. In particular, if we choose $\tau' = -\frac{n-1}{p}$ and $q = p$ in Proposition \ref{propFredholm}, we have
\[
\left|\tau' + \frac{n-1}{p} - \frac{n-1}{2}\right| = \frac{n-1}{2} < \frac{n+1}{2}.
\]
So, the operator $P: u \mapsto -\Delta u + n u$ is an isomorphism from $W^{2, p}_{\tau'}(M, \bR)$ onto $L^p_{\tau'}(M, \bR)$. In particular, since, by elliptic regularity we have $\cV \in W^{2, p}_{\tau'}(M, \bR)$, we deduce that $\cV \in \ker(P) = \{0\}$. This ends the proof of Claim \ref{cl8} and of Proposition \ref{propApproximateLapse}.
\end{proof}

Our next task is to show how Proposition \ref{propApproximateLapse} allows us to identify the hyperbolic isometry in Theorem \ref{thmRigidity}.

\begin{proposition}\label{propAsymptoticIsometry}
There is a linear map $A_{\Phi_2 \circ \Phi_1^{-1}}: \cN^b \to \cN^b$ defined, for any $V \in \cN^b$, by
\[
    A_{\Phi_2 \circ \Phi_1^{-1}} (V) = V \circ \Phi_2 \circ \Phi_1^{-1} + \epsilon(V),
\]
where $\epsilon(V) \in W^{2, p}_{\tau-1}(\bH^n \setminus K'', \bR)$. Further $A_{\Phi_2 \circ \Phi_1^{-1}}$ corresponds to a hyperbolic isometry, i.e. to an element $B \in O_\uparrow(n, 1)$ such that for all $ V \in \cN^b$ we have $A_{\Phi_2 \circ \Phi_1^{-1}}(V) = V \circ B + \epsilon(V)$.
\end{proposition}

\begin{proof}
Lemma \ref{lmHarmonic} defines a linear map $A_{\Phi_1}: \cN^b \to \cN^g$, where $V \in \cN^b$ is mapped to the unique function $\cV \in \cN^g$ such that
$\cV = V \circ \Phi_1 + \lot$, with correction terms belonging to $W^{2, p}_{\tau-1}(M, \bR)$. Note that $A_{\Phi_1}$ is an injective map between two spaces of dimension $n+1$ so it is bijective. Similarly, we can define a map $A_{\Phi_2}: \cN^b \to \cN^g$. As a consequence, we get an invertible map $A_{\Phi_2 \circ \Phi_1^{-1}}: \cN^b \to \cN^b$ defined by
\[
A_{\Phi_2 \circ \Phi_1^{-1}} = A_{\Phi_1}^{-1} \circ A_{\Phi_2}.
\]
From the definition, we see that, given $V \in \cN^b$, $A_{\Phi_2 \circ \Phi_1^{-1}}(V)$ is the unique function $V' \in \cN^b$ such that $V' = V \circ \Phi_2 \circ \Phi_1^{-1} + \epsilon(V)$ with $\epsilon(V) \in W^{2, p}_{\tau-1}(\bH^n \setminus K'', \bR)$, where $K'' \definedas \bH^n \setminus \Phi_1(M \setminus (K_1 \cup K_2))$.

We claim that $A \definedas A_{\Phi_2\circ \Phi_1^{-1}}$ is an isometry of $\cN^b \simeq (\bR^{n, 1})^*$. Note that the lapse functions $V^\mu$, $\mu=0,\ldots, n$, satisfy
\[
    (V^0)^2 - \sum_{i=1}^n (V^i)^2 = 1
\]
as they are nothing but the restrictions of the coordinate functions of $\bR^{n, 1}$ to the hyperboloid, whereas for any $x \in \bH^n$, we have
\[
    1= -\eta(x, x)= (V^0(x))^2 - \sum_{i=1}^n (V^i(x))^2.
\]
Composing with $\Phi_2 \circ \Phi_1^{-1}$, we get
\[
(V^0(\Phi_2 \circ \Phi_1^{-1}(x)))^2 - \sum_{i=1}^n (V^i(\Phi_2 \circ \Phi_1^{-1}(x)))^2 = 1.
\]
Since for each $\mu$ we have $V^\mu \circ \Phi_2 \circ \Phi_1^{-1} = A(V^\mu) - \epsilon(V^\mu)$ with $\epsilon(V^\mu) \in W^{2, p}_{\tau-1}(\bH^n \setminus K'', \bR)$, we conclude that
\[
(A(V^0)(x))^2 - \sum_{i=1}^n (A(V^i)(x))^2 - 1 \in W^{2, p}_{\tau-2}(\bH^n \setminus K'', \bR).
\]
Now remark that $(A(V^0))^2 - \sum_{i=1}^n (A(V^i))^2$ is the restriction of a quadratic form over $\bR^{n, 1}$ to the hyperboloid. As the restriction of a polynomial function over $\bR^{n, 1}$ to $\bH^n$ cannot belong to $W^{2, p}_{\tau-2}(\bH^n \setminus K'', \bR)$ unless the  function is identically zero, we conclude that
\begin{equation}\label{eqQuadratic}
(A(V^0)(x))^2 - \sum_{i=1}^n (A(V^i)(x))^2 \equiv 1
\end{equation}
for any $x \in \bH^n \setminus K''$. By analyticity, it follows that \eqref{eqQuadratic} holds for all $x \in \bH^n$. Further, by homogeneity, we conclude that, identifying $\cN^b$ with $(\bR^{n, 1})^*$ (see Proposition \ref{propLapse}), we have
\begin{equation}\label{eqQuadratic2}
(A(X^0))^2 - \sum_{i=1}^n (A(X^i))^2 = \eta(X, X),
\end{equation}
where $(X^0, X^1, \ldots, X^n)$ denotes the standard basis of $(\bR^{n, 1})^*$.
Let $B: \bR^{n, 1} \to \bR^{n, 1}$ be the transpose of $A$, i.e. the linear map such that
\[
\forall \mu = 0, \ldots n,~A(X^\mu) = X^\mu \circ B.
\]
Equation \eqref{eqQuadratic2} transforms into
\[
\eta(BX, BX) = (X^0 \circ B)^2 - \sum_{i=1}^n (X^i \circ B)^2 = \eta(X, X),
\]
showing that $B$ is an isometry of $\bR^{n, 1}$ as claimed. The only thing left to be shown is that $B$ preserves time orientation. To prove this fact, note that $\eta$ induces a symmetric bilinear form $\eta^*$ on $\bR^{n, 1}$ defined by
\[
\eta^*(X^\mu, X^\nu) = \eta^{\mu\nu}.
\]
The signature of $\eta^*$ is the same as that of $\eta$. In particular $\eta^*(X^0, X^0) < 0$ so $X^0$ belongs to one of the two connected components of time-like linear forms (i.e. forms $\phi \in (\bR^{n, 1})^*$ such that $\eta(\phi, \phi) < 0$). As $X^0$ restricts to $V^0$ on $\bH^n$ and all elements of this connected component can be reached by multiplying $X^0$ by a positive factor and composing it with an element of $O_\uparrow(n, 1)$, all elements in that connected component induce positive functions on $\bH^n$, while elements in the other connected component of time-like linear forms induce negative functions on $\bH^n$ and space-like linear forms induce linear forms changing sign over some totally geodesic hypersurface.

As $A(X^0)$ restricts to $V^0 \circ \Phi_2 \circ \Phi_1^{-1} + \hot$ over $\bH^n \setminus K''$ and $V^0 \circ \Phi_2 \circ \Phi_1^{-1}$ is obviously positive on $\bH^n \setminus K''$, $A(X^0)$ restricts to a positive function outside some compact set of $\bH^n$, hence, by the previous classification of the lapse functions, $A(X^0)$ must belong to the same connected component of time-like forms as $X^0$. In particular,
\[
0 < A(X^0)\left(\frac{\partial}{\partial X^0}\right) = X^0 \left(B\left(\frac{\partial}{\partial X^0}\right)\right),
\]
showing that $B(\frac{\partial}{\partial X^0})$ remains a future-pointing vector (which must be time-like as $B$ is a Lorentzian isometry).
\end{proof}

We now finish the proof of Theorem \ref{thmRigidity} with the following lemma:

\begin{lemma}\label{lmDecay}
Assuming that $A_{\Phi_2 \circ \Phi_1^{-1}} = \id_{\bH^n}$, we can write $\Phi_2 \circ \Phi_1^{-1}(x) = \exp_x \zeta(x)$ for some vector field $\zeta \in W^{2, p}_\tau(\bH^n \setminus K'', T\bH^n)$.
\end{lemma}

In the course of the proof of this lemma, we will need a result that we state here:

\begin{lemma}\label{lmCompSobolev}
For any symmetric $2$-tensor $T$ defined on $M \setminus (K_1 \cup K_2)$, we have
\begin{align*}
&\int_{M \setminus (K_1 \cup K_2)} (\rho \circ \Phi_1)^{-p\tau} \left[|\nabla T|_g^p + |T|_g^p \right] d\mu^g < \infty\\
& \qquad \Leftrightarrow 
\int_{\bH^n \setminus K''} \rho^{-p\tau} \left[|D ((\Phi_1)_*T)|_b^p + |((\Phi_1)_*T)|_b^p \right] d\mu^b < \infty.
\end{align*}
\end{lemma}

\begin{proof}
Note that the metric $g$ (or more exactly $(\Phi_1)_* g$) is uniformly equivalent to the metric $b$, so 
\begin{align*}
\int_{M \setminus (K_1 \cup K_2)} (\rho \circ \Phi_1)^{-p\tau} |T|_g^p d\mu^g
& \lesssim \int_{\bH^n \setminus K''} \rho^{-p\tau} |T|_b^p d\mu^b\\
\text{and}\quad 
\int_{\bH^n \setminus K''} \rho^{-p\tau} |T|_b^p d\mu^b
& \lesssim \int_{M \setminus (K_1 \cup K_2)} (\rho \circ \Phi_1)^{-p\tau} |T|_g^p d\mu^g,
\end{align*}
where we have chosen, here and in what follows, to write $T$ instead of $((\Phi_1)_*T)$ to avoid overloading the notation.
Further, note that $\nabla T = D T + \Gamma \star T$, where $\Gamma \star T$ denotes a contraction of the Christoffel symbols as defined in \eqref{eqChristoffel} and the tensor $T$. Hence,
\begin{align*}
\int_{M \setminus (K_1 \cup K_2)} (\rho \circ \Phi_1)^{-p\tau} |\nabla T|_g^p d\mu^g
& \lesssim \int_{\bH^n \setminus K''} \rho^{-p\tau} |\nabla T|_b^p d\mu^b\\
& \lesssim \int_{\bH^n \setminus K''} \rho^{-p\tau} \left[|D T|_b^p + |\Gamma \star T|_b^p\right] d\mu^b\\
& \lesssim \int_{\bH^n \setminus K''} \rho^{-p\tau} \left[|D T|_b^p + |De|^p_b |T|_b^p\right] d\mu^b.
\end{align*}
Assuming that the $W^{1, p}_\tau$-norm of $T$ defined with respect to the metric $b$ is finite:
\[
\int_{\bH^n \setminus K''} \rho^{-p\tau} \left[|D T|_b^p + |T|_b^p \right] d\mu^b < \infty,
\]
we have, by the Sobolev injection, that $T \in L^\infty(\bH^n\setminus K'', S_2\bH^n)$. Hence,
\begin{align*}
\int_{M \setminus (K_1 \cup K_2)} (\rho \circ \Phi_1)^{-p\tau} |\nabla T|_g^p d\mu^g
& \lesssim \int_{\bH^n \setminus K''} \rho^{-p\tau} |D T|_b^p d\mu^b\\
&\qquad + 
\|T\|_{L^\infty(\bH^n\setminus K'', S_2 M)}\int_{\bH^n \setminus K''} \rho^{-p\tau} |De|^p_b d\mu^b\\
& < \infty.
\end{align*}
Hence we have shown the following:
\begin{align*}
&\int_{\bH^n \setminus K''} (\rho \circ \Phi_1)^{-p\tau} \left[|D T|_b^p + |T|_b^p \right] d\mu^b < \infty\\
&\qquad \Rightarrow
\int_{M \setminus (K_1 \cup K_2)} \rho^{-p\tau} \left[|\nabla T|_g^p + |T|_g^p \right] d\mu^g < \infty.
\end{align*}
The converse implication is proven similarly.
\end{proof}

This lemma has the following important consequence. Applied to the chart $\Phi_2$ instead of $\Phi_1$, we have, due to the assumptions of Theorem \ref{thmRigidity},
\[
\int_{M \setminus (K_1 \cup K_2)} (\rho \circ \Phi_2)^{-p\tau} \left[|\nabla (\Phi_2)_*e_2|_g^p + |(\Phi_2)_* e_2|_g^p \right] d\mu^g < \infty.
\]
Note that we can safely replace $\rho \circ \Phi_2$ by $\rho \circ \Phi_1$ in this formula. Indeed, as explained in \cite{GicquaudCompactification}, both functions are uniformly comparable to $e^{-r}$, where $r$ is the distance function (with respect to the metric $g$) to any compact subset $K \neq \emptyset$ of $M$: in other words, there exists a constant $\Lambda > 0$ such that
\[
\Lambda^{-1} e^{-r(x)} \leq (\rho \circ \Phi_1)(x) \leq \Lambda e^{-r(x)}
~
\text{and}~
\Lambda^{-1} e^{-r(x)} \leq (\rho \circ \Phi_2)(x) \leq \Lambda e^{-r(x)}
.
\]
Hence,
\[
\int_{M \setminus (K_1 \cup K_2)} (\rho \circ \Phi_1)^{-p\tau} \left[|\nabla (\Phi_2)_*e_2|_g^p + |(\Phi_2)_* e_2|_g^p \right] d\mu^g < \infty.
\]
Applying once again Lemma \ref{lmCompSobolev}, we conclude that
\[
\int_{\bH^n \setminus K''} \rho^{-p\tau} \left[|D (\Phi_2 \circ \Phi_1^{-1})_*e_2|_b^p + |(\Phi_2 \circ \Phi_1^{-1})_* e_2|_b^p \right] d\mu^b < \infty.
\]
So, $(\Phi_2 \circ \Phi_1^{-1})_*e_2 \in W^{1, p}_\tau(\bH^n \setminus K'', S_2\bH^n)$.

As a consequence, setting $\btil \definedas (\Phi_2 \circ \Phi_1^{-1})_* b$, we have
\begin{align*}
\btil - b
& = (\Phi_2 \circ \Phi_1^{-1})_* b - b\\
& = (\Phi_2)_* ((\Phi_1)^* b - g) + (\Phi_2)_* g - b\\
& = (\Phi_2)_* (\Phi_1)^* (b - (\Phi_1)_*g) + e_2\\
& = -(\Phi_2 \circ \Phi_1^{-1})_* e_1 + e_2.
\end{align*}
As both terms in the last line belong to $W^{1, p}_\tau(\bH^n \setminus K'', S_2 \bH^n)$, we have shown
\begin{equation}\label{eqDiffMetrics}
\btil - b \in W^{1, p}_\tau(\bH^n \setminus K'', S_2 \bH^n).
\end{equation}

\begin{proof}[Proof of Lemma \ref{lmDecay}]
Assuming that $A_{\Phi_2 \circ \Phi_1^{-1}} = \id_{\bH^n}$, by Proposition \ref{propAsymptoticIsometry} we can write 
\[
    V^{\mu}\circ \Phi_2\circ \Phi_1^{-1} - V^{\mu} \in W^{2, p}_{\tau-1}(\bH^n \setminus K'', \bR).
\]
Now recall that the functions $V^\mu$, $\mu=0,\ldots,n$, provide a full set of coordinates for the standard embedding of the hyperbolic space into Minkowski spacetime. For any $x \in \bH^n \setminus K''$, we denote $X^\mu(x) = V^\mu(x)$ and $Y^\mu = V^{\mu}((\Phi_2\circ \Phi_1^{-1})(x))$.

Applying Proposition \ref{propExponentialXY-intrinsic} (see the comment after the Proposition) we find that $\Phi_2\circ \Phi_1^{-1}(x) = \exp_x \zeta(x) $ where $\zeta(x)$ is given by
\[
\zeta(x) = f(u) \left(U - u X\right),
\]
with $U = Y - X$, $u= \eta(U, U)/2$ and $\displaystyle f(u) = \frac{\ln(1+u + \sqrt{2 u + u^2})}{\sqrt{2 u + u^2}}$. We have $U \in W^{2, p}_{\tau-1}(\bH^n \setminus K'', \bR)$. As a consequence, $u \in W^{2, p}_{2(\tau-1)}(\bH^n \setminus K'', \bR)$. As $p > n$ and $\tau + \frac{n-1}{p} \geq 2$, we have $\tau-1 > 0$, showing that $u \in L^\infty(\bH^n\setminus K'', \bR)$. As $f$ is analytic, we have $f(u) - 1 \in W^{2, p}_{2(\tau-1)}(\bH^n \setminus K'', \bR)$. All in all, we see that $\zeta \in W^{2, p}_{\mu_0}(\bH^n \setminus K'', \bR^{n, 1})$ with ${\mu_0} = \min\{\tau-1, 2 \tau-3\}$. Since we assumed $\tau > 3/2$, we have ${\mu_0} > 0$. Applying Proposition \ref{propTransfert}, we conclude that $\zeta \in W^{2, p}_{\mu_0}(\bH^n \setminus K'', \bH^n)$.

We will now show how to improve this rough estimate for $\zeta$. First, from Proposition \ref{propTransfert}, we get $\zeta \in W^{2, p}_{\mu_0}(\bH^n \setminus K'', T\bH^n)$. Further, from \eqref{eqDiffMetrics}, we have $\btil - b \in W^{1, p}_\tau(\bH^n\setminus K'', S_2\bH^n)$ . We now use Equation \eqref{eqPullback2}. As we have ${\mu_0} > 0$, we have
\[
    c \definedas c(|\zeta|^2) = \cosh(|\zeta|) = 1 + O(|\zeta|^2), \quad
    s \definedas s(|\zeta|^2) = \frac{\sinh(|\zeta|)}{|\zeta|} = 1 + O(|\zeta|^2).
\]
So $c-1, s-1 \in W^{2, p}_{2\mu_0}(\bH^n\setminus K'', \bR)$. It then follows from the embedding
\[
W^{1, p}_{\mu_0}(\bH^n \setminus K'', T\bH^n) \hookrightarrow L^\infty_{\mu_0}(\bH^n \setminus K'', T\bH^n)
\]
and Equation \eqref{eqPullback2} which we rewrite as 
\begin{align*}
\lie_{\zeta} b
&= \frac{1}{c s} (\btil - c^2 b) - \frac{s}{c} \left(b(D_{\cdot} \zeta, D_{\cdot} \zeta) + \zeta^\flat \otimes \zeta^\flat\right)\\
& \qquad - \frac{1-s^2}{4 c s|\zeta|^2} d |\zeta|^2 \otimes d |\zeta|^2 - \frac{1-cs}{2 cs |\zeta|^2}\left[\zeta^\flat \otimes d|\zeta|^2 + d|\zeta|^2 \otimes \zeta^\flat \right],
\end{align*}
that we have
\[
    \lie_\zeta b = \btil - b + \upsilon,
\]
where $\upsilon \in W^{1, p}_{2\mu_0}(\bH^n\setminus K'', S_2\bH^n)$. We now take the divergence with respect to $b$ of the previous equation:
\[
    \divg (\lie_\zeta b) = \divg(\btil - b) + \divg(\upsilon) \in L^p_{\min\{\tau, 2\mu_0\}}(\bH^n \setminus K'').
\]
Now note that the operator $\zeta \mapsto \divg (\lie_\zeta b)$ is elliptic and satisfies the assumptions of \cite[Theorem C]{LeeFredholm} with indicial radius $R = \frac{n+1}{2}$ \footnote{The calculation of the indicial radius of the operator $\zeta \mapsto \divg (\rlie_\zeta b)$, i.e. the vector Laplacian, can be found in \cite[Proposition G]{LeeFredholm} or \cite[Lemme 6.1]{Gicquaud}, it is easy to check that the replacement of $\lie_\zeta b$ by  $\rlie_\zeta b$ does not change the indicial roots.}. In particular, since $\zeta \in W^{2, p}_{\mu_0}(\bH^n\setminus K'', T\bH^n)$, with ${\mu_0} > 0$, we have
\[
    {\mu_0} + \frac{n-1}{p} - \frac{n-1}{2} > -\frac{n+1}{2}.
\]
As a consequence, \cite[Proposition 6.5]{LeeFredholm} implies that $\zeta \in W^{2, p}_{\mu_1}(\bH^n\setminus K'', T\bH^n)$ with $\mu_1 = \min\{2\mu_0, \tau\} > \mu_0$. We can now repeat the argument using the improved estimate we have for $\zeta$ and get that $\upsilon \in W^{1, p}_{2\mu_1}(\bH^n\setminus K'', S_2\bH^n)$ showing that $\zeta \in W^{2, p}_{\mu_2}(\bH^n\setminus K'', T\bH^n)$, with $\mu_2 = \min\{2\mu_1, \tau\}$. After a finite number of steps, we conclude that $\zeta \in W^{2, p}_\tau (\bH^n\setminus K'', T\bH^n)$ as claimed.
\end{proof}

\begin{proof}[End of the proof of Theorem \ref{thmRigidity}]
From Proposition \ref{propAsymptoticIsometry}, we know that there exists an element $B \in O_\uparrow(n, 1)$ such that, identifying $B$ with an isometry of the hyperbolic space, we have, for any $V \in \cN_b$, 
\[
V \circ B = V\circ \Phi_2 \circ \Phi_1^{-1} + \epsilon(V),
\]
where $\epsilon(V) \in W^{2, p}_{\tau-1}(\bH^n\setminus K'', \bR)$. But $V \mapsto V \circ B^{-1}$ is an automorphism of $\cN^b$, so we can replace $V$ by $V \circ B^{-1}$ in the previous estimate:
\[
V = V\circ (B^{-1} \circ \Phi_2) \circ \Phi_1^{-1} + \epsilon(V \circ B^{-1}),
\]
where $\epsilon(V \circ B^{-1}) \in W^{2, p}_{\tau-1}(\bH^n \setminus K'')$. As a consequence, we can use Lemma \ref{lmDecay} with $\Phi_2$ replaced by $B^{-1}\circ \Phi_2$:
\[
(B^{-1}\circ \Phi_2) \circ \Phi_1^{-1}(x) = \exp_x \zeta(x)
\]
for some vector field $\zeta \in W^{2, p}_\tau(\bH^n \setminus K'', T\bH^n)$. Thus
\[
\Phi_2 \circ \Phi_1^{-1}(x) = B(\exp_x \zeta(x))
\]
with $\zeta \in W^{2, p}_\tau(\bH^n\setminus K'', T\bH^n)$.
\end{proof}

\subsection{Covariance of the mass and the mass aspect function}\label{secCovariance}
We are now equipped to prove the covariance of the mass aspect function. The strategy here is basically the same as the one outlined in \cite{MichelMass}, yet we are going to be more cautious about how we handle diffeomorphisms.

Assume given two charts at infinity $\Phi_1: M \setminus K_1 \to \bH^n \setminus K_1'$ and $\Phi_2: M \setminus K_2 \to \bH^n \setminus K_2'$ and set $e_1 \definedas (\Phi_1)_* g - b$ (resp. $e_2 \definedas (\Phi_2)_* g - b$). Assume further that $e_1, e_2 \in W^{1, p}_\tau(\bH^n \setminus K'', S_2\bH^n)$, where $p > n$ and $\tau > 3/2$ satisfies
\[
\tau + \frac{n-1}{p} \in
\begin{cases}
    [2, n)                      & \quad \text{if } n = 3, \\
    \left(\frac{n}{2}, n\right) & \quad\text{if } n > 3,
\end{cases}
\]
so that the assumptions of Theorem \ref{thmRigidity} are fulfilled (these assumptions will be made more restrictive in what follows). Then we can write $\Phi_{12} \definedas \Phi_2 \circ \Phi_1^{-1} = A \circ \Phi_0$, where $A$ is an isometry of the hyperbolic space and $\Phi_0(x) = \exp_x (\zeta(x))$ is a diffeomorphism asymptotic to the identity. We note that the lower bound
\[
    \tau + \frac{n-1}{p} > \frac{n}{2}
\]
is valid in dimension $3$ implying, by \cite[Lemma 3.6(b)]{LeeFredholm}, that $W^{1, p}_\tau(\bH^n \setminus K'', S_2\bH^n) \hookrightarrow W^{1, 2}_{1/2}(\bH^n \setminus K'', S_2\bH^n)$ (see \cite[Lemma 3.6(b)]{LeeFredholm}), which  guarantees that $g$ is asymptotically hyperbolic of order $1/2$ as in Definition \ref{defAH}.

To avoid repeating the formulas for the charge integrals, we denote the integrands in \eqref{eqChargeCH} and in \eqref{eqChargeAspect} by
\begin{equation}\label{eqChargeIntegrand}
\begin{aligned}
    \bU(b, V, e, \chibar_k)
    &\definedas V \left(\divg_b(e) - d \tr_b(e)\right)(-D^b \chibar_k)\\
    &\qquad + \tr_b(e) dV(-D^b \chibar_k) - e(D^b V, -D^b \chibar_k),
\end{aligned}
\end{equation}
where we highlight the reference metric $b$ as it will subsequently be modified by diffeomorphisms, and where $V \in C^\infty_{-1}(\bH^n, \bR)$ is an eigenfunction of the Laplacian as constructed in Proposition \ref{propEstimateEigenfunction2}. So the mass integrals are given by
\[
    p(e_1, V) = \lim_{k \to \infty} \int_{\bH^n} \bU(b, V, e_1, \chibar_k) d\mu^b
    \quad\text{and}\quad
    p(e_2, V) = \lim_{k \to \infty} \int_{\bH^n} \bU(b, V, e_2, \chibar_k) d\mu^b.
\]
Since $g = (\Phi_1^{-1})_* (b + e_1)$, we have
\[
e_2 = (\Phi_{12})_* (b + e_1) - b = (\Phi_{12})_* (b + e_1 - \Phi_{12}^*b).
\]
As a consequence, by invariance of the integral giving $p(e_2, V)$ under the diffeomorphism $\Phi_{12}$, we can write
\begin{align*}
    p(e_2, V)
     & = \lim_{k \to \infty} \int_{\bH^n} \bU(b, V, e_2, \chibar_k) d\mu^b                                                                               \\
     & = \lim_{k \to \infty} \int_{\bH^n} \bU(\Phi_{12}^* b, \Phi_{12}^* V, \Phi_{12}^* e_2, \Phi_{12}^* \chibar_k) d\mu^{\Phi_{12}^* b}                 \\
     & = \lim_{k \to \infty} \int_{\bH^n} \bU(\Phi_{12}^* b, V \circ \Phi_{12}, e_1 + b - \Phi_{12}^*b, \chibar_k \circ \Phi_{12}) d\mu^{\Phi_{12}^* b}.
\end{align*}
As the limit is independent of the choice of the cutoff functions $\chibar_k$ we choose, we can immediately replace $\chibar_k \circ \Phi_{12}$ by $\chibar_k$ in the last expression:
\begin{equation}\label{eqP2}
    p(e_2, V)
    = \lim_{k \to \infty} \int_{\bH^n} \bU(\Phi_{12}^* b, V \circ \Phi_{12}, e_1 + b - \Phi_{12}^*b, \chibar_k) d\mu^{\Phi_{12}^* b}.
\end{equation}

Now remark that, upon introducing the ``intermediate'' chart at infinity $\Phi_{1.5} \definedas \Phi_0 \circ \Phi_1$ so that $\Phi_{1.5} \circ \Phi_1^{-1} = \Phi_0$ and $\Phi_2 \circ \Phi_{1.5}^{-1}= A$, we can split the discussion of the covariance of the mass aspect into two cases:

$\bullet$ The case $\Phi_{12} = A \in O_{\uparrow}(n, 1)$, for which $\Phi_{12}^* b = b$: In this case, Equation \eqref{eqP2} immediately yields
\[
    p(e_2, V)
    = \lim_{k \to \infty} \int_{\bH^n} \bU(b, V \circ \Phi_{12}, e_1, \chibar_k) d\mu^b
    = p(e_1, V \circ A).
\]

$\bullet$ The case $\Phi_{12} = \Phi_0$ is a diffeomorphism asymptotic to the identity: This case is more complex. Our first task is to prove that we have
\[
    p(e_2, V)
    = \lim_{k \to \infty} \int_{\bH^n} \bU(b, V, e_1 - \cL_\zeta b, \chibar_k) d\mu^b,
\]
removing most of the occurences of $\Phi_{12}^*b$. First note that $b - \Phi_{12}^*b \in W^{1, p}_\tau(\bH^n \setminus K'', S_2\bH^n) \subset W^{1, 2}_{1/2 + \epsilon}(\bH^n \setminus K'', S_2\bH^n)$ for some constant $\epsilon > 0$. Indeed, as we have
\[
    \tau + \frac{n-1}{p} > \frac{1}{2} + \frac{n-1}{2},
\]
we can let $\epsilon > 0$ be so small that the inequality
\[
    \tau + \frac{n-1}{p} > \frac{1}{2} + \epsilon + \frac{n-1}{2}
\]
remains true. Hence, from \cite[Lemma 3.6(b)]{LeeFredholm}, we obtain that
\[
    W^{1, p}_\tau(\bH^n \setminus K'', S_2\bH^n) \subset W^{1, 2}_{1/2 + \epsilon}(\bH^n \setminus K'', S_2\bH^n).
\]
As we also have, from the Sobolev injection \cite[Lemma 3.6(c)]{LeeFredholm},
\[
    W^{1, p}_\tau(\bH^n \setminus K'', S_2\bH^n) \subset L^\infty(\bH^n \setminus K'', S_2\bH^n),
\]
we see that the error terms can be estimated as follows:
\begin{align*}
     & \int_{\bH^n} \left|\bU(\Phi_{12}^* b, V \circ \Phi_{12}, e_1 + b - \Phi_{12}^*b, \chibar_k) \frac{d\mu^{\Phi_{12}^* b}}{d\mu^b} - \bU(b, V \circ \Phi_{12}, e_1 + b - \Phi_{12}^*b, \chibar_k)\right| d\mu^b \\
     & \qquad \lesssim \left\|e_1 + b - \Phi_{12}^*b\right\|_{W^{1, 2}_{1/2}} \|\Phi_{12}^* b - b\|_{W^{1, 2}_{1/2}(\supp D \chibar_k)}                                                                             \\
     & \qquad \lesssim \left\|e_1 + b - \Phi_{12}^*b\right\|_{W^{1, 2}_{1/2}} \|\Phi_{12}^* b - b\|_{W^{1, 2}_{1/2 + \epsilon}} \max_{x \in \supp D\chibar_k} \rho^\epsilon(x)                                      \\
     & \qquad \to_{k \to \infty} 0.
\end{align*}
Next, it follows from Formula \eqref{eqPullback2} that $b - \Phi_{12}^*b = - \cL_\zeta b + O(|\zeta|^2 + |D\zeta|^2)$. So, by an argument similar to the previous one, we conclude that
\begin{align*}
     & \int_{\bH^n} \left|\bU(b, V \circ \Phi_{12}, e_1 + b - \Phi_{12}^*b, \chibar_k) - \bU(b, V \circ \Phi_{12}, e_1 - \cL_\zeta b, \chibar_k)\right| d\mu^b \\
     & \qquad \to_{k \to \infty} 0.
\end{align*}
Replacing $V \circ \Phi_{12}$ by $V$ in \eqref{eqP2} is allowed by the fact that $V \circ \Phi_{12} - V \in W^{1, p}_{\delta-1}(\bH^n \setminus K'', \bR)$. The details of the argument are left to the reader. All in all, we have proven the following:

\begin{align*}
     & \int_{\bH^n} \left|\bU(\Phi_{12}^* b, V \circ \Phi_{12}, e_1 + b - \Phi_{12}^*b, \chibar_k) \frac{d\mu^{\Phi_{12}^* b}}{d\mu^b} - \bU(b, V \circ \Phi_{12}, e_1 - \cL_\zeta b, \chibar_k)\right| d\mu^b \\
     & \qquad \to_{k \to \infty} 0.
\end{align*}
As a consequence, we have proven that, under the assumption that $e_1, e_2 \in W^{1, p}_\tau(\bH^n \setminus K'', S_2\bH^n)$, and $\Phi_{12}$ is asymptotic to the identity, then if the limit
\[
    p(e_2, V)
    = \lim_{k \to \infty} \int_{\bH^n} \bU(\Phi_{12}^* b, V \circ \Phi_{12}, e_1 + b - \Phi_{12}^*b, \chibar_k) d\mu^{\Phi_{12}^* b}
\]
exists, we have
\begin{equation}\label{eqAlmostThere2}
    p(e_2, V)
    = \lim_{k \to \infty} \int_{\bH^n} \bU(b, V, e_1 - \cL_\zeta b, \chibar_k) d\mu^b.
\end{equation}

Finally, we would like to get rid of the Lie derivative $\cL_\zeta b$ in the above formula. This is essentially Bartnik's celebrated ``curious cancellation'' (see \cite{BartnikMass}), that is explained in \cite{MichelMass}. Let $W=W_k$ denote the vector field defined by
\[
    (W_k)_i = S_{ij} (-D\chibar_k)^j, \quad\text{with}~ S_{ij} = (D_i \zeta_j - D_j \zeta_i) V + 2 \zeta_i D_j V - 2 \zeta_j D_i V.
\]
Note that $S_{ij}$ is antisymmetric. Hence,
\begin{align*}
    \divg(W)
    &= D^i W_i\\
    &= (D^iS_{ij}) (-D\chibar_k)^j - S_{ij} \hess^{ij} \chibar_k\\
    &= (D^iS_{ij}) (-D\chibar_k)^j\\
    &= \left[V D^i(D_i \zeta_j - D_j \zeta_i) + D^i V (D_i \zeta_j - D_j \zeta_i)\right] (-D\chibar_k)^j\\
    &\qquad\qquad + 2 \left[D^i \zeta_i D_j V + \zeta^i \hess_{ij} V - D_i V D^i\zeta_j - \zeta_j \Delta V\right] (-D\chibar_k)^j\\
    &= \left[V D^i ( D_i \zeta_j + D_j \zeta_i) - 2 VD^i D_j \zeta_i + D^i V (D_i \zeta_j - D_j \zeta_i)\right] (-D\chibar_k)^j\\
    &\qquad\qquad + 2 \left[D^i \zeta_i D_j V + \zeta^i \cT_{ij} - D_i V D^i\zeta_j\right] (-D\chibar_k)^j - 2(n-1) V \zeta_j,
\end{align*}
where, as before, $\cT \definedas \hess V - V b$. Since
\begin{align*}
    2 D^i D_j \zeta_i
    &= 2 D_j D^i \zeta_i - 2\riemudud{l}{i}{i}{j} \zeta_l\\
    &= D_j (\tr(\cL_\zeta b)) + 2 \ricud{l}{j} \zeta_l\\
    &= D_j (\tr(\cL_\zeta b)) - 2 (n-1) \zeta_j,
\end{align*}
we obtain from Equation \eqref{eqChargeIntegrand} that
\begin{align*}
    D^i W_i
    &= \left[V D^i (\cL_\zeta b)_{ij} - V D_j (\tr(\cL_\zeta b)) - D^i V (D_i \zeta_j + D_j \zeta_i)\right] (-D\chibar_k)^j\\
    &\qquad\qquad + \left[\tr(\cL_\zeta b) D_j V + 2 \zeta^i \cT_{ij}\right] (-D\chibar_k)^j\\
    &= \bU(b, V, \cL_\zeta b, \chibar_k) + 2 \zeta^i \cT_{ij} (-D\chibar_k)^j.
\end{align*}
As $W$ has compact support, the integral of the divergence of $W$ vanishes. As a consequence, we conclude that Equation \eqref{eqAlmostThere2} becomes
\begin{align}
    p(e_2, V)
    &= \lim_{k \to \infty} \int_{\bH^n} \left[\bU(b, V, e_1, \chibar_k) - \bU(b, V, \cL_\zeta b, \chibar_k)\right] d\mu^b\nonumber\\
    &= \lim_{k \to \infty} \int_{\bH^n} \left[\bU(b, V, e_1, \chibar_k) - \divg(W) + 2 \cT(\zeta, -D\chibar_k)\right] d\mu^b\nonumber\\
    &= \lim_{k \to \infty} \int_{\bH^n} \left[\bU(b, V, e_1, \chibar_k) + 2 \cT(\zeta, -D\chibar_k)\right] d\mu^b\label{eqDiffCharge}.
\end{align}
If $V \in \cN$, we have $\cT \equiv 0$ and we obtain
\begin{equation}\label{eqEquality}
p(e_2, V) = p(e_1, V).
\end{equation}
In summary, we have proven:

\begin{theorem}\label{thmCovariance}
Assume given two charts at infinity $\Phi_1$, $\Phi_2$ for $(M, g)$ such that $e_1 \definedas \Phi_{1*} g - b \in W^{1, p}_\tau(\bH^n \setminus K'_1)$ and  $e_2 \definedas \Phi_{2*} g - b \in W^{1, p}_\tau(\bH^n \setminus K'_2)$, where $p > n$ and $\tau > 3/2$ satisfy
\[
\tau + \frac{n-1}{p} \in
\begin{cases}
    [2, n)                      & \quad \text{if } n = 3, \\
    \left(\frac{n}{2}, n\right) & \quad \text{if } n > 3.
\end{cases}
\]
Set, as in the statement of Theorem \ref{thmRigidity}, $\Phi_2 \circ \Phi_1^{-1} = A \circ \Phi_0$ where $A \in O_\uparrow(n, 1)$ is an isometry of $\bH^n$ and $\Phi_0$ is a diffeomorphism asymptotic to the identity. Then we have, for any $V \in \cN$
\[
    p(e_2, V) = p(e_1, V \circ A).
\]
\end{theorem}

We will now focus on proving the covariance of the mass aspect function whenever the later is defined, see Proposition \ref{propMassAspect}. To this end, we need to prove that the $\cT$-term in \eqref{eqDiffCharge} tends to zero as $k$ goes to infinity. We start by a refinement of Lemma \ref{lmDecay}. 

\begin{lemma}\label{lmDecay2}
Under the assumptions of Theorem \ref{thmCovariance}, assume further that $e_1, e_2 \in L^1_\delta(\bH^n \setminus (K'_1 \cup K'_2), S_2\bH^n)$ for some $\delta \in (-1, 1]$. If the charts $\Phi_1$ and $\Phi_2$ satisfy $\Phi_2 \circ \Phi_1^{-1}(x) = \exp_x \zeta(x)$ for some $\zeta \in W^{2, p}_\tau(\bH^n \setminus K'')$, where $\tau$ is as in Theorem \ref{thmCovariance}, we have $\zeta \in W^{1, 1}_{\delta'}(\bH^n \setminus K'', T\bH^n)$ for any $\delta' < \delta$.
\end{lemma}

\begin{proof}
As $W^{1, p}_\tau(\bH^n \setminus K'_1, S_2\bH^n) \hookrightarrow L^\infty_\tau(\bH^n \setminus K'_1, S_2\bH^n)$, we have
\[
e_1 \in L^\infty_\tau(\bH^n \setminus K'_1, S_2\bH^n) \cap L^1_\delta(\bH^n \setminus K'_1, S_2\bH^n).
\]

By interpolation, we deduce that $e_1 \in L^q_\mu(\bH^n \setminus K'_1, S_2\bH^n)$ for any $q \in (1, \infty)$ and $\displaystyle \mu = \frac{1}{q} \delta + \left(1 - \frac{1}{q}\right) \tau$. Indeed, for $\lambda = 1/q$, we have
\begin{align*}
\|e_1\|_{L^q_{\mu}}^q
& = \int_{\bH^n \setminus K'_1} \rho^{-q\mu} |e_1|^q d\mu^{b}\\
& = \int_{\bH^n \setminus K'_1} \rho^{-(q-1)\tau - \delta} |e_1|^{q-1} |e_1| d\mu^{b}\\
& \leq \left(\sup_{\bH^n \setminus K'_1} \rho^{-\tau} |e_1|\right)^{q-1}\int_{\bH^n \setminus K'_1} \rho^{- \delta} |e_1| d\mu^{b}\\
& = \|e_1\|_{L^\infty_{\tau}}^{q-1} \|e_1\|_{L^1_{\delta}}\\
& < \infty.
\end{align*}
Similarly, $\|e_2\|_{L^q_\mu}(\bH^n \setminus K'_2, S_2\bH^n) < \infty$. Reasoning as in the proof of Lemma \ref{lmDecay}, we have
\[
\cL_\zeta b = \btil - b + v.
\]
Here, however, note that the arguments that led to \eqref{eqDiffMetrics} can be straightforwardly modified to yield $\btil - b \in L^q_\mu(\bH^n \setminus K'', S_2\bH^n)$.

Further, due to the fact that $v$ is quadratic in $\zeta$ and $D \zeta$ with
\[
\zeta , D\zeta \in L^p_\tau(\bH^n \setminus K'', S_2\bH^n) \bigcap L^\infty_\tau(\bH^n \setminus K'', S_2\bH^n),
\]
we deduce that $v \in L^{p/2}_{2\tau} (\bH^n \setminus K'', S_2\bH^n)$. As $\tau + \frac{n-1}{p} > \frac{n}{2}$, we have $2\tau + \frac{2(n-1)}{p} > n$.
On the other hand, since $\delta + n-1 \leq n$, we conclude that
\[
\mu + \frac{n-1}{q} = \tau \left(1 - \frac{1}{q}\right) + \frac{\delta+n-1}{q}
\leq n.
\]
Thus, we have $v \in L^q_\mu (\bH^n \setminus K'', S_2\bH^n)$ as long as $q \leq p/2$. All in all, we conclude that, as long as $q \leq p/2$, we have
\[
\cL_\zeta b \in L^q_\mu(\bH^n \setminus K'', S_2\bH^n).
\]
Taking the divergence, we have $\divg (\cL_\zeta b) \in W^{-1, q}_\mu(\bH^n \setminus K'', S_2\bH^n)$, where the negative exponent Sobolev spaces are defined in \cite[Section 3.1]{AllenLeeMaxwell}\footnote{We chose not to give a precise definition of these function spaces, as they are only used in this proof.}. From \cite[Lemma B.2]{AllenLeeMaxwell}, we conclude that $\zeta \in W^{1, q}_\mu(\bH^n \setminus K'', S_2\bH^n)$.

The last step is to convert this estimate for $\zeta$ into a $W^{1, 1}$-estimate. We fix an arbitrary $\delta'<\delta$ and use again \cite[Lemma 3.6]{LeeFredholm}: the embedding
\[
W^{1, q}_\mu (\bH^n \setminus K'_1\cup K'_2, S_2\bH^n) \hookrightarrow W^{1, 1}_{\delta'}(\bH^n \setminus K'', S_2\bH^n)
\]
holds true provided that $q > 1$, which we already assumed, and
\[
\mu + \frac{n-1}{q} > \delta' + \frac{n-1}{1}.
\]
However, we have
\[
\mu + \frac{n-1}{q}
= \frac{1}{q} \delta + \left(1 - \frac{1}{q}\right) \tau + \frac{n-1}{q}
= \frac{1}{q} \left(\delta + \frac{n-1}{1}\right) + \left(1 - \frac{1}{q}\right) \tau.
\]
Hence, upon choosing $q$ close enough to $1$, we can ensure that
\[
\mu + \frac{n-1}{q} > \delta' + \frac{n-1}{1}.
\]
This concludes the proof of the fact that $\zeta \in W^{1, 1}_{\delta'}(\bH^n \setminus K'', S_2\bH^n)$.
\end{proof}

We are now in a position to prove the covariance of the mass aspect function. It turns out that we cannot prove covariance for all functions $V \in \cE^\alpha$, because of the slight loss of weight we had to impose in Lemma \ref{lmStrongerDecay}, but only for a restricted class $\cE^\alpha_0$ of eigenfunctions $V$ that we now introduce.

The little H\"older space $c^\alpha(\bS_1(0), \bR)$, $\alpha \geq 0$, is defined as the closure of $C^\infty(\bS_1(0), \bR)$ in $C^\alpha(\bS_1(0), \bR)$. For non-integer $\alpha$, this space is strictly contained in $C^\alpha(\bS_1(0), \bR)$. For example, if $\alpha \in (0, 1)$, $c^\alpha(\bS_1(0), \bR)$ is defined as follows:
\[
f \in c^\alpha(\bS_1(0), \bR) \Leftrightarrow \lim_{\epsilon \to 0} \sup_{\substack{x, y \in \bS_1(0),\\ d(x, y) < \epsilon}} \frac{|f(x) - f(y)|}{d(x, y)^\alpha} = 0,
\]
see e.g. \cite[Section 1.1.2, Exercise 5]{LunardiInterpolation}. A similar definition holds for larger values of $\alpha$. We define the set $\cE_0^\alpha$ to be the set of eigenfunctions $V \in \cE^\alpha$ such that $\rho V$ restricts to a function $v \in c^\alpha(\bS_1(0), \bR)$ on the boundary of $\bH^n$.

\begin{theorem}\label{thmCovariance2}
Assume given two charts at infinity $\Phi_i$, $i=1,2$, for $(M, g)$ such that 
\[
e_i \definedas \Phi_{i*} g - b \in W^{1, p}_\tau(\bH^n \setminus K'_i, S_2 \bH^n) \cap L^1_\delta(\bH^n \setminus K_i', S_2 \bH^n),
\] 
where $p > n$ and $\tau > 3/2$ satisfy
\[
\tau + \frac{n-1}{p} \in
\begin{cases}
    [2, n)                      & \quad \text{if } n = 3, \\
    \left(\frac{n}{2}, n\right) & \quad\text{if } n > 3,
\end{cases}
\]
and $\delta \in (-1, 1]$. Set $\Phi_2 \circ \Phi_1^{-1} = A \circ \Phi_0$ where $A \in O_\uparrow(n, 1)$ is an isometry of $\bH^n$ and $\Phi_0$ is a diffeomorphism asymptotic to the identity. Let $V \in \cE_0^{1-\delta}$ be an eigenfunction of the Laplacian, see Definition \ref{defEigenspace}. Then we have
\[
    p(e_2, V) = p(e_1, V \circ A).
\]
\end{theorem}

\begin{proof}
As mentioned before, the only missing step is to prove that, in the case of a more general eigenfunction $V$, the equality \eqref{eqEquality} pertains, i.e. that the $\cT$-term in \eqref{eqDiffCharge} gives no contribution in the limit $k \to \infty$.

We shall not prove \eqref{eqDiffCharge} for all functions $V \in \cE^{1-\delta}$ but instead only for those $V$ such that $v=\rho V$ is a $C^{1-\delta'}$ function on $\mathbb{S}_1(0)$ for some $\delta' \in [-1, 1)$, $\delta' < \delta$. Indeed, as $C^{1-\delta'}(\bS_1(0), \bR)$ is dense in $c^{1-\delta}(\bS_1(0), \bR)$ and the maps $P(e_1, \cdot)$ and $P(e_2, \cdot)$ are continuous linear forms over $C^{1-\delta}(\bS_1(0), \bR)$ (see Proposition \ref{propContinuity}), if $P(e_1, \cdot)$ and $P(e_2, \cdot)$ agree on $C^{1-\delta'}(\bS_1(0), \bR)$, they agree on the whole of $C^{1-\delta}(\bS_1(0), \bR)$.

We claim that for a function $V$ as described, we have
\begin{equation}\label{eqTheLastOne}
\lim_{k \to \infty} \int_{\bH^n} \cT(\zeta, -D\chibar_k) d\mu^b = 0.
\end{equation}
Indeed, it follows from Lemma \ref{lmDecay2} that $\zeta \in L^1_{\delta'}(\bH^n \setminus K'', T\bH^n)$. As $\cT \in L^\infty_{-\delta'}(\bH^n, S_2\bH^n)$, we have $|\cT(\zeta, \cdot)| \in L^1(\bH^n \setminus K'', \bR)$. Since the sequence of $1$-forms $-D\chibar_k$ is bounded in $L^\infty(\bH^n, T^*\bH^n)$ and converges a.e. to zero, we conclude, by the dominated convergence theorem, that \eqref{eqTheLastOne} holds.
\end{proof}

\providecommand{\bysame}{\leavevmode\hbox to3em{\hrulefill}\thinspace}
\providecommand{\MR}{\relax\ifhmode\unskip\space\fi MR }
\providecommand{\MRhref}[2]{%
  \href{http://www.ams.org/mathscinet-getitem?mr=#1}{#2}
}
\providecommand{\href}[2]{#2}

\end{document}